\documentclass[12pt,a4paper,leqno]{article}

\usepackage[utf8]{inputenc}
\usepackage[T1]{fontenc}
\usepackage[english]{babel}
\usepackage{amsthm}
\usepackage{amscd}
\usepackage{amsfonts}         
\usepackage{amsmath}
\usepackage{amssymb}
\usepackage{hyperref}
\usepackage{enumerate}
\usepackage{tikz}
\usepackage{tikz-cd}
\usepackage{mathrsfs}  
\usepackage{float}

\newcommand{\fref}[1]{\hyperref[{#1}]{\ref*{#1}}}

\newcommand{\Ab}{\mathbb{A}}
\newcommand{\Bb}{\mathbb{B}}

\newcommand{\Eb}{\mathbb{E}}

\newcommand{\Gb}{\mathbb{G}}

\newcommand{\Lb}{\mathbb{L}}

\newcommand{\Pb}{\mathbb{P}}

\newcommand{\Tb}{\mathbb{T}}

\newcommand{\Zb}{\mathbb{Z}}

\newcommand{\Cc}{\mathcal{C}}

\newcommand{\Ec}{\mathcal{E}}
\newcommand{\Fc}{\mathcal{F}}

\newcommand{\Ic}{\mathcal{I}}

\newcommand{\Lc}{\mathcal{L}}
\newcommand{\Mc}{\mathcal{M}}
\newcommand{\Nc}{\mathcal{N}}
\newcommand{\Oc}{\mathcal{O}}
\newcommand{\Pc}{\mathcal{P}}

\newcommand{\Rc}{\mathcal{R}}
\newcommand{\Sc}{\mathcal{S}}

\newcommand{\Cs}{\mathscr{C}}

\newcommand{\Fs}{\mathscr{F}}

\newcommand{\Is}{\mathscr{I}}

\newcommand{\Ls}{\mathscr{L}}

\newcommand{\Ss}{\mathscr{S}}

\newcommand{\PCob}{{\underline\Omega}}

\newcommand{\op}{\mathrm{op}}

\newcommand{\coim}{\mathrm{\coim}}

\newcommand{\Spec}{\mathrm{Spec}}

\newcommand{\bl}{\mathrm{Bl}}

\newcommand{\wtil}{\widetilde}

\newcommand{\Sym}{\mathrm{Sym}}
\newcommand{\LSym}{\mathrm{LSym}}

\newcommand{\Fib}{\mathrm{Fib}}
\newcommand{\Cofib}{\mathrm{Cofib}}
\newcommand{\Fun}{\mathrm{Fun}}

\newcommand{\LM}{\mathrm{LM}}

\newcommand{\pre}{\mathrm{pre}}
\newcommand{\naive}{\mathrm{naive}}
\newcommand{\univ}{\mathrm{univ}}
\newcommand{\snc}{\mathrm{snc}}
\newcommand{\cl}{\mathrm{cl}}
\newcommand{\der}{\mathrm{der}}
\newcommand{\vir}{\mathrm{vir}}
\newcommand{\Proj}{\mathrm{Proj}}
\newcommand{\scr}{\mathrm{scr}}
\newcommand{\modmod}{/\!\!/}
\newcommand{\Id}{\mathrm{Id}}
\newcommand{\topo}{\mathrm{top}}

\newcommand{\dash}{{\text -}}

\restylefloat{table}

\newtheorem{theo}{Tplottin ubuntuheorem}[section]
\theoremstyle{plain}
\newtheorem{thm}[theo]{Theorem}
\newtheorem{lem}[theo]{Lemma}
\newtheorem{prop}[theo]{Proposition}
\newtheorem{cor}[theo]{Corollary}
\newtheorem*{thm*}{Theorem}
\newtheorem*{lem*}{Lemma}
\newtheorem*{prop*}{Proposition}
\newtheorem*{cor*}{Corollary}

\theoremstyle{definition}
\newtheorem{defn}[theo]{Definition}

\newtheorem{ex}[theo]{Example}
\newtheorem{cons}[theo]{Construction}

\newtheorem{rem}[theo]{Remark}
\newtheorem{war}[theo]{Warning}

\title{Precobordism and cobordism}
\author{Toni Annala}

\newcommand{\Addresses}{{
  \bigskip
  \footnotesize

  Toni Annala, \textsc{Department of Mathematics, University of British Columbia,
    Vancouver, BC V6T1Z2 Canada}\par\nopagebreak
  \textit{E-mail address:} \texttt{tannala@math.ubc.ca}

}}

\date{}

\begin{document}

\maketitle

\begin{abstract}
The purpose of this article is to compare several versions of bivariant algebraic cobordism constructed previously by the author and others. In particular, we show that a simple construction based on the universal precobordism theory of Annala--Yokura agrees with the more complicated theory of bivariant derived algebraic cobordism constructed earlier by the author, and that both of these theories admit a Grothendieck transformation to operational cobordism constructed by Luis González--Karu over fields of characteristic 0. The proofs are partly based on convenient universal characterizations of several cobordism theories, which should be of independent interest. Using similar techniques, we also strengthen a result of Vezzosi on operational derived $K$-theory. In the appendix, we give a detailed construction of virtual pullbacks in algebraic bordism, filling the gaps in the construction of Lowrey--Schürg.
\end{abstract}

\tableofcontents

\section{Introduction}

Algebraic cobordism was originally introduced by Voevodsky for his original approach to proving Milnor conjecture. A geometric study of a closely related theory of algebraic bordism was later initiated by Levine and Morel (see their foundational treatment \cite{LM}), and it has been successfully applied to problems arising as diverse areas as Donaldson--Thomas theory and arithmetic questions about algebraic groups (see the work of Levine--Pandharipande \cite{LP} and Sechin--Semenov \cite{SS}). Given a separated finite type scheme $X$ over a field $k$ of characteristic 0, the Levine--Morel bordism group $\Omega^\LM_*(X)$ is defined as the quotient of the free $\Lb$-module on equivalence classes of symbols
$$[V \to X; \Ls_1,..., \Ls_r],$$
where $V$ is smooth and quasi-projective, the morphism $V \to X$ is proper, and $\Ls_i$ are line bundles on $V$. We note that the role of the $\Ls_i$ is to keep track of pseudo-divisor data: if $\Ls_1$ admits a transverse section with smooth vanishing locus $D$, then the equality
$$[V \to X; \Ls_1,..., \Ls_r] = [D \to X; \Ls_2,...\Ls_n]$$
holds in $\Omega^\LM_*(X)$.

Subsequent research has introduced many variants of the original construction. We recall first of them in the following table.
\begin{table}[H]\label{ClassicalCobordismTheoriesTable}
    \begin{tabular}{ | l | l | }
    \hline
    Symbol & Theory  \\ \hline
    $\Omega^\LM_*$ & Levine--Morel algebraic bordism \cite{LM} \\ \hline
    $\omega_*$ & Levine--Pandharipande algebraic bordism \cite{LP} \\ \hline
    $\op \Omega^*_\LM$ & Operational algebraic cobordism \cite{KLG} \\ \hline
    $d\Omega^\LM_*$ & Extended Levine--Morel algebraic bordism \\ \hline
    $d\op \Omega^*_\LM$ & Extended operational algebraic cobordism  \\ \hline    
    \end{tabular}
    \caption{Classical bordism theories defined in characteristic 0, and their extended counterparts.}
\end{table}
\noindent The theory of Levine and Pandharipande was proven to be equivalent to $\Omega^\LM_*$ but the construction was much simpler: $\omega_*(X)$ is a quotient of the free Abelian group on equivalence classes $[V \to X]$ with $V$ smooth and quasi-projective over $k$, $V \to X$ proper, and the only relations are the so called double point cobordism relations. The operational cobordism of Luis González and Karu is obtained by extending $\Omega^\LM_*$ to a bivariant theory in a very formal way. Intuitively one should think of $\op \Omega^*_\LM$ as the coarsest bivariant extension of $\Omega^\LM_*$, and as such it fails to satisfy many properties that would be expected from the correct bivariant cobordism: for instance, one cannot recover the Grothendieck ring of vector bundles as a quotient of $\op \Omega^*_\LM(X)$, so the analogue of Conner--Floyd theorem fails. The extended theories are defined for quasi-projective derived $k$-schemes using truncations: for example, $d \Omega^\LM_*(X) := \Omega^\LM_*(\tau_0(X))$.

Lowrey and Schürg took a big step forward in \cite{LS} by constructing \emph{derived algebraic bordism groups} $d \Omega^k_*$ for quasi-projective derived schemes over a field $k$ (not necessarily of characteristic 0). The great advantage of their construction is that the the existence of quasi-smooth pullbacks (analogous to lci pullbacks, whose construction is the most subtle and laborious part of Levine--Morel) is completely trivial. However, this comes with a price: $d \Omega^k_*(X)$ is by definition a quotient of the free $\Lb$-module on equivalence classes $[V \to X]$ of projective maps from $V$ a quasi-smooth derived scheme over $k$. This makes the direct study of the derived bordism groups difficult, but if $k$ has characteristic 0, then the authors were able to show that the theories $d \Omega^k_*$ and $d \Omega^\LM_*$ are isomorphic, proving that $d \Omega^k_*$ gives the correct theory at least in characteristic 0. We note here that the derived structure on $V$ plays a similar role to the role played line bundles in the original construction of Levine and Morel. Besides the work introduced in the next paragraph, the results of Lowrey--Schürg have recently gathered other attention as well, see e.g. \cite{De, EHKSY1, EHKSY2}.

In the wake of the initial work by Lowrey and Schürg, the author has introduced several new bivariant theories whose construction makes use of derived algebraic geometry. 
\begin{table}[H]\label{DerivedCobordismTheoriesTable}
    \begin{tabular}{ | l | l | }
    \hline
    Symbol & Theory  \\ \hline
    $d \Omega^*_k$ & Bivariant derived algebraic cobordism over $k$ \cite{An1} \\ \hline
    $\PCob^*_A$ & Universal precobordism over $A$ \cite{AY} \\ \hline
    $\Omega^*_A$ & Bivariant algebraic cobordism over $A$ \\ \hline
    $\op d \Omega^*_k$ & Operational derived algebraic cobordism over $k$ \\ \hline
    \end{tabular}
    \caption{Derived cobordism theories defined in various degrees of generality.}
\end{table}
\noindent First of these, bivariant derived algebraic cobordism $d \Omega^*_k$ over a field $k$, is defined as a very natural bivariant extension of Lowrey--Schürg's derived bordism groups over $k$. Intuitively this is the finest bivariant extension of $d \Omega^k_*$, and many of the properties expected from the correct cobordism theory (such as the analogue of Conner--Floyd theorem) were verified in \cite{An1} when the base field $k$ has characteristic 0. The subsequent work by Annala--Yokura saw the introduction of universal precobordism theory $\PCob^*_A$, which is the derived and bivariant analogue of $\omega_*$, and which is defined for quasi-projective derived schemes over a Noetherian base ring $A$ of finite Krull dimension. Quite unexpectedly, it was possible to prove Conner--Floyd theorem and many other expected properties in \cite{An2,AY} without any further restrictions on the base ring $A$. Even though $\PCob^*_A$ has many good properties, we have good reason to expect that it is missing certain important relations, and therefore fails to give rise to the correct associated homology theory. Therefore we introduce bivariant algebraic cobordism $\Omega^*_A$ which is a quotient of $\PCob^*_A$ by the bivariant ideal generated by the so called snc relations. Finally, we will introduce the operational derived cobordism $\op d \Omega^*_k$ as a formal extension of Lowrey--Schürg bordism groups $d \Omega^k_*$. All the derived theories introduced here only over fields will be constructed over arbitrary Noetherian base rings of finite Krull dimension in the main text.

\subsection{Summary of results}

Let us then summarize the contributions of this paper. It is possible to define Lowrey--Schürg's derived bordism groups $d \Omega^A_*$ over a Noetherian ring $A$ of finite Krull dimension: we do this in Section \ref{DerivedBivariantCobConsSubSect}. Note that there are some ambiguities and gaps in the construction of \cite{LS}, so the reader is advised not to skip this section even if they are only interested in the case when $A$ is a field. Simultaneously, we also obtain definition of $d \Omega^*_A$ (and later $\op d \Omega^*_A$) over such an $A$.

Our first main result completely determines the relationship of $\Omega^*_A$ and $d \Omega^*_A$.

\begin{thm}\label{FirstComparisonThm}
Let $A$ be a Noetherian ring of finite Krull dimension. Then there exists a unique orientation preserving Grothendieck transformation
$$\eta: \Omega^*_A \to d\Omega^*_A.$$
Moreover, $\eta$ is an isomorphism of bivariant theories.
\end{thm}

\noindent In order to prove this result, we characterize $\PCob^*_A$ by a pleasant universal property in Theorem \ref{UnivPropOfBivariantPrecob}: $\PCob^*_A$ is the universal stably oriented additive bivariant theory satisfying the section and the formal group law axioms. As a consequence,  we show in Corollary \ref{UnivPropOfBivariantAlgCob} that $\Omega^*_A$ is the universal stably oriented additive bivariant theory satisfying the section, the formal group law and the snc axioms.

In Construction \ref{OpDCobCons} we define operational derived cobordism theory $\op d \Omega^*_A$ as a formal bivariant extension of $d \Omega^A_*$. Our second main result shows that when $A$ is a field of characteristic 0, the operational derived cobordism agrees with the classical operational cobordism introduced earlier.

\begin{thm}\label{SecondComparisonThm}
If $k$ is a field of characteristic 0, then there exists a unique Grothendieck transformation
$$\iota: d \op \Omega^*_\LM \to \op d \Omega^*_k$$
extending the isomorphism of Lowrey--Schürg $\iota_h: d \Omega^\LM_* \to d \Omega^k_*$ of the associated homology theories. Moreover, $\iota$ is an isomorphism of bivariant theories.
\end{thm}

\noindent The main steps in the proof are establishing the commutativity of the operational cobordism theory $\op \Omega^*_\LM$ (Theorem \ref{OpCobIsCommThm}), and providing $d \Omega^*_A$ with a universal property guaranteeing that it receives Grothendieck transformations from all nice enough bivariant extensions of $d\Omega^A_*$ (Theorem \ref{OpDCobUnivProp}). Another consequence of these results is Theorem \ref{BivCobToOpCobThm}, which states that there exists a unique orientation preserving Grothendieck transformation
$$\nu: \Omega^*_A \to \op d \Omega^*_A.$$
We warn the reader that operational cobordism is not expected to be isomorphic to any of the other bivariant theories considered in the article (see Warning \ref{OpCobIsNotCobWar}).

In order to achieve these results, we need to take care of a technical issue: in the heart of Lowrey's and Schürg's paper lies the construction of quasi-smooth pullbacks (\emph{virtual pullbacks}) for $d \Omega^\LM_*$. There are several gaps in their construction, which we fill in appendices. Let us record the statement (see Theorem \ref{VirtPullbackPropertiesThm} for a more precise statement).
\begin{thm}
There exist functorial pullback morphisms
$$f^\vir: d \Omega^\LM_*(\tau_0(Y)) \to d \Omega^\LM_*(\tau_0(X))$$
along quasi-smooth morphisms $f: X \to Y$ of quasi-projective derived $k$-schemes. Moreover, these pullbacks satisfy the push-pull formula and they agree with the lci pullbacks of Levine and Morel when $f$ is classical.
\end{thm}
\noindent Note that the main result of \cite{LS} (Theorem 5.12), which establishes the equivalence between $\Omega^\LM_*$ and $d\Omega^k_*$ over a field $k$ of characteristic 0, as well as the main results of \cite{An1} depend on this result, offering ample motivation for giving a detailed proof. The proof is based on detailed study of truncations of derived blow ups and the properties of refined lci pullbacks of Levine--Morel.

In Section \ref{OpKThySect} we prove a $K$-theoretic analogue of Theorem \ref{SecondComparisonThm}.
\begin{thm}\label{ThirdComparisonThm}
Let $k$ be a ground field. Then the Grothendieck transformation 
$$\alpha: \op K^\der \to d \op K^0$$
constructed by Vezzosi in \cite{AGP} Proposition B.7 is an isomorphism.
\end{thm} 
\noindent This strengthens the result of Vezzosi, which states that $\alpha$ is injective. Above, $d \op K^0$ is the operational $K$-theory of Anderson--Payne extended to derived schemes in a trivial fashion. Even though this is not directly related to the other main results of the article, it is clearly of similar spirit and the proofs use similar techniques.

\subsection{Structure of the article}

We start with Section \ref{BivThySect}, which is a background section on general properties of bivariant theories. Most of the results are old, but we also introduce some useful notions, which will help with simplifying the proofs in the main text. Especially, the results of Section \ref{BivChernSubSect} on Chern classes in bivariant theories have not appeared in this form elsewhere. After the necessary background is laid down, Sections \ref{FirstComparisonSect}, \ref{SecondComparisonSect} and \ref{OpKThySect} follow, with the main purpose of proving Theorems \ref{FirstComparisonThm}, \ref{SecondComparisonThm} and \ref{ThirdComparisonThm} respectively. There are three appendices in the end. Virtual pullbacks are constructed in Appendix \ref{VirtualPullbackSect}, and the construction uses the detailed study of truncations of derived blow ups done in Appendix \ref{DerivedBlowUpSect}. This in turn requires the naturality of the Hurewicz morphism proved in Appendix \ref{HurewiczSect}.
 
\subsection*{Conventions}

We will freely use the language of $\infty$-categories and derived algebraic geometry whenever necessary. In order to lighten the exposition, if we are given an object $X$ in an $\infty$-category $\Cc$ with a distinguished final object $pt$, we will denote by $\pi_X$ the essentially unique map $X \to pt$. A \emph{Noetherian} derived scheme has Noetherian truncation and all the homotopy sheaves of the structure sheaf are coherent. A quasi-projective derived scheme over a Noetherian ring of finite Krull dimension is implicitly assumed to be Noetherian. A \emph{quasi-smooth} morphism is a finite type morphism $X \to Y$ with the relative cotangent complex $\Lb_{X/Y}$ perfect and of Tor-amplitude 1. A quasi-smooth morphism with $\Lb_{X/Y}$ of Tor-amplitude 0 is \emph{smooth}, i.e., it is flat and has smooth homotopy fibres. Quasi-smooth closed immersions are often also called \emph{derived regular embeddings}. In order to avoid confusion, we will call the $\infty$-categorical fibre products of derived schemes \emph{derived fibre products}, and reserve the term \emph{fibre product} in the context of algebraic geometry to the classical fibre products of classical schemes. Derived fibre product will be denoted by $X \times_Z^R Y$, since it can be thought as the right derived functor of the classical fibre product. We will denote by $\Lb$ the \emph{Lazard ring}. Recall that it has two gradings: the \emph{cohomological} (or non-positive) grading $\Lb^*$ and the \emph{homological} (or non-negative) grading $\Lb_*$, which satisfy $\Lb^{-*} = \Lb_{*}$.

\subsection*{Acknowledgements}

The author would like to thank Adeel Khan for discussions, as well as the University of Regensburg for hospitality during the author's visit there, during which the writing of this article was initiated. He would also like to thank Fabien Morel, Stefan Schreieder, Olivier Haution and Marc Levine for asking questions about how the various derived cobordism groups relate to the classical bordism groups of Levine--Morel. He would also like to thank Gabriele Vezzosi for discussions about operational $K$-theory. The anonymous referees provided valuable feedback that helped to substantially revise the article. The author was supported by the Vilho, Yrj\"o and Kalle V\"ais\"al\"a Foundation of the Finnish Academy of Science and Letters. 

\section{Background on bivariant theories}\label{BivThySect}

The purpose of this section is to recall the necessary background on the bivariant formalism of Fulton and MacPherson, the construction of the universal bivariant theory of Yokura from \cite{Yo1}, and to study properties of Chern classes in certain types of bivariant theories. We will work in the context of $\infty$-categories, so our setup is more general than in the classical literature (e.g. \cite{FM, Yo1}). However, since all formal aspects of bivariant formalism generalize without any real effort, we sometimes refer directly to classical literature even when we are working with $\infty$-categories. We are going to use $\infty$-categorical language throughout this section. For example, when we talk about fibre products or Cartesian diagrams, we always mean the $\infty$-categorical notions, which closely related to the more classical notion of homotopy fibre product and homotopy Cartesian squares in model categories. Note that in Section \ref{BivChernSubSect} we will call fibre products of derived schemes derived fibre products, following one of the conventions of this article.

\subsection{Bivariant theories}\label{BivThySubSect}

In this section we recall the definition and basic properties of bivariant theories. Most of the content is old, but we also introduce some new terminology, such as rings of coefficients for bivariant theories.

\begin{defn}\label{FunctorialityDef}
A \emph{functoriality} is a tuple $\Fs = (\Cc, \Cs, \Is, \Ss)$, where 
\begin{enumerate}
\item $\Cc$ is an $\infty$-category with a (distinguished) final object $pt$ and all fibre products;

\item $\Cs$ is a class of morphisms in $\Cc$ called \emph{confined morphisms} which contains all isomorphisms and is closed under compositions, pullbacks and equivalences;

\item $\Is$ is a class of Cartesian squares in $\Cc$ called \emph{independent squares} which contains all squares of form
\begin{center}
\begin{tikzcd}
X \arrow[]{r}{f} \arrow[]{d}{\mathrm{Id}_X} & Y \arrow[]{d}{\mathrm{Id}_Y} \\
X \arrow[]{r}{f} & Y
\end{tikzcd}
\ \ and \ \ \
\begin{tikzcd}
X \arrow[]{r}{\mathrm{Id}_X} \arrow[]{d}{f} & X \arrow[]{d}{f} \\
Y \arrow[]{r}{\mathrm{Id}_Y} & Y,
\end{tikzcd}
\end{center}
and is closed under vertical and horizontal compositions in the obvious sense as well as equivalences of Cartesian squares

\item $\Ss$ is a class of morphisms in $\Cc$ called \emph{specialized morphisms} which contains all isomorphisms and is closed under compositions and equivalences.
\end{enumerate}
\end{defn}

\begin{defn}
A \emph{bivariant theory $\Bb^*$ (with functoriality $\Fs$)} assigns a graded Abelian group $\Bb^*(X \to Y)$ to all homotopy classes of morphisms in $\Cc$. Moreover
\begin{enumerate}
\item if $X \to Y$ factors as $X \xrightarrow{f} X' \to Y$ with $f$ confined, there is a \emph{bivariant pushforward} 
$$f_*: \Bb^*(X \to Y) \to \Bb^*(X' \to Y);$$
\item if the Cartesian square
$$
\begin{tikzcd}
X' \arrow[]{d} \arrow[]{r} & Y' \arrow[]{d}{g} \\
X \arrow[]{r} & Y
\end{tikzcd}
$$
is independent, then there is a \emph{bivariant (homotopy) pullback} 
$$g^*: \Bb^*(X \to Y) \to \Bb^*(X' \to Y');$$
\item given a composition $X \to Y \to Z$, there is a bilinear \emph{bivariant product}
$$\bullet: \Bb^*(X \to Y) \times \Bb^*(Y \to Z) \to \Bb^*(X \to Z).$$
\end{enumerate}
Moreover, this structure is required to satisfy the following axioms, which we recall since we will refer to them later:
\begin{enumerate}
\item[($A_1$)] \emph{associativity of $\bullet$}: given morphisms $X \to Y \to Z \to W$ and bivariant elements $\alpha \in \Bb^*(X \to Y)$, $\beta \in \Bb^*(Y \to Z)$ and $\gamma \in \Bb^*(Z \to W)$, then
$$(\alpha \bullet \beta) \bullet \gamma = \alpha \bullet (\beta \bullet \gamma) \in \Bb^*(X \to W);$$

\item[($A_2$)] \emph{covariant functoriality of bivariant pushforward}: if we have $\alpha \in \Bb^*(X \to Y)$, and $X \to Y$ factors through the composition  
$$X \xrightarrow{f} X' \xrightarrow{g} X''$$
with $f$ and $g$ confined, then
$$g_*(f_*(\alpha)) = (g \circ f)_*(\alpha) \in \Bb^*(X'' \to Y);$$

\item[($A_3$)] \emph{contravariant functoriality of bivariant pullback}: if we have $\alpha \in \Bb^*(X \to Y)$ and the Cartesian squares
$$
\begin{tikzcd}
X' \arrow[]{d} \arrow[]{r} & Y' \arrow[]{d}{g} \\
X \arrow[]{r} & Y
\end{tikzcd}
\text{ \ and \ }
\begin{tikzcd}
X'' \arrow[]{d} \arrow[]{r} & Y'' \arrow[]{d}{f} \\
X' \arrow[]{r} & Y'
\end{tikzcd}
$$
are independent, then
$$f^*(g^*(\alpha)) = (g \circ f)^*(\alpha) \in \Bb^*(X'' \to Y'');$$

\item[($A_{12}$)] \emph{product and pushforward commute}: given $\alpha \in \Bb^*(X \to Y)$ and $\beta \in \Bb^*(Y \to Z)$, and a confined morphism $f: X \to X'$ through which $X \to Y$ factors, then
$$f_*(\alpha \bullet \beta) = f_*(\alpha) \bullet \beta \in \Bb^*(X \to Y);$$

\item[($A_{13}$)] \emph{bivariant pullback is multiplicative}: given $\alpha \in \Bb^*(X \to Y)$ and $\beta \in \Bb^*(Y \to Z)$ and suppose the small squares in the Cartesian diagram
$$
\begin{tikzcd}
X' \arrow[]{d}{h''} \arrow[]{r}{} & Y' \arrow[]{d}{h'} \arrow[]{r} & Z' \arrow[]{d}{h} \\
X \arrow[]{r}{} & Y \arrow[]{r} & Z
\end{tikzcd}
$$
are independent, then 
$$h^*(\alpha \bullet \beta) = h'^*(\alpha) \bullet h^*(\beta) \in \Bb^*(X' \to Z');$$

\item[($A_{23}$)] \emph{bivariant push-pull formula}: given $\alpha \in \Bb^*(X \to Z)$ and a Cartesian diagram
$$
\begin{tikzcd}
X' \arrow[]{d}{h''} \arrow[]{r}{f'} & Y' \arrow[]{d}{h'} \arrow[]{r} & Z' \arrow[]{d}{h} \\
X \arrow[]{r}{f} & Y \arrow[]{r} & Z
\end{tikzcd}
$$
with $f$ confined and the rightmost small square as well as the large square independent, then
$$h^*(f_*\alpha) = f'_*(h^*(\alpha)) \in \Bb^*(Y' \to Z');$$

\item[($A_{123}$)] \emph{bivariant projection formula}: given an independent square
$$
\begin{tikzcd}
X \arrow[]{d}{g'} \arrow[]{r} & Y \arrow[]{d}{g} \\
X' \arrow[]{r} & Y'
\end{tikzcd}
$$
with $g$ confined, a morphism $Y' \to Z$ and elements $\alpha \in \Bb^*(Y \to Z)$ and $\beta \in \Bb^*(X' \to Y')$, then
$$g'_*(g^*(\beta) \bullet \alpha) = \beta \bullet g_*(\alpha) \in \Bb^*(X' \to Z);$$

\item[($U$)] \emph{existence of units}: for each $X \in \Cc$, there exists $1_X \in \Bb^*(X \to X)$ so that for all $\alpha \in \Bb^*(X \to Y)$ we have
$$1_X \bullet \alpha = \alpha \in \Bb^*(X \to Y)$$
and for all $\beta \in \Bb^*(Z \to X)$, we have
$$\beta \bullet 1_X = \beta \in \Bb^*(Z \to X).$$

\end{enumerate}
\end{defn}

\begin{defn}
If $\Bb^*$ is a bivariant theory with functoriality $\Fs$, then an \emph{orientation} $\theta$ on $\Bb^*$ is a choice of elements 
$$\theta(f) \in \Bb^*(X \xrightarrow{f} Y)$$
for all equivalence classes of specialized morphisms $f$, so that $\theta(\mathrm{Id}) = 1 $ and $\theta(f \circ g) = \theta(g) \bullet \theta(f)$. If the specialized morphisms are stable under independent pullbacks, we say that $\theta$ is \emph{stable under pullbacks} if bivariant pullbacks preserve orientations. An \emph{oriented bivariant theory} is a pair $(\Bb^*, \theta)$, where $\Bb^*$ is a bivariant theory and $\theta$ is an orientation on $\Bb^*$. We will usually omit $\theta$ from the notation. An oriented bivariant theory $\Bb^*$ whose orientation is stable under pullbacks is called \emph{stably oriented}.
\end{defn}

\begin{war}
Our definition of an oriented bivariant theory does not agree with the one used by Yokura. However, since Yokura's oriented theories are bivariant theories together with Chern class like operators for objects in a category $\Lc$ fibered over $\Cc$, we prefer to call them \emph{$\Lc$-oriented bivariant theories}.
\end{war}

\begin{defn}
Given a bivariant theory $\Bb^*$, one can define
\begin{enumerate}
\item the \emph{associated homology groups} $\Bb_*(X) := \Bb^{-*}(X \to pt)$ which are covariantly functorial for confined morphisms;
\item the \emph{associated cohomology rings} $\Bb^*(X) := \Bb^*(X \xrightarrow{\mathrm{Id}} X)$ which are contravariantly functorial for all morphisms, and whose ring structure is given by the bivariant product.
\end{enumerate}
Given an oriented bivariant theory $\Bb^*$, we can define
\begin{enumerate}
\item \emph{Gysin pullback morphisms} $f^!: \Bb_*(Y) \to \Bb_*(X)$ for $f: X \to Y$ specialized by 
$$\alpha \mapsto \theta(f) \bullet \alpha;$$
\item \emph{Gysin pushforward morphisms} $f_!: \Bb^*(X) \to \Bb^*(Y)$ for $f: X \to Y$ specialized and confined by
$$\alpha \mapsto f_*\bigl(\alpha \bullet \theta(f)\bigr).$$
\end{enumerate} 
It follows from the multiplicative properties of orientations that the Gysin pushforwards and pullbacks are functorial in the obvious sense.
\end{defn}

\begin{defn}
A \emph{Grothendieck transformation} $\eta: \Bb^*_1 \to \Bb^*_2$ consists of group homomorphisms $\eta_{X \to Y}: \Bb^*_1(X \to Y) \to \Bb^*_2(X \to Y)$ that commute in the obvious sense with bivariant pushforwards, pullbacks and products. If the theories $\Bb^*_1$ and $\Bb^*_2$ are oriented with orientations $\theta^{\Bb^*_1}$ and $\theta^{\Bb^*_2}$ respectively, then $\eta$ is said to be \emph{orientation preserving} if 
$$\eta \bigl( \theta^{\Bb^*_1}(f) \bigr) = \theta^{\Bb^*_2}(f)$$
for all equivalence classes of specialized morphisms $f$. 
\end{defn}

\subsubsection*{Rings of coefficients}

\begin{defn}\label{BivariantCoeffRingDef}
A \emph{ring of coefficients} for a bivariant theory $\Bb^*$ is a ring $R$ together with a map of rings $R \to \Bb^*(pt)$. We also call such a bivariant theory $\Bb^*$ \emph{$R$-linear}.
\end{defn}

The terminology is explained by the following result.

\begin{prop}\label{RLinearTheoryProp}
Let $\Bb^*$ be an $R$-linear bivariant theory. Then 
\begin{enumerate}
\item $\Bb^*(X \to Y)$ is an $R \dash R$-bimodule with the actions defined as
$$r \alpha := \pi_X^*(r) \bullet \alpha$$
and  
$$\alpha r := \alpha \bullet \pi^*_{Y}(r);$$

\item bivariant pullbacks and pushforwards are maps of $R \dash R$-bimodules;

\item bivariant products are $R$-balanced and they give rise to maps of $R \dash R$-bimodules
$$\Bb^*(X \to Y) \otimes_R \Bb^*(Y \to Z) \to \Bb^*(X \to Z).$$
\end{enumerate}
Moreover, if $\Bb^*$ is oriented, then
\begin{enumerate}
\item[4.] the Gysin-morphisms are maps of $R \dash R$-bimodules.
\end{enumerate}
\end{prop}
\begin{proof}
The fourth claim is an immediate consequence of the first three, so we will only prove them.
\begin{enumerate}
\item Since the bivariant product $\bullet$ is $\Zb$-bilinear, it follows that $\Bb^*(X \to Y)$ is both a left and a right $R$-module. The equality
$$r(\alpha s) = (r \alpha) s$$
follows from the associativity of $\bullet$.

\item Suppose that $X \to Y$ factors through a confined morphism $f: X \to X'$, and let $\alpha \in \Bb^*(X \to Y)$. Then
\begin{align*}
f_*(s \alpha) &= f_* (\pi_X^*(s) \bullet \alpha) \\
&= f_*(f^*(\pi_{X'}^*(s)) \bullet \alpha) \\
&= \pi_{X'}^*(s) \bullet f_*(\alpha) & (\text{bivariant axiom $A_{123}$}) \\
&= s f_*(\alpha)
\end{align*} 
and 
\begin{align*}
f_*(\alpha s) &= f_*(\alpha \bullet \pi_Y^*(s)) \\
&= f_*(\alpha) \bullet \pi^*_Y(s) & (\text{bivariant axiom $A_{12}$}) \\
&= f_*(\alpha)s,
\end{align*}
so bivariant pushforwards are maps of $R \dash R$-bimodules.

Consider then a morphism $g: Y' \to Y$ and form the Cartesian square
$$
\begin{tikzcd}
X' \arrow[]{r} \arrow[]{d}{g'} & Y' \arrow[]{d}{g} \\
X \arrow[]{r} & Y.
\end{tikzcd}
$$ 
Then, given $\alpha \in \Bb^*(X \to Y)$, we compute that 
\begin{align*}
g^*(r \alpha) &= g^*(\pi_X^*(r) \bullet \alpha) \\
&= g'^*(\pi_X^*(r)) \bullet g^*(\alpha) & (\text{bivariant axiom $A_{13}$}) \\
&= \pi^*_{X'}(r) \bullet g^*(\alpha) \\
&= r g^*(\alpha)
\end{align*}
and similarly one shows that $g^*(\alpha r) = g^*(\alpha) r$, so bivariant pullbacks are maps of $R \dash R$-bimodules as well.

\item Let us have $\alpha \in \Bb^*(X \to Y)$ and $\beta \in \Bb^*(Y \to Z)$. Then
\begin{align*}
(\alpha r) \bullet \beta &= (\alpha \bullet \pi^*_Y(r)) \bullet \beta \\
&= \alpha \bullet (\pi^*_Y(r) \bullet \beta) & (\text{associativity of $\bullet$}) \\
&= \alpha \bullet (r \beta),
\end{align*}
i.e., $\bullet$ is $R$-balanced. Moreover, the induced morphism
$$\Bb^*(X \to Y) \otimes_R \Bb^*(Y \to Z) \to \Bb^*(X \to Z)$$
is $R \dash R$-linear since
\begin{align*}
(r \alpha) \otimes \beta &\mapsto (r \alpha) \bullet \beta \\
&= (\pi^*_X(r) \bullet \alpha) \bullet \beta \\
&= \pi^*_X(r) \bullet (\alpha \bullet \beta) & (\text{associativity of $\bullet$}) \\
&= r(\alpha \bullet \beta)
\end{align*}
and similarly one shows that $\alpha \bullet (\beta r) \mapsto (\alpha \bullet \beta) r$. \qedhere
\end{enumerate}
\end{proof}

Another useful fact about $R$-linear bivariant theories is that their Grothendieck transformations are often $R$-linear.

\begin{defn}\label{RLinearGTransDef}
Let $\eta: \Bb_1^* \to \Bb_2^*$ be a Grothendieck transformation of $R$-linear bivariant theories. Then we say that $\eta$ is \emph{$R$-linear} if the induced morphisms
$$\eta: \Bb_1^*(X \to Y) \to \Bb_2^*(X \to Y)$$
are maps of $R \dash R$-bimodules.
\end{defn}

\begin{prop}\label{RLinearGTransProp}
Let $\eta: \Bb_1^* \to \Bb_2^*$ be a Grothendieck transformation of $R$-linear bivariant theories. Then $\eta$ is $R$-linear if and only if the induced morphism 
$$\eta: \Bb_1^*(pt) \to \Bb_2^*(pt)$$
is a map of $R$-algebras.
\end{prop}
\begin{proof}
Clearly an $R$-linear $\eta$ induces a map of $R$-algebras. To prove the converse, we consider $\alpha \in \Bb_1^*(X \to Y)$ and compute that
\begin{align*}
\eta (r \alpha) &= \eta(\pi^*_X(r) \bullet \alpha) \\
&= \eta(\pi^*_X(r)) \bullet \eta(\alpha) \\
&= \pi^*_X(\eta(r)) \bullet \eta(\alpha) \\
&= \pi^*_X(r) \bullet \eta(\alpha)  & (\text{$\eta$ induces a map of $R$-algebras}) \\
&= r \eta(\alpha).
\end{align*} 
One proves similarly that $\eta(\alpha r) = \eta(\alpha)r$, so we have proven the claim.
\end{proof}

\subsubsection*{Bivariant ideals}

The following terminology will be convenient when constructing interesting bivariant theories by imposing relations to ``free bivariant theories''.

\begin{defn}
Let $\Bb^*$ be a bivariant theory. A \emph{bivariant subset} $\Sc \subset \Bb^*$ is an assignment of a subset
$$\Sc(X \to Y) \subset \Bb^*(X \to Y)$$
for each equivalence class of morphisms in $\Cc$. A bivariant subset $\Ic$ is a \emph{bivariant ideal} if $\Ic(X \to Y)$ are subgroups stable under bivariant pullbacks and pushforwards, and if they satisfy the ideal conditions
$$\Bb^*(X \to Y) \times \Ic(Y \to Z) \xrightarrow{\bullet} \Ic(X \to Z)$$
and
$$\Ic(X \to Y) \times \Bb^*(Y \to Z) \xrightarrow{\bullet} \Ic(X \to Z)$$
for all $X \to Y \to Z$. Given a bivariant subset $\Sc \subset \Bb^*$, we will denote by $\langle \Sc \rangle_{\Bb^*}$ the smallest bivariant ideal of $\Bb^*$ containing $\Sc$, or in other words the bivariant ideal \emph{generated by} $\Sc$. If the bivariant theory $\Bb^*$ is clear from the context, we will often drop the subscript from the notation, denoting the generated ideal simply by $\langle \Sc \rangle$.
\end{defn}

We make the trivial observation that bivariant ideals are automatically compatible with any coefficient ring in the following sense.

\begin{prop}\label{IdealsInRLinearTheoriesProp}
Let $\Bb^*$ be an $R$-linear bivariant theory, and let $\Ic \subset \Bb^*$ be a bivariant ideal. Then 
$$\Ic(X \to Y) \subset \Bb^*(X \to Y)$$
are $R \dash R$-submodules. \qed
\end{prop}

The following result is essentially Lemma 3.8 of \cite{An1}, but see also Proposition 3.7 of \cite{AY}.

\begin{prop}\label{GeneratedBivariantIdealProp}
Let $\Bb^*$ be a bivariant theory having bivariant pullbacks along all Cartesian squares. If $\Sc \subset \Bb^*$ is a bivariant subset, then the elements in $\langle \Sc \rangle_{\Bb^*}$ are integral combinations of elements of form 
$$f_*(\alpha \bullet g^*(s) \bullet \beta)$$
where $\alpha$ and $\beta$ are arbitrary bilinear elements, $g$ is an arbitrary morphism and $f$ is a confined morphism (with the obvious restriction that the above expression should make sense).
\end{prop}

\subsubsection*{Commutativity}

Let us recall what it means for a bivariant theory to be commutative.

\begin{defn}\label{CommutativeBivariantTheoryDef}
A bivariant theory $\Bb^*$ is called \emph{commutative}, if all independent squares are closed under transposition, and if for all independent squares
$$
\begin{tikzcd}
X' \arrow[]{r} \arrow[]{d} & Y' \arrow[]{d}{g} \\
X \arrow[]{r}{f} & Y
\end{tikzcd}
$$
and all $\alpha \in \Bb^*(X \to Y)$, $\beta \in \Bb^*(Y' \to Y)$ we have
$$f^*(\beta) \bullet \alpha = g^*(\alpha) \bullet \beta \in \Bb^*(X' \to Y).$$
Note that the associated cohomology theory of a commutative bivariant theory takes values in commutative rings.
\end{defn}

For commutative bivariant theories $\Bb^*$, Proposition \ref{RLinearTheoryProp} could be simplified using the following observation.

\begin{lem}
Let $\Bb^*$ be a commutative bivariant theory, and suppose it is linear over $R$. Then, for all $\alpha \in \Bb^*(X \xrightarrow{f} Y)$, we have that
$$r \alpha = \alpha r \in \Bb^*(X \to Y).$$
\end{lem}
\begin{proof}
Indeed, 
\begin{align*}
r \alpha &= \pi^*_X(r) \bullet \alpha \\
&= f^*(\pi^*_Y(r)) \bullet \alpha \\
&= \alpha \bullet \pi^*_Y(r) & (\text{commutativity of $\Bb^*$}) \\
&= \alpha r
\end{align*}
which proves the claim.
\end{proof}

\subsubsection*{External products}

Next we need to recall the definition of bivariant external products. Note that these can be defined without any commutativity hypotheses, but then one has to make a choice of convention in the definition (see the definition below to understand). Moreover, almost all bivariant theories we are actually going to consider in this article are commutative, so we would gain nothing from a more general definition. 

\begin{defn}\label{BivariantExternalProductDef}
Assume (for simplicity) that all independent squares in $\Fs$ are independent, and that $\Bb^*$ is a commutative bivariant theory with functoriality $\Fs$. Then we define the \emph{external product}
$$\times : \Bb^*(X \to Y) \otimes \Bb(X' \to Y') \to \Bb^*(X \times X' \to Y \times Y')$$
by 
$$\alpha \times \beta := (g \circ f)^* (\beta) \bullet h^*(\alpha)$$
where $f$, $g$ and $h$ are as in the diagram
$$
\begin{tikzcd}
X \times X' \arrow[]{r} \arrow[]{d} & Y \times X' \arrow[]{r} \arrow[]{d} & X' \arrow[]{d} \\
X \times Y' \arrow[]{r}{f} \arrow[]{d} & Y \times Y' \arrow[]{r}{g} \arrow[]{d}{h} & Y' \arrow[]{d} \\
X \arrow[]{r} & Y \arrow[]{r} & pt.
\end{tikzcd}
$$
\end{defn}

Finally, we record a useful fact concerning bivariant ideals generated by homological elements: the bivariant ideal generated by a nice enough set of homology relations does not contain any extraneous elements in the associated homology groups.

\begin{prop}\label{BivariantHomIdealProp}
Let $\Bb^*$ be a stably oriented commutative bivariant theory having bivariant pullbacks along all Cartesian squares. Let us also assume that $\Bb^*$ is \emph{generated by orientations} in the sense that, for all $X \to Y$, the group $\Bb^*(X \to Y)$ is generated by $\Bb^*(pt)$-linear combinations of elements of form $f_*(\theta(g))$, where $f: Z \to X$ is confined with the composition $g: Z \to Y$  specialized.

If $\Sc \subset \Bb^*$ is a bivariant subset with $\Sc(X \to Y)$ empty unless $Y = pt$, then we have the equality
$$\Sc(X \to pt) = \langle \Sc \rangle_{\Bb^*} (X \to pt)$$
for all $X$ if and only if
\begin{enumerate}
\item $\Sc(X \to pt) \subset \Bb^*(X \to pt)$ is a subgroup for all $X$;
\item $\Sc$ is stable under pushforwards in the associated homology theory;
\item $\Sc$ is stable under Gysin pullbacks in the associated homology theory;
\item given $s \in \Sc(X \to pt)$ and $\beta \in \Bb^*(Y \to pt)$, we have that $s \times \beta \in \Sc(X \times Y \to pt)$.
\end{enumerate}
\end{prop}
\begin{proof}
Suppose $\Sc$ satisfies the four conditions. Since $\Bb^*$ satisfies the assumptions of Proposition \ref{GeneratedBivariantIdealProp}, $\langle \Sc \rangle_{\Bb^*}(X \to pt)$ is generated by elements of form
$$f_*(\alpha \bullet g^*(s) \bullet \beta)$$
with $s \in \Sc(X' \to pt)$, $g: Y \to pt$, $\beta \in \Bb^*(Y \to pt)$, $\alpha \in \Bb^*(X'' \to X' \times Y)$ and $f: X'' \to X$ confined. As 
$$g^*(s) \bullet \beta = s \times \beta,$$
it lies in $\Sc(X' \times Y \to pt)$, and we denote it by $s'$. Moreover, since we may assume without loss of generality that $\alpha = f'_*(\theta(g'))$ where $f': V \to X''$ is a confined morphism so that the composition $g': V \to X' \times Y$ is specialized, it follows that
\begin{align*}
\alpha \bullet s' &= f'_*(\theta(g')) \bullet s' & \\
&= f'_*(\theta(g') \bullet s' ). & (\text{bivariant axiom $A_{12}$})
\end{align*}
Since $\theta(g') \bullet s' $ lands in $\Sc$ by stability under Gysin pullbacks, the desired equality follows from the stability of $\Sc$ under pushforwards. The converse claim follows immediately from the stability properties of bivariant ideals.
\end{proof}

\subsection{Universal bivariant theory}\label{UBTSubSect}

The purpose of this section is to recall the construction and the basic properties of the universal bivariant theory $\Mc_\Fs$ of Yokura from \cite{Yo1}. We note that $\Mc_\Fs$ is a stably oriented bivariant theory.

We work with a functoriality $\Fs = (\Cc, \Cs, \Is, \Ss)$ so that
\begin{enumerate}
\item $\Is$ contains all Cartesian squares;
\item $\Ss$ is stable under pullbacks. 
\end{enumerate}

\begin{cons}\label{UnivBivThy}
We will construct $\Mc_\Fs$ as a stably oriented bivariant theory. Given an equivalence class of morphisms $X \to Y$ in $\Cc$, define $\Mc_\Fs(X \to Y)$ as the free Abelian group generated by equivalence classes of confined morphisms $V \to X$ so that the composition $V \to Y$ is specialized. The operations are defined as follows:
\begin{enumerate}
\item if $X \to Y$ factors through a confined morphism $f: X \to X'$, then the bivariant pushforward $f_*: \Mc_\Fs(X \to Y) \to \Mc_\Fs(X' \to Y)$ is defined by linearly extending
$$[V \xrightarrow{h} X] \mapsto [V \xrightarrow{f \circ h} X'];$$
\item the bivariant pullback $g^*: \Mc_\Fs(X \to Y) \to \Mc_\Fs(Y' \times_Y X \to Y')$ is defined by linearly extending
$$[V \xrightarrow{h} X] \mapsto [Y' \times_Y V \xrightarrow{\mathrm{Id} \times_Y h} Y' \times_Y X];$$  
\item given a composable pair of morphisms $X \to Y \to Z$, the bivariant product 
$$\Mc_\Fs(X \to Y) \times \Mc_\Fs(Y \to Z) \to \Mc_\Fs(X \to Z)$$
is defined by bilinearly extending the following formula: given cycles $[V \to X] \in \Mc_\Fs(X \to Y)$ and $[W \to Y] \in \Mc_\Fs(Y \to Z)$, form the Cartesian diagram 
$$
\begin{tikzcd}
V' \arrow[]{d} \arrow[]{r} & X' \arrow[]{d} \arrow[]{r} & W \arrow[]{d} \\
V \arrow[]{r} & X \arrow[]{r} & Y \arrow[]{r} & Z
\end{tikzcd}
$$
and set $[V \to X] \bullet [W \to Y] := [V' \to X]$. This is well defined since both confined and specialized morphisms were assumed to be stable under pullbacks and compositions.
\end{enumerate}
The orientation $\theta$ is given by
$$\theta(f) = [X \xrightarrow{\mathrm{Id}} X] \in \Mc_\Fs(X \xrightarrow{f} Y).$$
It is clear that $\theta$ is stable under pullbacks.
\end{cons}

Let us record several immediate observations as the following result.

\begin{prop}\label{UnivTheoryProperties}
The stably oriented bivariant theory $\Mc_\Fs$ is commutative and generated by orientations in the sense of Proposition \ref{BivariantHomIdealProp}. Moreover, the external product is given by linearly extending the formula $[V \to X] \times [W \to Y] = [V \times W \to X \times Y].$ \qed 
\end{prop}

The theory $\Mc_\Fs$ is the universal stably oriented bivariant theory.

\begin{thm}[\cite{Yo1} Theorem 3.1 (2)]\label{UnivPropOfUnivBivThy}
If $\Bb^*$ is a stably oriented bivariant theory with functoriality $\Fs$, then there exists a unique orientation preserving Grothendieck transformation
$$\eta: \Mc_\Fs \to \Bb^*.$$
\end{thm}

\begin{rem}
The assumptions on $\Fs$ are not weakest possible to prove Theorem \fref{UnivPropOfUnivBivThy}: instead of requiring every Cartesian square to be independent, one could ask for \emph{$\Cc$-independence} (see \cite{Yo1} for details). However, this slightly more general statement is superfluous for our purposes.
\end{rem}

\subsection{Chern classes of line bundles}\label{BivChernSubSect}

The purpose of this section is to define and study Chern classes of line bundles in certain oriented bivariant theories. Our main result is that bivariant theories satisfying the section and the formal group law axioms come equipped with a canonical $\Lb$-linear structure, where $\Lb$ is the Lazard ring (Theorem \ref{FGLIsLLinearThm}). We will also define and study the basic properties of bivariant theories satisfying the snc axiom.

Let us start with the following useful definition.

\begin{defn}\label{AdmissibleFunctorialityDef}
An \emph{admissible functoriality} $\Fs = (\Cc, \Cs, \Is, \Ss)$ is a functoriality where
\begin{enumerate}
\item $\Cc$ is an \emph{admissible subcategory} of the $\infty$-category $d \mathrm{Sch}$ of derived schemes, i.e., a full subcategory $\Cc \subset d \mathrm{Sch}$ so that
\begin{enumerate}
\item $\Cc$ is closed under derived fibre products in $d \mathrm{Sch};$
\item $\Cc$ if $X \in \Cc$, then $\Pb^n \times X \in \Cc;$
\item $\Cc$ if $X \in \Cc$ and $\Ls$ is a line bundle over $X$, then $\Ls \in \Cc;$
\item $\Cc$ has a (distinguished) final object $pt$;
\end{enumerate}

\item the confined morphisms $\Cs$ contain all closed embeddings and all projections $\Pb^n \times X \to X$;

\item all the derived Cartesian squares in $\Cc$ are independent;

\item all quasi-smooth morphisms are specialized.
\end{enumerate}
\end{defn}

\begin{ex}
The functoriality $\Fs_A$ used in in Section \ref{FirstComparisonSect} is admissible.
\end{ex}

\begin{rem}
Notice that if $\Cc$ is an admissible subcategory of $d \mathrm{Sch}$, $X \in \Cc$, $\Ls$ is a line bundle over $X$, and if $s$ is a global section of $\Ls$, then the derived vanishing locus $Z_s$ of $s$ lies inside $\Cc$. 
\end{rem}

From now on, $\Fs = (\Cc, \Cs, \Is, \Ss)$ will be a fixed admissible functoriality. Let us begin by defining Chern classes.

\begin{defn}\label{GeneralChernClassDefn}
Let $\Bb^*$ be an oriented bivariant theory with functoriality $\Fs$. Let $X \in \Cc$, and $\Ls$ a line bundle on $X$ with zero-section $s_0: X \hookrightarrow \Ls$.  We define the \emph{first Chern class of $\Ls$} as 
$$c_1^{\Bb^*}(\Ls) := s_0^* s_{0!} (1_X) \in \Bb^*(X).$$
If there is no danger of confusion, we will usually use the shorthand notation $c_1(\Ls)$.
\end{defn}

Chern classes are often natural in pullbacks as the next result shows.

\begin{lem}
If $\Bb^*$ is a stably oriented bivariant theory on $\Fs$, then the first Chern classes of line bundles are natural under pullbacks.
\end{lem}
\begin{proof}
Indeed, given a morphism $f: X \to Y$ and a line bundle $\Ls$ on $Y$, we can form the derived Cartesian square
$$
\begin{tikzcd}
X \arrow[]{d}{f} \arrow[]{r}{s'_0} & f^*\Ls \arrow[]{d}{\tilde f} \\
Y \arrow[]{r}{s_0} & \Ls
\end{tikzcd}
$$
and then compute that
\begin{align*}
f^*(c_1(\Ls)) &= f^*\big(s^*_0(s_{0!}(1_Y))\big) \\
&= s'^*_0\big(\tilde{f}^*(s_{0!}(1_Y))\big) \\
&= s'^*_0\big(\tilde{f}^*(s_{0*}(\theta(s_0)))\big) \\
&= s'^*_0\big(s'_{0*}(f^*(\theta(s_0)))\big) & (\text{bivariant axiom $A_{23}$}) \\
&= s'^*_0\big(s'_{0*}(\theta(s'_0))\big) & (\text{stability of $\theta$}) \\
&= s'^*_0\big(s'_{0!}(1_X)\big) \\
&= c_1(f^* \Ls)
\end{align*}
which proves the claim.
\end{proof}

In order to restrict our attention to bivariant theories whose Chern classes admit a simple geometric description, we make the following definition.

\begin{defn}\label{SectionAxiomDef}
Let $\Bb^*$ be an oriented bivariant theory with functoriality $\Fs$. We say that $\Bb^*$ satisfies the \emph{section axiom} if for all $X \in \Cc$, for all line bundles $\Ls$ on $X$, and for all global sections $s$ of $\Ls$, we have the equality
$$c_1(\Ls) = i_{s!}(1_{Z_s}) \in \Bb^*(X),$$
where $i_s: Z_s \hookrightarrow X$ is the inclusion of the derived vanishing locus of $s$. 
\end{defn}

\begin{rem}\label{SectionAxiomForZeroSectRem}
Using bivariant push-pull formula, we see that the equality
$$c_1(\Ls) = i_{s_0 !}(1_{Z_{s_0}}) $$
always holds. Hence, the section axiom merely asks that $i_{s!}(1_{Z_s}) = i_{s'!}(1_{Z_{s'}})$ for any two global sections $s,s'$ of $\Ls$.
\end{rem}

If $\Bb^*$ satisfies the section axiom and the orientation is stable, it is possible to give simple descriptions for products of Chern classes.

\begin{prop}\label{ProductsOfChernProp}
Suppose $\Bb^*$ is a stably oriented bivariant theory that satisfies the section axiom, and let $\Ls_1,..., \Ls_r$ be line bundles on a $X \in \Cc$. Then, if we denote by $i: Z \hookrightarrow X$ the inclusion of the derived intersection of all the $Z_{s_i}$ inside $X$, we have that
$$c_1(\Ls_1) \bullet \cdots \bullet c_1(\Ls_r) = i_!(1_Z) \in \Bb^*(X).$$
In particular, the first Chern class of a globally generated line bundle is nilpotent. 
\end{prop}
\begin{proof}
This follows immediately from Lemma \ref{IntersectionProductLem} below.
\end{proof}

\begin{lem}\label{IntersectionProductLem}
Let $\Bb^*$ be a stably oriented bivariant theory on $\Fs$. Let $X \in \Cc$, and let $f_1: V_1 \to X$ and $f_2: V_2 \to X$ be confined and specialized. Then
$$f_{1!}(1_{V_1}) \bullet f_{2!}(1_{V_2}) = f_{12!}(1_{V_1 \times_X V_2}),$$
where $f_{12}$ is the induced morphism $V_1 \times_X V_2 \to X$.
\end{lem}
\begin{proof}
As the square
$$
\begin{tikzcd}
V_1 \times_X V_2 \arrow[]{d}{\mathrm{pr}_1} \arrow[]{r}{\mathrm{pr}_2} & V_2 \arrow[]{d}{f_2} \\
V_1 \arrow[]{r}{f_1} & X
\end{tikzcd}
$$
is derived Cartesian, we compute that
\begin{align*}
f_{1!}(1_{V_1}) \bullet f_{2!}(1_{V_2}) &= f_{1*}(\theta(f_1)) \bullet f_{2*}(\theta(f_2)) \\
&= f_{1*} \big(\theta(f_1) \bullet f_{2*}(\theta(f_2))\big) & (\text{bivariant axiom $A_{12}$}) \\
&= f_{1*}\Big(\mathrm{pr}_{1*} \big(f_2^*(\theta(f_1)) \bullet \theta(f_2)\big)\Big) & (\text{bivariant axiom $A_{123}$}) \\
&= f_{1*}\Big(\mathrm{pr}_{1*} \big(\theta(\mathrm{pr}_2) \bullet \theta(f_2)\big)\Big) & (\text{stability of $\theta$})\\
&= f_{12*}(\theta(f_{12})) \\
&= f_{12!}(1_{V_1 \times_X V_2}),
\end{align*}
which is exactly what we wanted to show.
\end{proof}

Another desirable property for Chern classes is that their behavior under tensor products should be governed by a formal group law. This is encapsulated by the following definition.

\begin{defn}\label{FGLAxiomDef}
Let $\Bb^*$ be an oriented bivariant theory with functoriality $\Fs$. We say that $\Bb^*$ satisfies the \emph{formal group law axiom}  if the Chern classes $c_1(\Ls)$ are nilpotent and there exists a formal group law $F^{\Bb^*}(x,y) \in \Bb^*(pt)[[x,y]]$ so that for all $X \in \Cc$ and for all line bundles $\Ls_1$ and $\Ls_2$ on $X$,
$$c_1(\Ls_1 \otimes \Ls_2) = F^{\Bb^*}\bigl(c_1(\Ls_1), c_1(\Ls_2)\bigr) \in \Bb^*(X).$$
\end{defn}

The following result shows that $F^{\Bb^*}$ is unique in good situations.

\begin{lem}\label{FGLUniqueLem}
Suppose that $\Bb^*$ is a stably oriented bivariant theory satisfying the section axiom and suppose the Chern classes $c_1(\Ls)$ are nilpotent for all line bundles $\Ls$. If there exists two power series $F(x,y), F'(x,y) \in \Bb^*(pt)[[x,y]]$ so that for all $X$ and for all line bundles $\Ls_1$ and $\Ls_2$ on $X$, the equality
$$F_1\big(c_1(\Ls_1), c_1(\Ls_2)\big) = F_2\big(c_1(\Ls_1), c_1(\Ls_2)\big) \in \Bb^*(X)$$
holds, then $F_1 = F_2$.
\end{lem}
\begin{proof}
It is clearly enough to show that a formal power series $F(x,y)$ is $0$ if $F\big(c_1(\Ls_1), c_1(\Ls_2)\big)$ always vanishes. Let us denote 
$$F(x,y) = \sum_{i,j \geq 0} b_{ij} x^i y^j.$$
If $F$ is not $0$, then we can find a minimal nonzero coefficient $b_{nm}$ in the sense that $i \leq n$, $j \leq m$ and $b_{ij} \not = 0$ imply that $i=n$ and $j=m$. 

Consider then the line bundle $\Oc(1,1)$ on $\Pb^n \times \Pb^m$. Using Proposition \ref{ProductsOfChernProp} to obtain geometric representatives of products of Chern classes, we may conclude that
$$c_1(\Oc(1,0))^i \bullet c_1(\Oc(0,1))^j = 
\begin{cases}
0 & \text{if $i>n$ or $j>m$;} \\
\iota_!\big(1_{pt}\big) & \text{if $i=n$ and $j=m$,} 
\end{cases}
$$
where $\iota$ is a constant section $pt \hookrightarrow \Pb^n \times \Pb^m$. Hence, 
\begin{align*}
b_{nm} &= \pi_{\Pb^n \times \Pb^m !}\Big(F\big(c_1(\Oc(1,0)), c_1(\Oc(0,1)) \big)\Big) \\
&= \pi_{\Pb^n \times \Pb^m !}(0) \\
&= 0
\end{align*}
contradicting the choice of $b_{nm}$.
\end{proof}

Moreover, the formal group law is preserved under orientation preserving Grothendieck transformations.

\begin{prop}\label{OrientedGTranPreserveFGLProp}
Let $\eta: \Bb^*_1 \to \Bb^*_2$ be an orientation preserving Grothendieck transformation, and suppose $\Bb^*_1$ is stably oriented and satisfies the section and the formal group law axioms. Then also $\Bb^*_2$ is stably oriented and satisfies the section and the formal law axioms, and moreover
$$F^{\Bb^*_2}(x,y) := \eta\big(F^{\Bb^*_1}(x,y)\big) \in \Bb^*_2(pt)[[x,y]].$$ 
is the unique formal power series $F$ so that 
$$c_1(\Ls_1 \otimes \Ls_2) = F\big(c_1(\Ls_1), c_1(\Ls_2)\big) \in \Bb_2^*(X)$$
for all $X$, $\Ls_1$ and $\Ls_2$.
\end{prop}
\begin{proof}
The first three claims follow immediately from the fact that $\eta$ is an orientation preserving Grothendieck transformation. The uniqueness of $F^{\Bb^*_2}$ follows from Lemma \ref{FGLUniqueLem}.
\end{proof}

The above result has interesting consequences. Recall that the \emph{Lazard ring} $\Lb$ is the coefficient ring of the \emph{universal formal group law}
$$F_\univ(x,y) := \sum_{i,j \geq 0} a_{ij} x^i y^j \in \Lb[[x,y]]$$
and that for a commutative ring $R$ there is a bijective correspondence between formal groups laws with coefficients in $R$ and ring homomorphisms $\Lb \to R$. In particular, if $\Bb^*$ is an oriented bivariant theory satisfying the formal group law axiom, then the formal group law $F^{\Bb^*}$ induces a morphism $\Lb \to \Bb^*(pt)$. We summarize our observations in the following theorem.

\begin{thm}\label{FGLIsLLinearThm}
Stably oriented bivariant theories satisfying the section and the formal group law axioms are $\Lb$-linear in a unique way compatible with the formal group law acting on Chern classes. Moreover, any orientation preserving Grothendieck transformation between such theories is $\Lb$-linear and preserves the formal group law.
\end{thm}
\begin{proof}
Indeed, as we already observed above, the choice of the formal group law $F^{\Bb^*}$ $\Lb$-linearizes $\Bb^*$. Since $F^{\Bb^*}$ is unique by Lemma \ref{FGLUniqueLem}, so is the associated $\Lb$-linearization. An orientation preserving Grothendieck transformation $\eta: \Bb^*_1 \to \Bb^*_2$ between such theories preserves the formal group law by Proposition \ref{OrientedGTranPreserveFGLProp}, so the induced morphism $\Bb^*_1(pt) \to \Bb^*_2(pt)$ is a map of $\Lb$-algebras, proving that $\eta$ is $\Lb$-linear by Proposition \ref{RLinearGTransProp}.
\end{proof}

We end this section by discussing the snc axiom, which is the third defining feature of bivariant algebraic cobordism (after the section and the formal group law axioms). Let us first clarify what we mean by an snc divisor in $\Cc$.

\begin{defn}\label{ASNCDivDef}
Let $W \in \Cc$ be a derived scheme smooth over the final object $pt$ in $\Cc$. A $\Cc$-snc divisor on $W$ is the data of an effective Cartier divisor $D \hookrightarrow W$, Zariski connected divisors $D_1,...,D_r \hookrightarrow W$ with 
$$D_I := \cap_{i \in I} D_i$$
smooth and of the expected relative dimension over $pt$ for all $I \subset \{1,2,...,r\}$, and positive integers $n_1,...,n_r$ such that
$$D \simeq n_1 D_1 + \cdots + n_r D_r$$
as effective Cartier divisors on $W$. 
\end{defn}

\begin{rem}
If $\Cc$ is a category of separated finite type derived schemes over a perfect field $k$, then the above definition recovers the usual definition of an snc divisor with its prime decomposition.
\end{rem}

Suppose then that $\Bb^*$ is a stably oriented bivariant theory satisfying the section axiom and the formal group law axiom with formal group law $F$. Let $W \in \Cc$ be smooth and of pure relative dimension over $pt$, and let $D \hookrightarrow W$ be an $\Cc$-snc divisor. Denoting by $+_F$ the formal addition given by the formal group law $F$ and by $[n]_F$ the formal multiplication (iterated formal addition), the formal power series
$$[n_1]_F \cdot x_1 +_F \cdots +_F [n_r]_F \cdot x_r$$ 
in $r$ variables has a unique expression of form 
$$\sum_{I \subset \{1,2,...,r \}} \textbf{x}^I F^{n_1,...,n_r}_{I}(x_1,...,x_r),$$
where 
$$\textbf{x}^I = \prod_{i \in I} x_i$$
and $F_{ I}^{n_1,...,n_r}(x_1,...,x_r)$ contains only variables $x_i$ such that $i \in I$ (see \cite{LM} Lemma 3.1.2 for details). With this notation we can make the following definition.

\begin{defn}\label{ZetaCLassDef}
Let everything be as above. We define
$$
\zeta_{W,D, \Bb^*} := \sum_{I \subset \{1,2,...,r\}} \iota^I_*\Biggl(F^{n_1,...,n_r}_I \Bigl( c_1\bigl(\Oc(D_1)\bigr), ..., c_1\bigl(\Oc(D_r)\bigr) \Bigr) \bullet \theta(\pi_{D_I}) \Biggr) \in \Bb^*(D \to pt),
$$
where $\iota^I$ is the inclusion $D_I \hookrightarrow D$. 
\end{defn}

The following result shows that we can drop $\Bb^*$ from the notation as $\zeta_{W,D,\Bb^*}$ does not depend on the bivariant theory.

\begin{lem}\label{ZetaClassIsNaturalLem}
Let $\eta: \Bb_1^* \to \Bb^*_2$ be an orientation preserving Grothendieck transformation between stably oriented bivariant theories, which satisfy the section and the formal group law axioms. If $D \hookrightarrow W$ is an $A$-snc divisor, then
$$\eta(\zeta_{W,D,\Bb^*_1}) = \zeta_{W,D,\Bb^*_2} \in \Bb^*_2(D \to pt).$$
\end{lem}
\begin{proof}
This follows immediately from the fact that $\eta$ preserves the formal group law and all other relevant structure.
\end{proof}

From now on, we shall denote $\zeta_{W,D,\Bb^*}$ simply by $\zeta_{W.D}$. The following definition allows us to restrict our attention to bivariant theories where fundamental classes looking like they should have a geometric decomposition, actually do admit such a decomposition.

\begin{defn}\label{SNCAxiomDef}
Let $\Bb^*$ be an oriented bivariant theory satisfying the section and the formal group law axioms. We say that $\Bb^*$ satisfies the \emph{snc axiom} if for all $\Cc$-snc divisors $D \hookrightarrow W$, we have
$$\zeta_{D,W} = \theta(\pi_D) \in \Bb_*(D).$$
\end{defn}

\section{Proof of Theorem \ref{FirstComparisonThm}}\label{FirstComparisonSect}
The purpose of this section is to prove Theorem \ref{FirstComparisonThm}, which is done at the end of Section \ref{ProofOfComparison1SubSect}. In Section \ref{DerivedBivariantCobConsSubSect} we recall the construction and the basic properties of the Lowrey--Schürg algebraic bordism $d \Omega^A_*$ and the bivariant derived algebraic cobordism $d \Omega^*_A$. In Section \ref{UnivPrecobConsSubSect} we recall the construction of universal precobordism and prove that it is the universal stably oriented additive bivariant theory satisfying the section and the formal group law axioms. We also construct the bivariant algebraic cobordism, and provide it with a similar universal property.

Throughout the section, we will denote by $A$ the base ring that is assumed to be Noetherian and of finite Krull dimension. Let $\Fs_A = (\Cc_A, \Cs_A, \Is_A, \Ss_A)$ be the functoriality where $\Cc$ is the homotopy category of quasi-projective derived $A$-schemes, where $\Cs$ is the class of \emph{projective morphisms}, where $\Is$ is the class of all \emph{derived Cartesian squares} and where $\Ss$ is the class of \emph{quasi-smooth morphisms}. We note that $\Fs_A$ is admissible in the sense of Definition \ref{AdmissibleFunctorialityDef}. To set up the stage, we introduce the following notions.

\begin{defn}\label{AdditivityDef}
A bivariant theory $\Bb^*$ with functoriality $\Fs_A$ is called \emph{additive} if the morphisms
$$\Bb^*(X_1 \to Y) \oplus \Bb^*(X_2 \to Y) \xrightarrow{\iota_{1*} + \iota_{2*}} \Bb^*\Bigr(X_1 \coprod X_2 \to Y\Bigr)$$
are isomorphisms for all $X_1, X_2$ and $Y$, where $\iota_i$ is the canonical inclusion $X_i \to X_1 \coprod X_2$.
\end{defn}

\begin{defn}\label{UnivAdditiveBivariantThyDef}
We will denote by $\Mc_{A,+}^*$ the universal \emph{additive} bivariant theory with functoriality $\Fs_A$, i.e., the theory which is obtained from $\Mc_{\Fs_A}$ by imposing relations which identify taking disjoint union with summation. It has a grading defined as follows: $\Mc_{A,+}^d(X \to Y)$ is the group completion of the Abelian monoid on equivalence classes
$$[V \to X]$$
of projective morphisms of quasi-projective derived $A$-schemes having the property that the composition $V \to Y$ is quasi-smooth and of relative virtual dimension $-d$. 
\end{defn}

\subsection{Construction of derived bordism and derived bivariant cobordism}\label{DerivedBivariantCobConsSubSect}

Our goal in this section is to give a detailed construction of Lowrey--Schürg algebraic bordism groups $d \Omega^{A}_*(X)$ which both generalizes the construction from \cite{LS} (Lowrey and Schürg work over a field), and fixes minor gaps in the original construction. We will simultaneously construct derived bivariant algebraic cobordism $d \Omega^*_A$ as the finest oriented bivariant theory on $\Fs_A$ extending $d\Omega^{A}_*(X)$. Let us denote
$$\Mc^{A,+}_*(X) := \Mc^{-*}_{A,+}(X \to pt).$$
Given a bivariant theory $\Bb^*$ on $\Fs_A$, and subsets $\Sc(X) \subset \Bb_*(X)$, we will denote by $\langle \Sc \rangle_{\Bb_*}$ the \emph{homology ideal} generated by $\Sc$, i.e., the smallest collection of subsets containing $\Sc$ that is stable under sums, pushforwards, Gysin pullbacks, and external products by arbitrary elements. Note that $\Sc$ and $\langle \Sc \rangle_{\Bb_*}$ clearly generate the same bivariant ideal inside $\Bb^*$.

We will construct $d \Omega^{A}_*$ as a quotient of $\Lb_* \otimes \Mc^{A,+}_*$. We will use the following fact several times: since $\Lb_* \otimes \Mc^{A,+}_*$ is generated by orientations in the sense of Proposition \ref{BivariantHomIdealProp}, any $\Bb_* := \Lb_* \otimes \Mc^{A,+}_* / \Ic$ is the associated homology theory of the bivariant theory $\Bb^* := \Lb^* \otimes \Mc_{A,+}^* / \langle \Ic \rangle$. Another useful fact that we will need is the following result.

\begin{lem}\label{StrongOrientationLem}
Let $\Bb^*$ be an oriented bivariant theory with functoriality $\Fs_A$. Then the orientation of $\Bb^*$ is strong along smooth morphisms. In other words, if $p: Y' \to Y$ is smooth, then the induced morphism
$$- \bullet \theta(p): \Bb^*(X \to Y') \to \Bb^*(X \to Y)$$
is an isomorphism.
\end{lem}
\begin{proof}
This is a special case of \cite{An1} Proposition 2.9.
\end{proof}

We will construct $d \Omega^*_A$ and $d \Omega^A_*$ in several steps following Section 3 of \cite{LS}.

\begin{cons}[Cf. \cite{LS} Definition 3.4]\label{NaiveDerivedCobCons}
We define
$$d\Omega^{A,\mathrm{naive}}_* := \Lb_* \otimes \Mc^{A,+}_* / \Rc^\mathrm{fib},$$
where $\Rc^\mathrm{fib}$ is the homology ideal of homotopy fibre relations. In other words, $\Rc^\mathrm{fib}_A(X)$ is the $\Lb$-module generated by elements of form
$$[W_0 \to X] - [W_\infty \to X] \in \Lb_* \otimes \Mc^{A,+}_*(X),$$
where $W \to \Pb^1 \times X$ is a projective morphism with $W$ quasi-smooth over $A$, and where $W_0$ and $W_\infty$ are the fibres of $W \to \Pb^1$ over $0$ and $\infty$ respectively. This is the homology theory associated to
$$d \Omega^*_{A, \naive} :=  \Lb^* \otimes \Mc_{A,+}^* / \langle \Rc^\mathrm{fib} \rangle.$$
\end{cons}

\begin{lem}
The bivariant theory $d \Omega^*_{A, \naive}$ satisfies the section axiom in the sense of Definition \ref{SectionAxiomDef}.
\end{lem}
\begin{proof}
By Lemma \ref{StrongOrientationLem} we have that 
$$[\{ 0 \} \to \Pb^1] = [\{ \infty \} \to \Pb^1] \in d \Omega^*_{A, \naive}( \Pb^1).$$
If $X$ is a derived scheme, $\Ls$ is a line bundle on $X$ and $s$ is a global section of $\Ls$, then $x_0 s$ is a global section of $\Ls(1)$ on $\Pb^1 \times X$. Denoting by $W$ the derived vanishing locus of $x_0 s$, and by $f$ and $g$ the induced morphisms $W \to \Pb^1$ and $W \to X$ respectively, we compute that
\begin{align*}
[Z_s \to X] &= g_! \big( f^*( [\{ \infty \} \to \Pb^1] ) \big) \\
&= g_! \big( f^*( [\{ 0 \} \to \Pb^1] ) \big) \\
&= [Z_{s_0} \to X] \\
&= c_1(\Ls), & (\text{Remark \ref{SectionAxiomForZeroSectRem}}) 
\end{align*}
which is exactly what we wanted to prove.
\end{proof}

\begin{cons}[Cf. \cite{LS} Definition 3.16 and Lemma 3.17]\label{BivariantDerivedPrecobCons}
The theory $d\Omega^*_{A,\mathrm{pre}}$ is constructed from $\Omega^*_{A,\mathrm{naive}}$ in two steps. 
\begin{enumerate}
\item First of all, recall that globally generated line bundles have nilpotent Chern classes by Proposition \ref{ProductsOfChernProp}. We may therefore define $\Rc^\mathrm{fgl}(X) \subset d \Omega^{A,\naive}_*(X)$ as consisting of elements of form
$$\big(c_1(\Ls_1 \otimes \Ls_2) - F_\univ(c_1(\Ls_1), c_1(\Ls_2))\big) \bullet [X \to X]$$
for $X$ smooth over $A$ and $\Ls_i$ globally generated line bundles on $X$. Set
$$d \Omega'^{A, \pre}_* := d \Omega^{A,\naive}_* / \langle \Rc^\mathrm{fgl} \rangle_{d \Omega^{A,\naive}_*}.$$
This is the homology theory associated to $d \Omega'^*_{A, \pre} := d \Omega^*_{A, \naive} / \langle \Rc^\mathrm{fgl} \rangle_{\Omega^*_{A,\naive}}.$

\item Secondly, we set $d\Omega^{A, \pre}_* := d\Omega'^{A, \pre}_* / \langle \Rc^\mathrm{fgl}_+ \rangle_{d \Omega'^{A, \pre}_*}$, where $\Rc^\mathrm{fgl}_{+}(X)$ for $X$ smooth over $A$ consists of 
$$\Big( c_1(\Ls) - F_\univ\bigl(c_1(\Ls_1), \iota(c_1(\Ls_2)) \bigr)\Big) \bullet [X \to X],$$
where $\Ls_1$ and $\Ls_2$ are globally generated, $\Ls \simeq \Ls_1 \otimes \Ls_2^\vee$, and $\iota$ is the formal inverse power series of $F_\univ$ (see e.g. \cite{H} Lemma 1.1.4). This is the homology theory associated to 
$$d \Omega^*_{A, \pre} :=  d \Omega'^*_{A, \pre} / \langle \Rc^\mathrm{fgl}_+ \rangle_{d \Omega'^*_{A, \pre}}.$$ 
\end{enumerate}
\end{cons}

\begin{war}\label{LSFGLWar}
Note that the second step of Construction \ref{BivariantDerivedPrecobCons} is forcing Lemma 3.17 of \cite{LS} to be true. There seems to be a mistake in the proof of the said theorem in \emph{loc. cit.}: the authors claim that the equality 
$$c_1(\Ls_1 \otimes \Ls_2^\vee) = F_\univ(c_1(\Ls_1), \iota(c_1(\Ls_2)))$$
of operators on $d \Omega'^{A,\pre}_*(X)$, where $\Ls_1$ and $\Ls_2$ globally generated on $X$ smooth over $A$, can be proven similarly to \cite{LP}, with which they seem to be referring to \cite{LP} Lemma 9.1. What can actually be proven is that 
$$F_\univ\big(c_1(\Ls_1), \iota(c_1(\Ls_2))\big)$$
only depends on the equivalence class of $\Ls_1 \otimes \Ls_2^\vee$, but it is not clear why this would equal to $c_1 (\Ls_1 \otimes \Ls_2^\vee)$.
\end{war}

\begin{lem}\label{DerivedPrecobFGLLem}
The bivariant theory $d \Omega^*_{A,\pre}$ satisfies the formal group law axiom for he formal group law coming from the $\Lb$-linear structure: given line bundles $\Ls_1$ and $\Ls_2$ on a quasi-projective $A$-scheme $X$, we have that
$$c_1(\Ls_1 \otimes \Ls_2) = F_\univ\big(c_1(\Ls_1), c_1(\Ls_2)\big) \in d \Omega^*_{A, \pre}(X).$$
\end{lem}
\begin{proof}
As in \cite{LS} Proposition 3.19 one proves that 
$$\big( c_1(\Ls_1 \otimes \Ls_2) - F_\univ(c_1(\Ls_1), c_1(\Ls_2)) \big) \bullet [X \to X] = 0 \in d \Omega^{A,\pre}_*(X)$$
for all $X \in \Cc_A$. In the case when $X$ is smooth over $A$, Lemma \ref{StrongOrientationLem} lifts this to an equality of Chern classes in $d \Omega^*_{A, \pre}(X)$, so in particular we obtain the desired formula for $X_n = (\Pb^n)^{\times 4}$, $\Ls_1 = \Oc(1,-1,0,0)$ and $\Ls_2 = \Oc(0,0,1,-1)$. In general, given $X \in \Cc_A$ and line bundles $\Ls'_1$ and $\Ls'_2$ on $X$, then, for $n \gg 0$, there exists a morphism $f: X \to X_n$ so that $\Ls'_i \simeq f^* \Ls_i$. Therefore the desired formula for $X$ and $\Ls'_i$ follows from the formula for $X_n$ and $\Ls_i$.
\end{proof}

We are finally ready to construct $d \Omega^*_A$ and $d \Omega^A_*$.

\begin{cons}[Cf. \cite{LS} Definition 3.20]\label{BivariantDerivedCobCons} 
We define $\Rc^\snc \subset d \Omega_*^{A,\pre}$ as the collection of homology elements containing
$$\zeta_{W,D} - [D \to D] \in d \Omega^{A,\pre}_*(D)$$
where $D \hookrightarrow W$ ranges over all $\Cc_A$-snc divisors in all $W$ smooth over $A$, and $\zeta_{W,D}$ is as in Definition \ref{ZetaCLassDef}. Then the \emph{derived algebraic $A$-bordism} is defined as
$$d \Omega^A_*(X) := d \Omega^{A,\pre}_* / \langle \Rc^\snc \rangle_{d\Omega^{A,\pre}_*},$$
and it is the homology theory associated to
$$d\Omega^*_A := d \Omega^*_{A, \pre}/ \langle \Rc^\snc \rangle_{d \Omega^*_{A,\pre}}.$$
\end{cons}

\subsection{Construction and the universal property of universal precobordism}\label{UnivPrecobConsSubSect}

The purpose of this section is to recall the construction of the universal precobordism theory from \cite{AY} Section 6.1, and to provide it with a convenient universal property. Afterwards, we will construct bivariant algebraic cobordism $\Omega^*_A$ as a quotient of $\PCob^*_A$, and provide $\Omega^*_A$ with a universal property.

\begin{cons}\label{UnivPrecobCons}
The \emph{universal $A$-precobordism} theory $\PCob^*_A$ is constructed as $\Mc^*_{A,+} / \Rc^\mathrm{dpc}_A$, where $\Rc^\mathrm{dpc}_A$ is the bivariant ideal of double point cobordism relations. In other words, $\Rc^\mathrm{dpc}_A(X \to Y)$ is generated, as an Abelian group, by elements of form
$$[W_0 \to X] - [D_1 \to X] - [D_2 \to X] + [\Pb_{D_1 \cap D_2}(\Oc_W(D_1) \oplus \Oc) \to X],$$
where $W \to \Pb^1 \times X$ is projective, the composition $W \to \Pb^1 \times Y$ is quasi-smooth, $W_0$ is the fibre over $0$, and $D_1$ and $D_2$ are virtual Cartier divisors on $W$ whose sum is the fibre over $\infty$. It is an easy exercise to show that $\Rc^\mathrm{dpc}$ is indeed a bivariant ideal.
\end{cons}

Let us recall the following result from \cite{An2}.

\begin{prop}
The bivariant theory $\PCob^*_A$ satisfies the section and the formal group law axioms.
\end{prop}
\begin{proof}
Indeed, the formal group law axiom holds by \cite{An2} Theorem 3.4, and proving that section axiom holds is done with the standard argument (see e.g. \cite{LS} Lemma 3.9).
\end{proof}

Recall that by \cite{An2} Theorem 3.11 the associated cohomology theory of $\PCob^*_A$ is the universal oriented cohomology theory on quasi-projective derived $A$-schemes. We will extend this to a universal property of the whole bivariant theory. In order to do so, we make the following simple observation.

\begin{lem}\label{CohomGenRelsForPrecob}
Let $\Rc^\mathrm{dpc}_A \subset \Mc^*_{A,+}$ be the bivariant ideal of double point cobordisms. Then 
$$\Rc^\mathrm{dpc}_A = \langle \Rc^\mathrm{dpc}_{A,\mathrm{coh}} \rangle,$$
where $\Rc^\mathrm{dpc}_{A,\mathrm{coh}}$ is the bivariant subset of $\Mc^*_{A,+}$ with 
$$\Rc^\mathrm{dpc}_{A,\mathrm{coh}}(X \xrightarrow{\mathrm{Id}} X) = \Rc^\mathrm{dpc}_A(X \xrightarrow{\mathrm{Id}} X)$$ 
and $\Rc^\mathrm{dpc}_{A,\mathrm{coh}}(X \to Y)$ is empty otherwise.
\end{lem}
\begin{proof}
Since $\Rc^\mathrm{dpc}_A$ is a bivariant ideal containing $\Rc^\mathrm{dpc}_{A,\mathrm{coh}}$, we have that $\langle \Rc^\mathrm{dpc}_{A,\mathrm{coh}} \rangle \subset \Rc^\mathrm{dpc}_A$. We are therefore left to show the reverse inclusion $\Rc^\mathrm{dpc}_A \subset \langle \Rc^\mathrm{dpc}_{A,\mathrm{coh}} \rangle$. To show this, consider $f: X \to Y$ and a projective morphism $ g: W \to \Pb^1 \times X$ so that the composition $(\mathrm{Id} \times f) \circ g: W \to \Pb^1 \times Y$ is quasi-smooth, giving rise to an element 
$$\alpha := [W_0 \to X] - [D_1 \to X] - [D_2 \to X] + [\Pb_{D_1 \cap D_2}(\Oc_W(D_1) \oplus \Oc) \to X] \in \Rc^\mathrm{dpc}_A(X \to Y).$$
Since
$$\beta := [W_0 \to W] - [D_1 \to W] - [D_2 \to W] + [\Pb_{D_1 \cap D_2}(\Oc_W(D_1) \oplus \Oc) \to W]$$
is in $\Rc^\mathrm{dpc}_{A,\mathrm{coh}} (W \to W)$, we can conclude that
$$\alpha = \mathrm{pr}_{2*} \big( g_*\bigl( \beta \bullet \theta\bigl(\mathrm{pr}_2 \circ (\mathrm{Id} \times f) \circ g\bigr) \bigr)\big)$$
is in $\langle \Rc^\mathrm{dpc}_{A,\mathrm{coh}} \rangle(X \to Y),$ proving the claim.
\end{proof}

\begin{thm}\label{UnivPropOfBivariantPrecob}
Let $\Bb^*$ be a stably oriented additive bivariant theory on $\Fs_A$. Then there exists a unique orientation preserving Grothendieck transformation 
$$\eta: \PCob^*_A \to \Bb^*$$
if and only if $\Bb^*$ satisfies the section and the formal group law axioms. In other words, $\PCob^*_A$ is the universal stably oriented additive bivariant theory satisfying the section and the formal group law axioms.
\end{thm}
\begin{proof}
The only if direction follows from Proposition \ref{OrientedGTranPreserveFGLProp}. Let us prove the other direction. By the universal property of $\Mc^*_{A,+}$ there exists a unique orientation preserving Grothendieck transformation $\eta': \Mc^*_{A,+} \to \Bb^*$, so all that is left to do is to show that $\eta'$ factors through the canonical surjection $\Mc^*_{A,+} \to \PCob^*_A$. However, since being stably oriented and satisfying the section and the formal group law axioms is precisely what is needed in order for the cohomology theory associated to $\Bb^*$ to be oriented in the sense of \cite{An2} Definition 3.6 (the other requirements follow from the standard properties of Gysin morphisms in bivariant theories and the stability of the orientation), it follows from \emph{op. cit.} Theorem 3.11 that $\eta'$ must send all the elements of $\Rc^\mathrm{dpc}_\mathrm{coh}$ to zero. But this implies that $\eta'$ kills $\langle \Rc^\mathrm{dpc}_\mathrm{coh} \rangle$, so the well definedness of $\eta$ follows from Lemma \fref{CohomGenRelsForPrecob}.
\end{proof}

Finally, we define bivariant algebraic cobordism as a quotient of $\PCob^*_A$ by enforcing the snc axiom to hold.

\begin{cons}\label{BivariantAlgCobCons}
Define $\Rc'^\snc \subset  \PCob^*_A$ as the bivariant subset containing
$$\zeta_{W,D} - \theta(\pi_D) \in \PCob^*_A(D \to pt)$$
where $D \hookrightarrow W$ ranges over all $\Cc_A$-snc divisors in all $W$ smooth over $A$. \emph{Bivariant algebraic $A$-cobordism} is then defined as the quotient
$$\Omega^*_A :=  \PCob^*_A / \langle \Rc'^\snc \rangle.$$
\end{cons}

As a consequence of Theorem \ref{UnivPropOfBivariantPrecob}, we obtain a pleasant universal property for $\Omega^*_A$.

\begin{cor}\label{UnivPropOfBivariantAlgCob}
Let $\Bb^*$ be a stably oriented additive bivariant theory on $\Fs_A$. Then there exists a unique orientation preserving Grothendieck transformation 
$$\eta: \Omega^*_A \to \Bb^*$$
if and only if $\Bb^*$ satisfies the section, the formal group law and the snc axioms. In other words, $\Omega^*_A$ is the universal stably oriented additive bivariant theory satisfying the section, the formal group law and the snc axioms.
\end{cor}
\begin{proof}
Follows immediately from Theorem \ref{UnivPropOfBivariantPrecob} and Lemma \ref{ZetaClassIsNaturalLem}.
\end{proof}

\subsection{Proof of the comparison theorem}\label{ProofOfComparison1SubSect}

In this section we prove Theorem \ref{FirstComparisonThm}. We will first find an orientation preserving Grothendieck equivalence $\eta'$ between the universal precobordism $\PCob^*_A$ and the derived $A$-precobordism theory $d \Omega^*_{A, \pre}$ from Construction \ref{BivariantDerivedPrecobCons}. After that, it is relatively simple to find the desired Grothendieck equivalence between $d \Omega^*_A$ and the bivariant algebraic cobordism $\Omega^*_A$.

As $d \Omega^*_{A, \pre}$ satisfies the section and the formal group law axioms, it admits a unique orientation preserving Grothendieck transformation from $\PCob^*_A$, so we are left to show that this is an isomorphism.

\begin{thm}\label{UnivPrecobVsDerivedPrecobThm}
The unique Grothendieck transformation
$$\eta': \PCob^*_A \to d\Omega^*_{A, \pre}$$
is an isomorphism.
\end{thm}
\begin{proof}
Since $\eta'$ is $\Lb$-linear by Theorem \ref{FGLIsLLinearThm}, $\eta'$ must be surjective as the $\Lb$-span of its image is clearly $d\Omega^*_{A, \pre}$. We are left to show the injectivity of $\eta'$, which takes rest of the proof. We do this by constructing a retraction for $\eta'$. Observe that the $\Lb$-linear structure on $\PCob^*_A$ gives rise to a unique orientation preserving Grothendieck transformation $r: \Lb^* \otimes \Mc_{A,+}^* \to \PCob^*_A$. Our plan is to go through the construction of $d\Omega^*_{A,\pre}$ as a quotient of $\Lb^* \otimes \Mc_{A,+}^*$, and show that $r$ is well defined modulo all the imposed relations. 

Let $W \to \Pb^1 \times X$ be projective morphism with the composition $W \to \Pb^1 \times Y$ quasi-smooth, and let $W_0$ and $W_\infty$ be the fibres of $W \to \Pb^1$ over $0$ and $\infty$ respectively. As 
$$[W_0 \to X] - [W_\infty \to X] = 0 \in \PCob^*_A(X \to Y),$$
it follows that $r$ factors as $r': d\Omega^*_{A,\mathrm{naive}} \to \PCob^*_A$. 

Consider then line bundles $\Ls_1$ and $\Ls_2$ on a quasi-projective derived $A$-scheme $X$. Since $r'$ sends the universal formal group law $F_\univ$ to $F^{\PCob^*_A}$, we compute that, if $\Ls_1$ and $\Ls_2$ are globally generated, then
\begin{align*}
r' \big( F_\univ(c_1(\Ls_1), c_1(\Ls_2)) - c_1(\Ls_1 \otimes \Ls_2) \big) &= F^{\PCob^*_A}(c_1(\Ls_1), c_1(\Ls_2)) - c_1(\Ls_1 \otimes \Ls_2) \\
&= 0 \in \PCob^*_A(X).
\end{align*}
This shows that $r'$ factors as $r'': d\Omega'^*_{A, \pre} \to \PCob^*_A$. Finally, since
$$c_1(\Ls_1  \otimes \Ls_2) = F^{\PCob^*_A}(c_1(\Ls_1), c_1(\Ls_2)) \in \PCob^*_A(X)$$
holds for any line bundles $\Ls_1$ and $\Ls_2$ on $X$, it follows formally that also the equality
$$c_1(\Ls_1 \otimes \Ls_2^\vee) = F^{\PCob^*_A}\big(c_1(\Ls_2), \iota(c_1(\Ls_2))\big)$$
holds for any $\Ls_i$ on $X$. In particular the latter equation is true for $\Ls_1$ and $\Ls_2$ globally generated, showing that $r''$ factors as $r''': d\Omega^*_{A, \pre} \to \PCob^*_A$. Since $r'''$ is orientation preserving, the composition $r''' \circ \eta'$ must be the identity transformation by Theorem \ref{UnivPropOfBivariantPrecob}. Thus $\eta_1$ is injective, and we are done.
\end{proof}

Proof of Theorem \ref{FirstComparisonThm} is then essentially finished by the following observation, which is an immediate consequence of Lemma \ref{ZetaClassIsNaturalLem}.

\begin{lem}\label{SncToSncLem}
The Grothendieck equivalence $\eta': \PCob^*_A \to d\Omega^*_{A, \pre}$ sends $\Rc'^{\mathrm{snc}}_A \subset \PCob^*_A$ to $\Rc^{\mathrm{snc}}_A \subset d\Omega^*_{A,\pre}$. \qed
\end{lem}

We are then ready to prove the main result of the section. 

\begin{proof}[Proof of Theorem \ref{FirstComparisonThm}]
Indeed, form the square
$$
\begin{tikzcd}
& \PCob^*_A \arrow[twoheadrightarrow]{d} \arrow[]{r}{\cong}[swap]{\eta'} & d \Omega^*_{A, \pre} \arrow[twoheadrightarrow]{d} & \\
\Omega^*_A \arrow[dash, shift left]{r} \arrow[dash]{r} & \PCob^*_A/\langle \Rc'^{\mathrm{snc}}_A \rangle \arrow[]{r}{\cong}[swap]{\eta} & d \Omega^*_{A, \pre} / \langle \Rc^{\mathrm{snc}}_A \rangle \arrow[dash, shift left]{r} \arrow[dash]{r} & d \Omega^*_A
\end{tikzcd}
$$
using Lemma \ref{SncToSncLem}. Then $\eta$ is the desired Grothendieck equivalence. Moreover, it is the only orientation preserving Grothendieck transformation $\Omega^*_A \to d \Omega^*_A$ by Corollary \ref{UnivPropOfBivariantAlgCob}. 
\end{proof}

\section{Proof of Theorem \ref{SecondComparisonThm}}\label{SecondComparisonSect}

The purpose of this section is to compare the two operational bivariant theories $d \op \Omega^*_\LM$ and $\op d \Omega^*_A$, which are built in a formal fashion from the homology theories $\Omega^\LM_*$ and $d\Omega^A_*$ respectively. We start by recalling the operational cobordism theory $\op \Omega^*_\LM$ and proving that it is commutative in Section \ref{OperationalCobReviewSect}. Then, in Section \ref{OpDCobConsSubSect} we construct operational derived cobordism theory $\op d \Omega^*_A$, and prove that it receives a unique orientation preserving Grothendieck transformation from any commutative stably oriented stably oriented bivariant extension of $d \Omega^A_*$. Finally, in Section \ref{ProofOfComparison2SubSect}, we use the virtual pullbacks constructed in Appendix \ref{VirtualPullbackSect} to equip $d \op \Omega^*_\LM$ with orientations along quasi-smooth morphisms, and prove that the resulting Grothendieck transformation
$$\iota: d \op \Omega^*_\LM \to \op d \Omega^*_k,$$
where $k$ is a field of characteristic 0, is an isomorphism, proving Theorem \ref{SecondComparisonThm}.

\subsection{Review of operational cobordism}\label{OperationalCobReviewSect}

The main purpose of this section is to recall the construction of $\op \Omega^*_\LM$ from \cite{KLG} (at least in the case of quasi-projective schemes). We will also study its basic properties, and in particular we will prove that $d \Omega^*_\LM$ is commutative. We will denote by $\Fs^\cl_k$ the functoriality $(\Cc^\cl_k, \Cs^\cl_k, \Is^\cl_k, \Ss^\cl_k)$ where $\Cc$ is the category of quasi-projective $k$-schemes, where $\Cs$ is the class of \emph{projective morphisms}, where $\Is$ is the class of all \emph{Cartesian squares} and where $\Ss$ is the class of \emph{lci morphisms}. Note that lci morphisms are not stable under pullbacks.

\begin{cons}\label{OpCobCons}
We will construct $\op \Omega^*_\LM$ as an oriented bivariant theory with functoriality $\Fs^\cl_k$. Given a morphism $X \to Y$ in $\Cc^\cl_k$, a \emph{degree $d$ operational cobordism class} $c \in \op \Omega^d_\LM(X \to Y)$ consists of a collection of morphisms
$$c_g: \Omega^\LM_*(Y') \to \Omega^\LM_{*-d}(Y' \times_Y X)$$
indexed by morphisms $g: Y' \to Y$, which have to satisfy the following conditions:
\begin{enumerate}
\item[$(C_1)$] if $h: Y'' \to Y'$ is projective, then for any Cartesian diagram
\begin{center}
\begin{tikzcd}
X'' \arrow[]{d}{h'} \arrow[]{r}{} & Y'' \arrow[]{d}{h} \\
X' \arrow[]{d}{g'} \arrow[]{r}{} & Y' \arrow[]{d}{g} \\
X \arrow[]{r}{} & Y,
\end{tikzcd}
\end{center}
we require that
$$h'_* \circ c_{g \circ h} = c_g \circ h_*$$
as morphisms $\Omega^\LM_*(Y'') \to \Omega^\LM_{*-d}(X')$;

\item[$(C_2)$] if $h: Z'' \to Z'$ is an lci morphism of relative dimension $d'$, then for any Cartesian diagram
\begin{center}
\begin{tikzcd}
X'' \arrow[]{d}{h''} \arrow[]{r}{} & Y'' \arrow[]{d}{h'} \arrow[]{r} & Z'' \arrow[]{d}{h} \\
X' \arrow[]{d}{g'} \arrow[]{r}{f'} & Y' \arrow[]{r}{f} \arrow[]{d}{g} &  Z' \\
X \arrow[]{r}{} & Y,
\end{tikzcd}
\end{center}
we require that
$$c_{g \circ h'} \circ h^!_f = h^!_{f \circ f'} \circ c_g$$
as morphisms $\Omega^\LM_*(Y') \to \Omega^\LM_{*-d + d'}(X'')$, where $h^!_f$ and $h^!_{f \circ f'}$ are the \emph{refined Gysin pullback morphisms} from Section 6.6 of \cite{LM}.
\end{enumerate}

The bivariant operations are defined as follows
\begin{enumerate}
\item \emph{bivariant product} is defined in the obvious way by composing homomorphisms;

\item \emph{pushforward}: if $f: X \to Y$ is projective, $Y \to Z$ is arbitrary, and $c \in \op \Omega_\LM^d(X \to Z)$, then we define $f_*(c)$ by the formula
$$f_*(c)_g := f'_* \circ c_g$$
where $f'$ comes from the Cartesian diagram
\begin{center}
\begin{tikzcd}
X' \arrow[]{d}{g'} \arrow[]{r}{f'} & Y' \arrow[]{d}{g'} \arrow[]{r}{} & Z' \arrow[]{d}{g} \\
X \arrow[]{r}{f} & Y \arrow[]{r}{} & Z;
\end{tikzcd}
\end{center}  

\item \emph{pullback}: if $g: Y' \to Y$ is arbitrary and $c \in \op \Omega_\LM^*(X \to Y)$, then we define $g^*(c)$ by the formula
$$f^*(c)_h :=  c_{g \circ h},$$
where $h: Y'' \to Y'$ is an arbitrary morphism.
\end{enumerate}
The verification that with these operations $\op\Omega_\LM^*$ becomes a bivariant theory is done in \cite{KLG}. The orientation is given by the refined Gysin pullbacks using the formula
$$\theta(f)_g := f^!_g:  \Omega_*^\LM(Y') \to \Omega_*^\LM(Y' \times_Y X).$$
This is indeed a valid operational class as $\theta(f)$ satisfies $(C_1)$ by \cite{LM} Theorem 6.6.6 2 (a) and $(C_2)$ by \emph{op. cit.} Theorem 6.6.10. 
\end{cons}

\begin{rem}
The general definition of a bivariant class in \cite{KLG} is more complicated than the one we have used in Construction \fref{OpCobCons}, but in the case of algebraic cobordism, these simplified axioms suffice. This follows from Remark 3.4 of \emph{op. cit.}, and from the simple fact that smooth pullback is a special case of refined lci pullback.
\end{rem}

Let us recall some basic properties of orientations in $\op \Omega^*_\LM$.

\begin{prop}
The orientations $\theta(f)$ are stable under Tor-independent pullbacks, and induce isomorphisms
$$- \bullet \theta(f): \op \Omega^*_\LM(Z \to X) \to \op\Omega^*_\LM(Z \to Y)$$
whenever $f: X \to Y$ is smooth (i.e., the orientations are \emph{strong} along smooth morphisms). 
\end{prop}
\begin{proof}
The first claim follows from \cite{LM} Lemma 6.6.2, and the second claim follows from \cite{An1} Proposition 2.9.
\end{proof}

Our next goal is to show that $\op \Omega^*_\LM$ is commutative. In order to do so, we need to recall some properties of algebraic cobordism proven in \cite{KLG2, KLG}. We recall that a projective morphism $X' \to X$ is an \emph{envelope} if for every subvariety $V$ of $X$ there exists a subvariety $V'$ of $X'$ that maps birationally onto $V$. An envelope is \emph{smooth} if the source $X'$ is a smooth $k$-variety. Envelopes are clearly stable under pullbacks. 

\begin{lem}[\cite{KLG} Theorem 5.3]\label{PullbackToEnvIsInjLem}
If $g: Y' \to Y$ is an envelope, then for all $X \to Y$ the bivariant pullback
$$g^*: \op \Omega^*_\LM(X \to Y) \to \op \Omega^*_\LM(Y' \times_Y X \to Y')$$
is an injection.
\end{lem}

Combined with the following observation, smooth envelopes give a powerful tool for studying operational theories.

\begin{lem}[Cf. \cite{KLG} Proposition 3.1]\label{BivariantToSmIsHomLem}
If $Y$ is smooth and of pure dimension $d$, then the morphism
$$\op \Omega^{*}_\LM(X \to Y) \to \Omega^\LM_{d-*}(X)$$
defined by evaluating at $\pi_Y^!(1_{pt}) \in \Omega^\LM_*(Y)$, is an isomorphism.
\end{lem}
\begin{proof}
Karu and Luis González show this when $Y$ is a point. The general case follows by composing the isomorphism
$$- \bullet \theta(\pi_Y): \op \Omega^*_\LM(X \to Y) \xrightarrow{\cong} \op \Omega^*_\LM(X \to pt)$$
with evaluation at $1_{pt}$.
\end{proof}

We are finally ready to prove the commutativity result.

\begin{thm}\label{OpCobIsCommThm}
The bivariant theory $\op \Omega^*_\LM$ is commutative.
\end{thm}
\begin{proof}
Given a Cartesian square
$$
\begin{tikzcd}
X' \arrow[]{d}{g'} \arrow[]{r}{f'} & Y' \arrow[]{d}{g} \\
X \arrow[]{r}{f} & Y
\end{tikzcd}
$$
and elements $\alpha \in \op\Omega^*_\LM(X \to Y)$ and $\beta \in \op\Omega^*_\LM(Y' \to Y)$, we have to show that 
$$f^*(\beta) \bullet \alpha = g^*(\alpha) \bullet \beta.$$
Using resolution of singularities, we can find a smooth envelope $\tilde Y \to Y$, and therefore we can use Lemma \ref{PullbackToEnvIsInjLem} reduce to the case when $Y$ is smooth.

Then using Lemma \ref{BivariantToSmIsHomLem} we can further reduce to the case when $\alpha = h_* (\theta(f \circ h))$ and $\beta = i_*(\theta(g \circ i))$, where everything is as in the Cartesian diagram
$$
\begin{tikzcd}
Z'' \arrow[]{r}{h''} \arrow[]{d}{i''} & X'' \arrow[]{r}{f''} \arrow[]{d}{i'} & Y'' \arrow[]{d}{i} \\
Z' \arrow[]{r}{h'} \arrow[]{d}{g''} & X' \arrow[]{r}{f'} \arrow[]{d}{g'} & Y' \arrow[]{d}{g} \\
Z \arrow[]{r}{h} & X \arrow[]{r}{f} & Y,
\end{tikzcd}
$$
with $Z$ and $Y''$ smooth and $i$ and $h$ projective. We can then compute that
\begin{align*}
f^*(\beta) \bullet \alpha &= f^*\big(i_*(\theta(g \circ i))\big) \bullet h_* (\theta(f \circ h)) \\
&= h'_* \Big( (f \circ h)^*\big(i_*(\theta(g \circ i))\big) \bullet \theta(f \circ h)\Big) & (\text{bivariant axiom $A_{123}$}) \\
&= h'_* \Big( i''_* \big((f \circ h)^*(\theta(g \circ i)) \big) \bullet \theta(f \circ h) \Big) & (\text{bivariant axiom $A_{23}$}) \\
&= (h' \circ i'')_* \big( (f \circ h)^*(\theta(g \circ i)) \bullet \theta(f \circ h) \big) & (\text{bivariant axiom $A_{12}$})
\end{align*}
and similarly that
$$g^*(\alpha) \bullet \beta = (i' \circ h'')_* \big( (g \circ i)^* (\theta(f \circ h)) \bullet \theta(g \circ i) \big).$$
As $h' \circ i'' = i' \circ h''$ the desired equality follows from
$$(f \circ h)^*(\theta(g \circ i)) \bullet \theta(f \circ h) = (g \circ i)^* (\theta(f \circ h)) \bullet \theta(g \circ i)$$
which is equivalent to the commutativity of refined Gysin pullbacks proved in \cite{LM} Theorem 6.6.10.
\end{proof}

We finish by extending $d \Omega^*_\LM$ to quasi-projective derived schemes over $k$. 

\begin{cons}
Let $k$ a field of characteristic 0. Then we define the \emph{extended operational cobordism} $d \op \Omega^*_\LM$ as a bivariant theory (without orientation) with functoriality $\Fs_k$ (see the beginning of Section \ref{FirstComparisonSect}) by the formula
$$d \op \Omega^*_\LM(X \to Y) := \op \Omega^*_\LM(\tau_0(X) \to \tau_0(Y)),$$
where $X$ and $Y$ are quasi-projective derived schemes over $k$. 
\end{cons}

\subsection{Construction of operational derived cobordism}\label{OpDCobConsSubSect}

In this section, we will construct the operational derived cobordism theory $\op d \Omega^*_A$, and prove that it is in some sense the coarsest bivariant extension of the derived bordism theory $d \Omega^A_*$ constructed in Section \ref{DerivedBivariantCobConsSubSect}.

\begin{cons}\label{OpDCobCons}
Let $A$ be a Noetherian ring of finite Krull dimension. We will construct $\op d \Omega^*_A$ as an oriented bivariant theory with functoriality $\Fs_A$ (see the beginning of Section \ref{FirstComparisonSect}). Given a homotopy class of morphisms $X \to Y$ of quasi-projective derived $A$-schemes, a \emph{degree $d$ operational derived cobordism class} $c \in \op d \Omega^d_A$ consists of a collection of morphisms
$$c_g: d\Omega^A_*(Y') \to d\Omega^A_{*-d}(Y' \times^R_Y X)$$
indexed by morphisms $g: Y' \to Y$ of derived schemes, which have to satisfy:
\begin{enumerate}
\item[$(dC_1)$] if $h: Y'' \to Y'$ is projective, then for any derived Cartesian diagram
\begin{center}
\begin{tikzcd}
X'' \arrow[]{d}{h'} \arrow[]{r}{} & Y'' \arrow[]{d}{h} \\
X' \arrow[]{d}{g'} \arrow[]{r}{} & Y' \arrow[]{d}{g} \\
X \arrow[]{r}{} & Y,
\end{tikzcd}
\end{center}
we require that
$$h'_* \circ c_{g \circ h} = c_g \circ h_*$$
as morphisms $d\Omega^A_*(Y'') \to d\Omega^A_{*-d}(X')$;

\item[$(dC_2)$] if $h: Y'' \to Y'$ is quasi-smooth of relative dimension $d'$, then for any derived Cartesian diagram
\begin{center}
\begin{tikzcd}
X'' \arrow[]{d}{h'} \arrow[]{r}{} & Y'' \arrow[]{d}{h} \\
X' \arrow[]{d}{g'} \arrow[]{r}{} & Y' \arrow[]{d}{g} \\
X \arrow[]{r}{} & Y,
\end{tikzcd}
\end{center}
we require that
$$c_{g \circ h} \circ h^! = h'^! \circ c_g$$
as morphisms $d\Omega^A_*(Y'') \to d\Omega^A_{* + d' - d}(X')$;
\end{enumerate}
The bivariant operations are defined analogously to Construction \fref{OpCobCons}, and the verification that $\op d \Omega^A_*$ is a bivariant theory is similar to the verification that $\op \Omega^*_\LM$ is a bivariant theory. The orientation $\theta$ is defined as follows: if $f: X \to Y$ is quasi-smooth, then define $\theta(f)$ by the formula
$$\theta(f)_g := f'^! : d \Omega^A_*(Y') \to d \Omega^A_*(X'),$$
where $f'$ comes from the derived Cartesian diagram
$$
\begin{tikzcd}
X' \arrow[]{r}{f'} \arrow[]{d} & Y' \arrow[]{d}{g} \\
X \arrow[]{r}{f} & Y.
\end{tikzcd}
$$
The fact that $\theta(f)$ is an operational derived cobordism class follows from functoriality of Gysin pullbacks and push-pull formula.
\end{cons}

This operational theory has the following universal property.

\begin{thm}\label{OpDCobUnivProp}
Let $\Bb^*$ be a stably oriented commutative bivariant theory with functoriality $\Fs_A$, and let $\phi$ be an assignment of isomorphisms
$$\phi_X: \Bb_*(X) \xrightarrow{\cong} d \Omega^A_*(X)$$
for each $X \in \Cc_A$ commuting with projective pushforwards and quasi-smooth pullbacks. Then $\phi$ extends in a unique way into an orientation preserving Grothendieck transformation
$$\tilde \phi: \Bb^* \to \op d \Omega^*_A.$$ 
\end{thm}
\begin{proof}
Indeed, given $\beta \in \Bb^*(X \to Y)$, we define $\tilde \phi (\beta)$ by formula
$$\tilde \phi (\beta)_g := \phi \circ \big(g^*(\beta) \bullet - \big) \circ \phi^{-1}: d \Omega^A_*(Y') \to d \Omega^A_*(X'),$$
where 
$$
\begin{tikzcd}
X' \arrow[]{d} \arrow[]{r} & Y' \arrow[]{d}{g} \\
X \arrow[]{r} & Y.
\end{tikzcd}
$$
is derived Cartesian. We have to show that $\tilde\phi(\beta)$ satisfies $(d C_1)$ and $(d C_2)$.

Given a projective morphism $h: Y'' \to Y'$, an arbitrary morphism $g: Y' \to Y$ and an element $\alpha \in d \Omega^A_*(Y'')$, we compute that 
\begin{align*}
h'_* \big(\tilde\phi(\beta)_{g \circ h}(\alpha)\big) &= h'_* \Big(\phi \big( h^*(g^*(\beta)) \bullet \phi^{-1}  (\alpha)\big)\Big) \\
&= \phi \Big(h'_* \big( h^*(g^*(\beta)) \bullet \phi^{-1}  (\alpha)\big)\Big) \\
&= \phi \Big( g^*(\beta) \bullet h_*\big(\phi^{-1}  (\alpha)\big)\Big) & (\text{bivariant axiom $A_{123}$}) \\
&= \phi \Big( g^*(\beta) \bullet \phi^{-1} (h_*(\alpha))\Big) \\
&= \tilde\phi(\beta)_g (h_*(\alpha)),
\end{align*}
where $h'$ is the pullback of $h$ along $Y' \times^R_Y X \to Y'$. Therefore $\tilde\phi(\beta)$ satisfies $(dC_1)$.

Let us then prove that $\tilde\phi(\beta)$ satisfies $(dC_2)$. Suppose we are given a derived Cartesian diagram
$$
\begin{tikzcd}
X'' \arrow[]{d}{h'} \arrow[]{r}{f''} & Y'' \arrow[]{d}{h} \\
X' \arrow[]{d}{g'} \arrow[]{r}{f'} & Y' \arrow[]{d}{g} \\
X \arrow[]{r}{f} & Y,
\end{tikzcd}
$$
with $h$ quasi-smooth and of pure relative dimension, and $\alpha \in \Bb_*(Y'')$. We can then compute that
\begin{align*}
h'^! \big(\tilde\phi(\beta)_{g}(\alpha)\big) &=  h'^! \Big(\phi\big(g^*(\beta) \bullet \phi^{-1}(\alpha) \big)\Big) \\
&= \phi\Big(h'^! \big(g^*(\beta) \bullet \phi^{-1}(\alpha) \big)\Big) \\
&= \phi\big(\theta(h') \bullet g^*(\beta) \bullet \phi^{-1}(\alpha) \big) \\
&= \phi\Big(f'^*(\theta(h)) \bullet g^*(\beta) \bullet \phi^{-1}(\alpha) \Big) & (\text{$\Bb^*$ is stably oriented}) \\
&= \phi\Big( h^*\big(g^*(\beta)\big) \bullet \theta(h) \bullet \phi^{-1}(\alpha) \Big) & (\text{$\Bb^*$ is commutative}) \\
&= \phi\Big( h^*\big(g^*(\beta)\big) \bullet h^!\big(\phi^{-1}(\alpha) \big) \Big) \\
&= \phi\Big( h^*\big(g^*(\beta)\big) \bullet \phi^{-1}\big(h^!(\alpha) \big) \Big) \\
&= \tilde\phi(\beta)_{g \circ h} \big( h^!(\alpha) \big),
\end{align*} 
so $\tilde \phi (\beta) \in \op d \Omega^*_A(X \to Y)$. 

Finally, to show that $\tilde \phi$ preserves orientations, suppose that $f: X \to Y$ is a quasi-smooth morphism, $g: Y' \to Y$ is an arbitrary morphism, and $f': X' \to Y'$ is the derived pullback of $f$ along $g$. Then, given $\alpha \in d \Omega^A_*(Y')$, we compute that
\begin{align*}
\tilde\phi(\theta(f))_g(\alpha) &= \phi\big( g^*(\theta(f)) \bullet \phi^{-1}(\alpha) \big) \\
&= \phi\big( \theta(f') \bullet \phi^{-1}(\alpha) \big) & (\text{$\Bb^*$ is stably oriented}) \\
&= \phi\big( f'^! (\phi^{-1}(\alpha)) \big) \\
&= f'^!(\alpha) \\
&= \theta^\op(f)_g(\alpha), 
\end{align*} 
where for clarity $\theta^\op$ denotes the orientation of $\op d \Omega^*_A$. Hence $\tilde\phi(\theta(f)) = \theta^\op(f)$.
\end{proof}

In particular $\op d \Omega^*_A$ admits a unique orientation preserving Grothendieck transformation from $\Omega^*_A$.

\begin{thm}\label{BivCobToOpCobThm}
There exists a unique orientation preserving Grothendieck transformation
$$\nu: \Omega^*_A \to \op d \Omega^*_A.$$
\end{thm}
\begin{proof}
This can be proven in two different ways. Since $\Omega^*_A$ is canonically equivalent to $d \Omega^*_A$ by Theorem \ref{FirstComparisonThm}, and since $d \Omega^*_A$ is a commutative stably oriented bivariant theory extending $d \Omega^*_A$ by construction, the claim follows from Theorem \ref{OpDCobUnivProp}. On the other hand, since $\op d \Omega^*_A$ is a stably oriented additive bivariant theory, and since it can be easily checked to satisfy the section, the formal group law and the snc axioms, the claim follows from Corollary \ref{UnivPropOfBivariantAlgCob}. 
\end{proof}

\begin{war}\label{OpCobIsNotCobWar}
The Grothendieck transformation $\nu$ is expected not to be an equivalence. If $k$ is a field characteristic 0, we will show that $d \op\Omega^*_\LM = \op d \Omega^*_k$, and hence the operational theory is $\Ab^1$-invariant. On the other hand, since the Grothendieck rings of vector bundles can be recovered from $\Omega^*_A$, bivariant algebraic cobordism is not $\Ab^1$-invariant, proving the inequality in characteristic 0.
\end{war}

\subsection{Proof of the comparison theorem}\label{ProofOfComparison2SubSect}

In this section we will prove Theorem \ref{SecondComparisonThm}. We start by providing $d \op \Omega^*_\LM$ with a stable orientation, which gives rise to a unique orientation preserving Grothendieck transformation $\iota: d \op \Omega^*_\LM \to \op d \Omega^*_k$ since $d \op \Omega^*_\LM$ is commutative by Theorem \ref{OpCobIsCommThm}. After that, we will show that $\iota$ is an isomorphism. Throughout this section, $k$ will be a field of characteristic 0.

Let us start by constructing the orientation on $d \op \Omega^*_\LM$.

\begin{cons}\label{OperationalOrientationCons}
Let $f: X \to Y$ be a morphism in $\Cc_k$. We define a collection of morphisms $\theta(f)$ indexed by equivalence classes of morphisms $g: Y' \to Y$ with the source $Y'$ classical using the formula
$$\theta(f)_g := f'^\vir : \Omega^\LM_*(\tau_0(Y')) \to \Omega^\LM_*(\tau_0(X')),$$
where $f': X' \to Y'$ is the derived pullback of $f$ along $g$, and $f'^\vir$ is the virtual pullback constructed in Appendix \ref{VirtualPullbackSect}.
\end{cons}

\begin{lem}
The collection $\theta(f)$ is an element of $d \op\Omega^*_\LM(X \to Y)$.
\end{lem}
\begin{proof}
Since there is a canonical bijection between maps $Y' \to Y$ and $Y' \to \tau_0(Y)$ when $Y'$ is classical, we see that at least the claim makes sense. To conclude, we have to show that $\theta(f)$ satisfies $(C_1)$ and $(C_2)$.

To prove that $\theta(f)$ satisfies $(C_1)$, we notice that given a projective morphism $h: Y'' \to Y'$ with $Y''$ and $Y'$ classical and a derived Cartesian diagram
\begin{center}
\begin{tikzcd}
X'' \arrow[]{d}{h'} \arrow[]{r}{f''} & Y'' \arrow[]{d}{h} \\
X' \arrow[]{d}{g'} \arrow[]{r}{f'} & Y' \arrow[]{d}{g} \\
X \arrow[]{r}{f} & Y,
\end{tikzcd}
\end{center}
then, by Theorem \ref{VirtPullbackPropertiesThm}, we have that
$$f'^\vir \circ h_* = \tau_0(h')_* \circ f''^\vir$$
proving that $(C_1)$ holds.

To prove that $\theta(f)$ satisfies $(C_2)$ consider a derived Cartesian diagram
$$
\begin{tikzcd}
X'' \arrow[]{d}{h''} \arrow[]{r}{f''} & Y'' \arrow[]{d}{h'} \arrow[]{r}{i''} & Z'' \arrow[]{d}{h} \\
X' \arrow[]{d}{g'} \arrow[]{r}{f'} & Y' \arrow[]{r}{i'} \arrow[]{d}{g} &  Z' \\
X \arrow[]{r}{f} & Y,
\end{tikzcd}
$$
with $Y'$, $Z'$ and $Z''$ classical and $h$ lci. Then by Theorem \ref{VirtPullbackPropertiesThm} we have that $h'^\vir = h^!_{i'}$ and $h''^\vir = h^!_{i' \circ f'}$, so the claim follows from functoriality of virtual pullbacks, which is another part of Theorem \ref{VirtPullbackPropertiesThm}.
\end{proof}

It is then easy to see that $\theta$ is an orientation.

\begin{thm}\label{OpOrientationThm}
The elements 
$$\theta(f) \in d\op\Omega^*_\LM(X \to Y)$$
give a stable orientation for $d \op \Omega^*(X \to Y)$ along quasi-smooth morphisms.
\end{thm}
\begin{proof}
All claims follow immediately from Theorem \ref{VirtPullbackPropertiesThm}.
\end{proof}

It is clear that the Gysin pullbacks induced by the above orientation on the associated homology theory $d \Omega_*^\LM$ of $d \op \Omega^*_\LM$ are exactly the virtual pullbacks. Recall that Lowrey and Schürg construct (see \cite{LS} Corollary 4.2) morphisms
$$\iota_{h,X}: \Omega^\LM_*(\tau_0(X)) \to d\Omega^k_*(X)$$
which are manifestly compatible with projective pushforwards. The main results of their article are that the $\iota_{h,X}$ are isomorphisms which identify the Gysin pullbacks on $d \Omega^k_*$ with the virtual pullbacks on $d \Omega^\LM_*$ (see \cite{LS} Corollary 5.11 and Theorem 5.12). Applying Theorem \ref{OpDCobUnivProp}, we immediately obtain the following result.

\begin{prop}
There exists a unique orientation preserving Grothendieck transformation 
$$\iota: d \op \Omega^*_\LM \to \op d \Omega^*_k$$
extending the above isomorphism $\iota_h: d \Omega^*_\LM \to  d\Omega^k_*$ of homology theories. \qed
\end{prop}

We are finally ready to prove the main result of this section.

\begin{proof}[Proof of Theorem \ref{SecondComparisonThm}]
We will prove that $\iota$ is an isomorphism by constructing its inverse transformation. We define $\iota': \op d \Omega^*_k \to d \op \Omega^*_\LM$ as follows: if $\alpha \in \op d \Omega^*_k(X \to Y)$, then
$$\iota'(\alpha)_g := \iota_h^{-1} \circ \alpha_g \circ \iota_h: \Omega^*_\LM(\tau_0(Y')) \to \Omega^*_\LM(\tau_0(X \times^R_Y Y')),$$
where $g: Y' \to Y$. As $\iota_h$ commutes with pushforwards and $\alpha$ satisfies $(dC_1)$, $\iota'(\alpha)$ satisfies $(C_1)$. As $\iota_h$ commutes with Gysin pullbacks, $\alpha$ satisfies $(dC_2)$, and refined lci pullbacks are a special case of quasi-smooth pullbacks by Theorem \ref{VirtPullbackPropertiesThm}, $\iota'(\alpha)$ satisfies $(C_2)$, showing that $\iota'$ is well defined. It is also clear that $\iota'$ is the inverse of $\iota$, since by definition
$$\iota(\beta)_g = \iota_h \circ \beta_g \circ \iota_h^{-1}$$
and therefore $\iota: d \op \Omega^*_\LM \to \op d \Omega^*_k$ is an isomorphism.
\end{proof}

\section{A note on operational $K$-theory}\label{OpKThySect}

The purpose of this section is to prove Theorem \ref{ThirdComparisonThm}, which can be thought as a $K$-theoretic version of Theorem \ref{SecondComparisonThm}. Throughout the section $k$ will denote the ground field, and given a derived scheme $X$, $G_0(X)$ will denote its $G$-theory group, i.e., the Grothendieck group of coherent sheaves on $X$. Recall that given a morphism $f: X \to Y$ of derived schemes with finite Tor-dimension (e.g. $f$ is flat or quasi-smooth), then there exists a Gysin pullback morphism $f^!: G_0(Y) \to G_0(X)$ given by the derived pullback of quasi-coherent sheaves. These pullbacks are contravariantly functorial and satisfy push-pull formula for derived Cartesian squares 
$$
\begin{tikzcd}
X' \arrow[]{d} \arrow[]{r} & Y' \arrow[]{d}{g} \\
X \arrow[]{r}{f} & Y
\end{tikzcd}
$$ 
with $g$ proper and $f$ of finite Tor-dimension.
  
We will start by recalling the relevant bivariant theories. Throughout this section, we will denote by $\Fc'^\cl_{k}$ the functoriality $(\Cc'^\cl_k, \Cs'^\cl_k, \Is'^\cl_k, \Ss'^\cl_k)$ where $\Cc'^\cl_k$ is the category of finite type separated $k$-schemes, $\Cs'^\cl_k$ consists of proper morphisms, $\Is'^\cl_k$ consists of all Cartesian squares and $\Ss'^\cl_k$ consists of regular embeddings. Similarly $\Fc'_k = (\Cc'_k, \Cs'_k, \Is'_k, \Ss'_k)$, where $\Cc'_k$ is the $\infty$-category of finite type separated derived $k$-schemes, $\Cs'_k$ are the proper morphisms, $\Is'_k$ consists of all derived Cartesian squares and $\Ss'_k$ consists of all derived regular embeddings. 

Let us first recall the construction of operational $K$-theory from \cite{AP}.

\begin{cons}\label{OpKThyCons}
We will construct operational $K$-theory as a bivariant theory with functoriality $\Fs'^\cl_k$. An \emph{operational $K$-theory class} $c \in \op K^0(X \to Y)$ consists of homomorphisms
$$c_g: G_0(Y') \to G_0(Y' \times_Y X)$$
indexed by morphisms $g: Y' \to Y$ and satisfying:
\begin{enumerate}
\item[$(B_1)$] given a Cartesian diagram
$$
\begin{tikzcd}
X'' \arrow[]{r} \arrow[]{d}{h'} & Y'' \arrow[]{d}{h} \\
X' \arrow[]{r} \arrow[]{d}{g'} & Y' \arrow[]{d}{g} \\
X \arrow[]{r} & Y
\end{tikzcd}
$$
with $h$ proper, then 
$$h'_* \circ c_{g \circ h} = c_g \circ h_*: G_0(Y'') \to G_0(X');$$
\item[$(B_2)$] given a Cartesian diagram as above but $h$ is flat instead, then  
$$h'^! \circ c_g = c_{g \circ h} \circ h^!: G_0(Y') \to G_0(X'');$$

\item[$(B_3)$] given a Cartesian diagram 
$$
\begin{tikzcd}
X'' \arrow[]{r}{f''} \arrow[hook]{d}{i''} & Y'' \arrow[]{r}{h''} \arrow[hook]{d}{i'} & Z'' \arrow[hook]{d}{i} \\
X' \arrow[]{r}{f'} \arrow[]{d}{g''} & Y' \arrow[]{r}{h'} \arrow[]{d}{g'} & Z' \\
X \arrow[]{r}{f} & Y 
\end{tikzcd}
$$
with $i$ a regular embedding, then
$$i^!_{h' \circ f'} \circ c_{g'} = c_{g' \circ i'} \circ i^!_{h'}: G_0(Y') \to G_0(X'')$$
where $i^!_{h' \circ f'}$ and $i^!_{h'}$ are the refined Gysin pullbacks defined in \cite{AP} Section 3.4.
\end{enumerate}
\end{cons}

Secondly, we recall the construction of operational derived $K$-theory from Appendix B of \cite{AGP}.

\begin{cons}
We will construct operational derived $K$-theory as a bivariant theory with functoriality $\Fs'_k$. An \emph{operational derived $K$-theory class} $c \in \op K^{\der}(X \to Y)$ consists of homomorphisms
$$c_g: G_0(Y') \to G_0(Y' \times^R_Y X)$$
indexed by equivalence classes of morphisms $g: Y' \to Y$ and satisfying:
\begin{enumerate}
\item[$(dB_1)$] given a derived Cartesian diagram
$$
\begin{tikzcd}
X'' \arrow[]{r} \arrow[]{d}{h'} & Y'' \arrow[]{d}{h} \\
X' \arrow[]{r} \arrow[]{d}{g'} & Y' \arrow[]{d}{g} \\
X \arrow[]{r} & Y
\end{tikzcd}
$$
with $h$ proper, then 
$$h'_* \circ c_{g \circ h} = c_g \circ h_*: G_0(Y'') \to G_0(X');$$
\item[$(dB_2)$] given a derived Cartesian diagram as above but $h$ is flat instead, then  
$$h'^! \circ c_g = c_{g \circ h} \circ h^!: G_0(Y') \to G_0(X'');$$

\item[$(dB_3)$] given a derived Cartesian diagram as above but $h$ is a derived regular embedding instead, then
$$h'^! \circ c_g = c_{g \circ h} \circ h^!: G_0(Y') \to G_0(X'');$$
\end{enumerate}
\end{cons} 

It turns out, that in order to achieve the main purpose of this section --- showing that the bivariant groups $\op K^\der(X \to Y)$ are canonically isomorphic to $\op K^0(\tau_0(X) \to \tau_0(Y))$ --- we have to prove that classical operational $K$-theory classes are compatible with \emph{virtual pullbacks}, which we define below. From now on, given $X \in \Cc'_k$, we will denote by $i_X$ the canonical inclusion of the truncation $\tau_0(X) \hookrightarrow X$.

\begin{defn}\label{GThyVirtPullbackDef}
Let $f: X \to Y$ be a morphism of finite Tor-dimension of derived schemes. Then we define the \emph{virtual pullback} as
$$f^\vir := i_{X*}^{-1} \circ f^! \circ i_{Y*}: G_0(\tau_0(Y)) \to G_0(\tau_0(X)).$$
It is clear that $f^\vir$ are contravariantly functorial.
\end{defn}  

Let us then recall the Grothendieck transformation constructed by Vezzosi in \cite{AGP} Proposition B.7.

\begin{cons}
We denote by $d \op K^0$ the bivariant theory on $\Fs'_k$ defined by
$$d \op K^0(X \to Y) := \op K^0(\tau_0(X) \to \tau_0(Y)).$$
There exists an injective Grothendieck transformation
$$\alpha: \op K^\der \to d \op K^0$$
constructed as follows: given $c \in \op K^\der(X \to Y)$, define $\alpha(c) \in \op K^0(\tau_0(X) \to \tau_0(Y))$ by
$$\alpha(c)_g := i_{X'*}^{-1} \circ c_{i_Y \circ g}$$
where $g: Y' \to \tau_0(Y)$ is a classical morphism and $X' \to Y'$ is the derived pullback of $X \to Y$ along $i_Y \circ g$. This is well defined by \cite{AGP} Proposition B.7.
\end{cons}

Since $\alpha$ is injective, to achieve our goal, we need only to show that it is surjective. The following result, giving an easily checked criterion for operational classes to be in the image of $\alpha$, is an important part of the proof.

\begin{lem}\label{ImOfAlphaLem}
Let $X \to Y$ be a morphism of derived schemes in $\Cc'_k$. Then $c \in \op K^0(\tau_0(X) \to \tau_0(Y))$ is in the image of 
$$\alpha: \op K^\der(X \to Y) \to \op K^0(\tau_0(X) \to \tau_0(Y))$$
if and only if it is compatible with virtual pullbacks along derived regular embeddings in the following sense: for all derived Cartesian diagrams
$$
\begin{tikzcd}
X'' \arrow[]{r} \arrow[hook]{d}{h'} & Y'' \arrow[hook]{d}{h} \\ 
X' \arrow[]{r} \arrow[]{d}{g'} & Y' \arrow[]{d}{g} \\
X \arrow[]{r} & Y,
\end{tikzcd}
$$ 
with $h$ a derived regular embedding, $c$ should satisfy
$$h'^\vir \circ c_{\tau_0(g)} =  c_{\tau_0(g \circ h)} \circ h^\vir.$$
\end{lem}
\begin{proof}
Given any $c \in \op K^0(\tau_0(X) \to \tau_0(Y))$ and a morphism of derived schemes $g: Y' \to Y$, we define 
$$\tilde{c}_g := i_{X' *} \circ c_{\tau_0(g)} \circ i_{Y'*}^{-1}: G_0(Y') \to G_0(X')$$
where $X' \to Y'$ is the derived pullback of $X \to Y$ along $g$. We claim that $\tilde c$ satisfies $(d B_1)$ and $(d B_2)$, and that $\tilde c$ satisfies $(d B_3)$ if and only if $c$ is compatible with virtual pullbacks along derived regular embeddings. This proves the claim as clearly $\alpha(\tilde c) = c$. 

For the rest of the proof, we consider the following derived Cartesian diagram
$$
\begin{tikzcd}
X'' \arrow[]{r} \arrow[]{d}{h'} & Y'' \arrow[]{d}{h} \\
X' \arrow[]{r} \arrow[]{d}{g'} & Y' \arrow[]{d}{g} \\
X \arrow[]{r} & Y.
\end{tikzcd}
$$

To prove that $\tilde c$ satisfies $(dB_1)$, we consider the case when $h$ is proper and compute that
\begin{align*}
\tilde c_{g} \circ h_* &= i_{X'*} \circ c_{\tau_0(g)} \circ i_{Y'*}^{-1}  \circ h_* \\
&= i_{X'*} \circ c_{\tau_0(g)} \circ \tau_0(h)_* \circ i_{Y''*}^{-1} & (\text{pushforwards are functorial}) \\
&= i_{X'*} \circ \tau_0(h')_* \circ  c_{\tau_0(g \circ h)}  \circ i_{Y''*}^{-1} & (B_1) \\
&= h'_* \circ i_{X''*} \circ c_{\tau_0(g \circ h)}  \circ i_{Y''*}^{-1} & (\text{pushforwards are functorial}) \\
&=  h'_* \circ \tilde c_{g \circ h}
\end{align*}
which is exactly what we wanted.

To prove that $\tilde c$ satisfies $(dB_2)$, we consider the case when $h$ is flat. Notice that for a flat morphism $f: A \to B$ of derived schemes, the square
$$
\begin{tikzcd}
\tau_0(A) \arrow[]{d}{\tau_0(f)} \arrow[]{r}{i_A} & A \arrow[]{d}{f} \\
\tau_0(B) \arrow[]{r}{i_B} & B 
\end{tikzcd}
$$
is derived Cartesian, and so 
$$f^! \circ i_{B*} = i_{A*} \circ \tau_0(f)^!$$
by the push-pull formula. Using this, we compute that
\begin{align*}
\tilde c_{g \circ h} \circ h^! &= i_{X''*} \circ c_{\tau_0(g \circ h)}  \circ i_{Y''*}^{-1} \circ h^! \\
&= i_{X''*} \circ c_{\tau_0(g \circ h)} \circ \tau_0(h)^! \circ i_{Y'*}^{-1} \\
&= i_{X''*} \circ \tau_0(h')^! \circ c_{\tau_0(g)} \circ i_{Y'*}^{-1} & (B_2) \\
&= h'^! \circ i_{X'*} \circ c_{\tau_0(g)} \circ i_{Y'*}^{-1} \\
&= h'^! \circ \tilde c_g,
\end{align*}
so $\tilde c$ satisfies $(dB_2)$.

We are left to show that $\tilde c$ satisfies $(d B_3)$ if and only if $c$ is compatible with virtual pullbacks along derived regular embeddings. Assume that $h$ is a derived regular embedding and compute that
\begin{align*}
\tilde c_{g \circ h} \circ h^! &= i_{X''*} \circ c_{\tau_0(g \circ h)}  \circ i_{Y''*}^{-1} \circ h^! \\
&= i_{X''*} \circ c_{\tau_0(g \circ h)} \circ h^\vir \circ i_{Y'_*}^{-1} & (\text{Definition \ref{GThyVirtPullbackDef}}) \\
&= i_{X''*} \circ h'^\vir \circ c_{\tau_0(g)} \circ i_{Y'_*}^{-1} \\
&= h'^! \circ i_{X'*}^{-1} \circ c_{\tau_0(g)} \circ i_{Y'_*}^{-1} & (\text{Definition \ref{GThyVirtPullbackDef}}) \\ 
&= h'^! \circ \tilde c_{g},
\end{align*} 
and therefore $\tilde c$ satisfies $(dB_3)$. Similarly, one shows that $\tilde c$ satisfying $(dB_3)$ implies that $c$ is compatible with virtual pullbacks, so we are done. 
\end{proof}

Hence, we are done if we can show that all operational classes $c \in \op K^0(X \to Y)$ are compatible with virtual pullbacks in the above sense. To do so, we begin with the following observation. Suppose $j: Z \hookrightarrow X$ is a quasi-smooth closed immersion. We can then use  \cite{Khan} Theorem 4.1.13 to form the derived Cartesian diagram
$$
\begin{tikzcd}
X & \arrow[]{l}[swap]{\mathrm{pr}_2} \Gb_m \times X \arrow[]{d} \arrow[hook]{r}{u} & M(Z/X) \arrow[]{d}{\pi} & \arrow[hook']{l}[swap]{\iota} \Nc_{Z/X} \arrow[]{d} \arrow[]{r}{p} & Z \\
& \Gb_m \arrow[hook]{r} & \Ab^1 & \arrow[hook']{l}[swap]{\iota^0} \{ 0 \}, &
\end{tikzcd}
$$
where $M(Z/X)$ is the derived deformation to the normal bundle, and $p: \Nc_{Z/X} \to Z$ is the vector bundle projection. We then make the following claim.

\begin{lem}\label{FormulaForQSmPullbackLem}
Let everything be as above. Then $\iota^! \circ (u^!)^{-1}$ is well defined and 
$$ f^!  =  (p^!)^{-1} \circ \big( \iota^! \circ (u^!)^{-1} \big) \circ \mathrm{pr}_2^!.$$
\end{lem}
\begin{proof}
Since $\iota^! \circ \iota_*$ is identically zero, and since 
$$G_0(\Nc_{Z/X}) \xrightarrow{\iota_*} G_0(M(Z/X)) \xrightarrow{u^!}  G_0(\Gb_m \times X) \to 0$$
is exact, we see that $\iota^! \circ (u^!)^{-1}$ is well defined. To prove the second claim, let us denote by 
$$\tilde f: \Ab^1 \times Z \hookrightarrow M(Z/X)$$
the strict transform of $\mathrm{Id}_{\Ab^1} \times j$, and consider the derived Cartesian diagram
$$
\begin{tikzcd}
Z \arrow[hook]{r}{s} \arrow[hook]{d}{i_0} & \Nc_{Z/X} \arrow[hook]{d}{\iota} \\
\Ab^1 \times Z \arrow[hook]{r}{\tilde f} & M(Z/X)\\
Z \arrow[hook]{r}{f} \arrow[hook']{u}[swap]{i_1} & X \arrow[hook']{u}[swap]{\iota'}, \\
\end{tikzcd}
$$
where $\iota'$ is the inclusion of the fibre of $\pi$ over $1$, and $s$ is the zero section. Given $a \in G_0(X)$, we can choose $\tilde a \in G_0(M(Z/X))$ so that $u^!(\tilde a) = \mathrm{pr}_2^!(a)$, and then compute that
\begin{align*}
(p^!)^{-1} \circ \big( \iota^! \circ (u^!)^{-1} \big) \circ \mathrm{pr}_2^!(a) &= (p^!)^{-1} \big( \iota^!(\tilde a) \big) \\
&= s^! \big(\iota^!(\tilde a) \big) \\
&= i_0^! \big(\tilde{f}^!(\tilde a) \big)  \\
&= i_1^! \big(\tilde{f}^!(\tilde a) \big)  \\
&= f^! \big(\iota'^!(\tilde a) \big) \\
&= f^!(a),
\end{align*}
which is exactly what we wanted.
\end{proof}

Therefore
$$f^\vir = (\tau_0(p)^!)^{-1} \circ \big( \iota^\vir \circ (\tau_0(u)^!)^{-1} \big) \circ \tau_0(\mathrm{pr})^!_2.$$
In other words, we have found a formula for the virtual pullback in terms of flat pullbacks and a single very simple virtual pullback $\iota^\vir$, which can actually be expressed as a refined Gysin pullback as is shown by the next result. 

\begin{lem}\label{FormulaForIotaLem}
Let everything be as above. Then 
$$\iota^\vir = \iota^{0!}_{\tau_0(\pi)}.$$
\end{lem}
\begin{proof}
To fix the notation, let  
$$
\begin{tikzcd}
W \arrow[]{d} \arrow[hook]{r}{\iota'} & \tau_0(M(Z/X)) \arrow[]{d}{\tau_0(\pi)} \\
\{ 0 \} \arrow[hook]{r}{\iota^0} & \Ab^1
\end{tikzcd}
$$
be derived Cartesian, and let $\Fc$ be a coherent sheaf on $\tau_0(M(Z/X))$. Then, by definition $\iota^{0!}_{\tau_0(\pi)}[\Fc] = [C] - [K]$, where
$$0 \to K \to \Fc \xrightarrow{\tau_0(\pi) \cdot} \Fc \to C \to 0$$
is an exact sequence. On the other hand, $i'^![\Fc]$ is just the cofibre of $\tau_0(\pi) \cdot$, regarded in a natural way as a quasi-coherent sheaf on the derived vanishing locus $W$ of $\tau_0(\pi)$, and therefore
$$\iota'^! = i_{W*} \circ \iota^{0!}_{\tau_0(\pi)}.$$
Since clearly $i_{W*}^{-1} \circ \iota'^! = \iota^{vir}$, we are done.
\end{proof}

\begin{proof}[Proof of Theorem \ref{ThirdComparisonThm}]
We need to show that $\alpha$ is surjective. By Lemma \ref{ImOfAlphaLem}, it is enough to show the following: if $c \in \op K^0(X \to Y)$ is an operational class ($X$ and $Y$ are classical), then $c$ is compatible with virtual pullbacks along derived regular embeddings. To prove the claim, consider a derived Cartesian diagram
$$
\begin{tikzcd}
X'' \arrow[]{r}{f''} \arrow[hook]{d}{h'} & Y'' \arrow[hook]{d}{h} \\
X' \arrow[]{r}{f'} \arrow[]{d}{g'} & Y' \arrow[]{d}{g} \\
X \arrow[]{r}{f} & Y
\end{tikzcd}
$$
with $h$ a derived regular embedding, and to fix notation, form the Cartesian diagram
$$
\begin{tikzcd}
\tau_0(X') \arrow[]{d}{\tau_0(f')} & \arrow[]{l}[swap]{\mathrm{pr}_2} \Gb_m \times \tau_0(X') \arrow[]{d} \arrow[hook]{r}{u'_0} & \tau_0(M(X''/X')) \arrow[]{d}{} & \arrow[hook']{l}[swap]{\iota'_0} \tau_0(\Nc_{X''/X'}) \arrow[]{d} \arrow[]{r}{p'_0} & \tau_0(X'') \arrow[]{d}{\tau_0(f'')} \\
\tau_0(Y') & \arrow[]{l}[swap]{\mathrm{pr}_2} \Gb_m \times \tau_0(Y') \arrow[]{d} \arrow[hook]{r}{u_0} & \tau_0(M(Y''/Y')) \arrow[]{d}{\tau_0(\pi)} & \arrow[hook']{l}[swap]{\iota_0} \tau_0(\Nc_{Y''/Y'}) \arrow[]{d} \arrow[]{r}{p_0} & \tau_0(Y'') \\
& \Gb_m \arrow[hook]{r} & \Ab^1 & \arrow[hook']{l}[swap]{\iota^0} \{ 0 \}, &
\end{tikzcd}
$$
and denote by $\tau_0(\pi')$ the composition $\tau_0(M(X''/X')) \to \Ab^1$. Then, combining Lemmas \ref{FormulaForQSmPullbackLem} and \ref{FormulaForIotaLem}, we compute that
\begin{align*}
h'^\vir \circ c_{\tau_0(g)} &= (p_0'^!)^{-1} \circ \big( \iota^{0!}_{\tau_0(\pi')} \circ (u'^!_0)^{-1} \big) \circ \mathrm{pr}^!_2 \circ c_{\tau_0(g)} \\
&= c_{\tau_0(g \circ h)} \circ (p_0^!)^{-1} \circ \big( \iota^{0!}_{\tau_0(\pi)} \circ (u_0^!)^{-1} \big) \circ \mathrm{pr}^!_2  & (\text{$(B_2)$ and $(B_3)$}) \\
&= c_{\tau_0(g \circ h)} \circ h^\vir,
\end{align*}
proving that $c$ satisfies the condition of Lemma \ref{ImOfAlphaLem}. Since $c$ was arbitrary, we conclude that $\alpha$ is surjective.
\end{proof}

\appendix

\section{Hurewicz morphism}\label{HurewiczSect}

The purpose of this section is to study the naturality properties of the Hurewicz morphism, which gives an easily understandable approximation to the relative cotangent complex. The main result of this section is Corollary \fref{ConormalVsIdealCor}, which shows, roughly speaking, that the (zeroth homotopy of the) conormal sheaf of a closed embedding can be naturally identified with the (zeroth homotopy of the) restriction of the sheaf of ``derived ideals'' on the closed subscheme. This result will play an important role in Section \fref{DerivedBlowUpSect}, because it gives an easily verifiable criterion for a commutative square of derived schemes to be a relative virtual Cartier divisor (see Definition \fref{DivisorOverZDefn}). 

Throughout this section, we will assume familiarity with Sections 7.3 and 7.4 of \cite{Lur2}, as well as Chapter 25 of \cite{Lur3}. Given a simplicial commutative ring $A$, we will denote by $A^\circ$ its underlying $\Eb_\infty$-ring spectrum. Recall that the stable $\infty$-categories of modules over $A$ and $A^\circ$ are canonically equivalent. We will denote the cotangent complex of a map $A \to B$ of $\Eb_\infty$-ring spectra by $\Lb^\topo_{B/A}$, and reserve the notation $\Lb_{A/B}$ for the cotangent complex of simiplicial commutative rings. 

Let us fix some notation. We will denote by $\Cc_\scr$ and $\Cc_{\Eb_{\infty}}$ the $\infty$-categories of simplicial commutative rings and $\Eb_\infty$-ring spectra respectively. If $\Cc$ is either of those two, then
$$e: \Tb_\Cc \to \Fun(\Delta^1, \Cc)$$
is the \emph{tangent bundle}, where $\Tb_\Cc$ is roughly speaking the $\infty$-category of pairs $(A,M)$ with $A \in \Cc$ and $M$ is an $A$-module, and where $e$ sends $(A, M)$ to the trivial square zero extension $A \oplus M \to A$. Note that the fibre of 
$$p := \Tb_\Cc \xrightarrow{e} \Fun(\Delta^1, \Cc) \to \Fun(\{1\}, \Cc) = \Cc$$
over $A \in \Cc$ is the stable $\infty$-category of $A$-modules. We will denote by 
$$\pi: \Mc^{\Tb}(\Cc) \to \Delta^1 \times \Cc$$
the \emph{tangent correspondence}, where $\Mc^\Tb(\Cc)$ consists of a copy of $\Cc$ over $\{0\} \times \Cc$ and $p: \Tb_\Cc \to \Cc$ over $\{1\} \times \Cc$ in such a way that the functor $\Cc \to \Tb_{\Cc}$ classified by the coCartesian fibration
$$q:= \Mc^{\Tb}(\Cc) \xrightarrow{\pi} \Delta^1 \times \Cc \xrightarrow{\mathrm{pr}_1} \Delta^1$$
coincides with the cotangent complex functor. Note that the arrows in $\Mc^{\Tb}(\Cc)$ lying over $\Delta^1 \times \{A\} \subset \Delta^1 \times \Cc$ are exactly \emph{derivations} of $A$ taking value in some $A$-module. We note that $p$ and $q$ are both Cartesian and coCartesian fibrations.

We start by constructing a functorial version of the topological Hurewicz morphism, and then use it to obtain the functorial Hurewicz of simplicial commutative rings. The functorial topological Hurewicz morphism is an $\infty$-functor $\Fun(\Delta^1, \Cc_{\Eb_\infty}) \to \Fun(\Delta^1, \Tb_{\Cc_{\Eb_\infty}})$ which sends a morphism $\psi: A \to B$ in $\Cc_{\Eb_\infty}$ to a natural morphism $\epsilon^\topo_\psi: B \otimes_A \Cofib(\psi) \to \Lb^{\topo}_{B/A}$ lying over $B$.

\begin{cons}[Functorial topological Hurewicz morphism]\label{FunctorialTopoHurewiczCons}
We start by applying \cite{Lur1} Proposition 4.3.2.15 several times in order to construct an $\infty$-functor 
$$\Fun(\Delta^1, \Cc_{\Eb_\infty}) \to \Fun_{\Delta^1}(\Delta^1 \times \Delta^1/_{\mathrm{pr}_2} \Delta^1, \Mc^\Tb(\Cc_{\Eb_\infty})).$$  
As $\Cc_{\Eb_\infty}$ can be identified as the fibre of $q$ over $0$, we can apply \emph{loc. cit.} to the coCartesian fibration $q$ to obtain an $\infty$-functor that can be described on objects as
\begin{equation*}
\begin{tikzcd}
A \arrow[]{d}{\psi} \\
B 
\end{tikzcd} 
\mapsto
\begin{tikzcd}
A \arrow[]{r}{d_A} \arrow[]{d}{\psi} & \Lb^\topo_A \arrow[]{d} \\
B \arrow[]{r}{d_B} & \Lb^\topo_B,
\end{tikzcd} 
\end{equation*}
where $d_A$ and $d_B$ are the universal derivations of $A$ and $B$ respectively. Applying multiple times \emph{loc. cit.} itself, and once its dual (not in this order), to $p$ (which is the the fibre of $\pi$ over $\{1\} \times \Cc$), and composing with the above $\infty$-functor, we obtain an $\infty$-functor that is morally described by
\begin{equation*}
\begin{tikzcd}
A \arrow[]{d}{\psi} \\
B 
\end{tikzcd} 
\mapsto
\begin{tikzcd}
A \arrow[]{r}{d_A} \arrow[]{d}{\psi} & \Lb^\topo_A \arrow[]{d} \arrow[]{r}{} & 0_A \arrow[]{r}{} \arrow[]{d} & 0_B \arrow[]{ld} \\
B \arrow[]{r}{d_B} & \Lb^\topo_B \arrow[]{r} & \Lb^\topo_{B/A}
\end{tikzcd} 
\end{equation*}
and applying the obvious forgetful functor sends the right hand side to the square
$$
\begin{tikzcd}
A \arrow[]{r} \arrow[]{d}{\psi} & 0_B \arrow[]{d} \\
B \arrow[]{r}{d_{B/A}} & \Lb^\topo_{B/A}
\end{tikzcd}
$$
in $\Mc^\Tb(\Cc_{\Eb_\infty})$. This concludes the first step of the construction of $\epsilon^\topo_{B/A}$.

Let us finish the construction. Using the fact that the morphism $\Tb_{\Cc_{\Eb_\infty}} \to \Cc_{\Eb_\infty}$ classified by $q$ considered as a Cartesian fibration coincides with the composition 
$$\Tb_\Cc \xrightarrow{e} \Fun(\Delta^1, \Cc) \to \Fun(\{0\}, \Cc) = \Cc,$$
we can apply \cite{Lur1} Proposition 4.3.2.15 as above and the forgetful morphism $\Mc^\Tb(\Cc_{\Eb_\infty}) \to \Cc_{\Eb_\infty}$ to send the above square $\infty$-categorically to the square
$$
\begin{tikzcd}
A \arrow[]{r}{\psi} \arrow[]{d}{\psi} & B \arrow[]{d}{\eta_0} \\
B \arrow[]{r}{\eta_{B/A}} & B \oplus \Lb^\topo_{B/A}
\end{tikzcd}
$$
in $\Cc_{\Eb_\infty}$, where $\eta_{B/A}$ and $\eta_0$ are the morphisms $B \to B\oplus \Lb^\topo_{B/A}$ induced by the universal and the trivial $A$-derivations. Composing with the section of $p: \Tb_{\Cc_{\Eb_\infty}} \to \Cc_{\Eb_\infty}$ sending $C$ to $C$ considered as a $C$-module, and applying \cite{Lur1} Proposition 4.3.2.15 as above, we can send the above square $\infty$-functorially to
$$
\begin{tikzcd}
A \arrow[]{r}{\psi} \arrow[]{d}{\psi} & B \arrow[]{d}{\eta_0} \\
B \arrow[]{r}{\eta_{B/A}} & B \oplus \Lb^\topo_{B/A}
\end{tikzcd}
$$
where the objects on the left hand side lie over $A$ (i.e., are considered as $A$-modules), the objects on the right hand side lie over $B$, and the horizontal arrows lie over $\psi$. Taking vertical cofibres, and using one last time \cite{Lur1} Proposition 4.3.2.15, we obtain a morphism
$$B \otimes_A \Cofib(\psi) \to \Cofib(\eta_0).$$
functorial in $\psi$. Composing with the natural identification $\Cofib(\eta_0) \simeq \Lb^\topo_{B/A}$ of $B$-modules, we obtain the desired $\infty$-functor $\psi \mapsto \epsilon^\topo_\psi$.
\end{cons}

Next we prove that this coincides with the Hurewicz morphism constructed by Lurie.

\begin{lem}
Let $\psi: A \to B$ be a morphism in $\Cc_{\Eb_\infty}$. Then the morphism
$$\epsilon^\topo_\psi: B \otimes_A \Cofib(\psi) \to \Lb^\topo_{B/A}$$
constructed above is equivalent to the morphism of \cite{Lur2} Theorem 7.4.3.1.
\end{lem}
\begin{proof}
Let us denote by $\epsilon'_\psi$ the morphism defined in \cite{Lur2}. Recall that it is constructed using the pullback diagram
$$
\begin{tikzcd}
B^{\eta_{B/A}} \arrow[]{d}{\psi''} \arrow[]{r}{\psi''} & B \arrow[]{d}{\eta_0} \\
B \arrow[]{r}{\eta_{B/A}} & B \oplus \Lb^\topo_{B/A}
\end{tikzcd}
$$
of $B$-algebras, inducing a commutative triangle 
\begin{center}
\begin{tikzcd}
A \arrow[]{r}{\psi'} \arrow[]{rd}{\psi} & B^{\eta_{B/A}} \arrow[]{d}{\psi''} \\
  & B
\end{tikzcd}
\end{center}
which in turn induces $\epsilon'_\psi$ as the morphism $B \otimes_A \Cofib(\psi) \to \Cofib(\psi'')$. This is clearly equivalent to our definition, so the claim follows.
\end{proof}

Next, we use the functorial topological Hurewicz morphism to obtain the functorial Hurewicz morphism for simplicial commutative rings.

\begin{cons}[Functorial Hurewicz morphism]\label{FunctorialHurewiczCons}
Using similar arguments as above, one shows that the comparison morphism $\rho_{B/A}: \Lb^\topo_{B^\circ / A^\circ} \to \Lb_{B/A}$ from \cite{Lur3} Section 25.3.5 can be described as an $\infty$-functor $\Fun(\Delta^1, \Cc_\scr) \to \Fun(\Delta^1, \Tb_{\Cc_\scr})$. On the other hand, Construction \ref{FunctorialTopoHurewiczCons} gives an $\infty$-functor sending $\psi: A \to B$ to $\epsilon^\topo_{\psi^\circ}: B^\circ \otimes_{A^\circ} \Cofib(\psi^\circ) \to \Lb^\topo_{B^\circ/A^\circ}$. Taking the composition of these two functorial morphism, and making natural identifications of modules, we obtain the desired functorial Hurewicz morphism
$$\epsilon_\psi: B \otimes_{A} \Cofib(\psi) \to \Lb_{B/A}.$$
\end{cons}

Since everything is canonical, it is easy to globalize the Hurewicz morphism.

\begin{defn}\label{GlobalHurewicz}
If $f: X \to Y$ is a morphism of derived schemes, then we are going to denote by $f^\sharp$ the induced morphism $\Oc_Y \to f_* \Oc_X$. It is then immediate that the Hurewicz morphisms of Construction \fref{FunctorialHurewiczCons} glue together to give the \emph{global Hurewicz morphism}
$$\epsilon_f: f^* \Cofib(f^\sharp) \to \Lb_{X/Y}.$$
\end{defn}

The naturality of the Hurewicz morphism in commutative squares of algebras immediately translates into the following statement.

\begin{thm}\label{NaturalityOfGlobalHurewiczThm}
Suppose 
\begin{center}
\begin{tikzcd}
X' \arrow[]{r}{f'} \arrow[]{d}{g'} & Y' \arrow[]{d}{g} \\
X \arrow[]{r}{f} & Y
\end{tikzcd}
\end{center}
is a commutative square of derived schemes, and consider it as a morphism $G: f' \to f$. Then the commutative squares induced by the functoriality of Construction \fref{FunctorialHurewiczCons} glue together to give a commutative square
\begin{center}
\begin{tikzcd}
f^* \Cofib(f^\sharp) \arrow[]{r}{\epsilon_{f}} \arrow[]{d}{f^* (\psi'_G)} &  \Lb_{X/Y} \arrow[]{d}{\Lb'_G} \\
g'_* f'^* \Cofib(f'^\sharp) \arrow[]{r}{g'_*(\epsilon_{f'})} & g'_*\Lb_{X'/Y'}
\end{tikzcd}
\end{center}
of coherent sheaves, where $\Lb'_G$ is the natural morphism on relative cotangent complexes induced by $G$, and $\psi'_G$ is the morphism induced by taking horizontal cofibres of
\begin{center}
\begin{tikzcd}
\Oc_Y \arrow[]{r}{f^\sharp} \arrow[]{d}{g^\sharp} & f_* \Oc_X \arrow[]{d}{f_*(g'^\sharp)} \\
g_*\Oc_{Y'} \arrow[]{r}{g_*(f'^\sharp)} & g_* f'_* \Oc_{X'}. 
\end{tikzcd}
\end{center}
\end{thm}

In the sequel, we are only going to use the following corollary, which is the derived analogue of the classical fact that the conormal sheaf of a closed embedding determined by a sheaf of ideals $\Ic$ can be naturally identified with $\Ic / \Ic^2$.

\begin{cor}\label{ConormalVsIdealCor}
Suppose
\begin{equation*}
\begin{tikzcd}
Z' \arrow[hook]{r}{i'} \arrow[]{d}{g'} & X' \arrow[]{d}{g} \\
Z \arrow[hook]{r}{i} & X,
\end{tikzcd}
\end{equation*}
is a commutative square of derived schemes with $i$ and $i'$ closed embeddings, and consider it as a morphism $G: i' \to i$. Then the induced morphism 
$$\pi_1(\Lb_G): \pi_1(g'^* \Lb_{Z/X}) \to \pi_1(\Lb_{Z'/X'})$$
can be naturally identified with the induced morphism
$$\pi_0(i'^*\psi_G ): \pi_0( i'^* g^*\Ic) \to \pi_0(i'^* \Ic'),$$
where $\Ic := \Fib(i^\sharp)$ and $\Ic' := \Fib(i'^\sharp),$ and where $\psi_G$ and $\Lb_G$ are the adjoints of their primed counterparts in Theorem  \fref{NaturalityOfGlobalHurewiczThm}. 
\end{cor}
\begin{proof}
Taking adjoint of the square of Theorem \fref{NaturalityOfGlobalHurewiczThm}, we obtain a commutative square 
\begin{center}
\begin{tikzcd}
i'^*g^* \Cofib(i^\sharp) \arrow[]{r}{g'^*(\epsilon_{i})} \arrow[]{d}{i'^* (\psi_G)} &  g'^* \Lb_{Z/X} \arrow[]{d}{\Lb_G} \\
i'^* \Cofib(i'^\sharp) \arrow[]{r}{\epsilon_{i'}} & \Lb_{Z'/X'}.
\end{tikzcd}
\end{center}
By \cite{Lur2} Theorem 7.4.3.1, the morphisms $\epsilon_{i'}$ and $g'^*(\epsilon_i)$ induce isomorphisms on $\pi_0$ and $\pi_1$, so the claim follows from the natural identification of $\Cofib(i^\sharp)$ and $\Cofib(i'^\sharp)$ with $\Ic[1]$ and $\Ic'[1]$ respectively. 
\end{proof}

\section{Derived blow ups and deformation to normal bundle}\label{DerivedBlowUpSect}

The purpose of this section is to study the constructions of derived blow ups and deformation to normal cone from \cite{Khan}, and to record how they interact with the analogous classical constructions. The results of this section can be regarded as elaborating on \cite{Khan} Theorem 4.1.5. (vii) and Section 4.1.11: we not only want to understand how the classical blow up sits inside the derived blow up, but also for example the relationship between the exceptional divisors. The results of this section play a fundamental role in the construction and study of virtual pullbacks in Appendix \fref{VirtualPullbackSect}. For simplicity, we will assume throughout the section that all derived schemes are Noetherian. We will denote $\Nc^\vee_{X/Y} := \Lb_{X/Y}[-1]$, and call it the the \emph{conormal complex} (or the \emph{conormal sheaf}); note that if $X \hookrightarrow Y$ is a derived regular embedding, then $\Nc^\vee_{X/Y}$ is a vector bundle.

The structure of this section is as follows: in Section \fref{DerivedBlowUpSubSect},  we recall the definition and basic properties of derived blow up from \cite{Khan}. In Section \fref{IdealLikeSheafSubSect} we study schemes associated to so called \emph{ideal-like sheaves}, which we will prove to model the truncations of derived blow ups in Section \fref{BlowUpTruncationSubSect}. Section \fref{ComparisonWithClassicalSubSect} is devoted to understanding how classical blow up sits inside derived blow up. Finally, in Section \fref{DeformationToNormalBundleSubSect}, we will use the results of the previous sections to study the related construction of \emph{derived deformation to normal bundle}, and to prove several results analogous to those in Verdier's fundamental article \cite{Ver}.

\subsection{Derived blow ups}\label{DerivedBlowUpSubSect}

Recall that a \emph{quasi-smooth closed immersion} (or a \emph{derived regular embedding}) is a closed immersion of derived schemes that affine locally looks like the opposite of $A \to A \modmod (a_1,...,a_r)$, where the right hand side is the derived quotient defined in \cite{Khan} Section 2.3.1. The following two definitions are from \cite{Khan}.

\begin{defn}\label{DivisorOverZDefn}
Let $Z \hookrightarrow X$ be a quasi-smooth closed embedding of derived schemes. Then, for any $X$-scheme $S$, a \emph{virtual Cartier divisor} on $S$ \emph{lying over $Z$} is the datum of a coherently commutating diagram
\begin{center}
\begin{tikzcd}
D \arrow[hook]{r}{i_D} \arrow[]{d}{g} & S \arrow[]{d}{} \\
Z \arrow[hook]{r}{} & X
\end{tikzcd}
\end{center}
such that
\begin{enumerate}
\item $i_D$ is a quasi-smooth closed embedding of virtual codimension 1;
\item the underlying square of classical schemes is Cartesian;
\item the canonical morphism
$$g^* \Nc^\vee_{Z/X} \to \Nc^\vee_{D/S}$$
induces a surjection of sheaves on $\pi_0$.
\end{enumerate}
\end{defn}

\begin{defn}
Let $Z \hookrightarrow X$ be a quasi-smooth closed embedding of derived schemes. Then the \emph{derived blow up} $\bl_Z(X)$ is the $X$-scheme representing virtual Cartier divisors lying over $Z$. In other words, given an $X$-scheme $S$, the space of $X$-morphisms 
$$S \to \bl_Z(X)$$
is naturally identified with the maximal sub $\infty$-groupoid of the $\infty$-category of virtual Cartier divisors of $S$ that lie over $Z$.
\end{defn}

The following is part of Theorem 4.1.5 of \cite{Khan} with one minor modification.

\begin{thm}\label{BlowUpPropertiesThm}
Let $i: Z \hookrightarrow X$ be a quasi-smooth closed embedding of quasi-projective derived $k$-schemes. Then
\begin{enumerate}
\item the derived blow up $\bl_Z(X)$ exists and is unique up to contractible space of choices;

\item the structure morphism $\pi: \bl_Z(X) \to X$ is projective, quasi-smooth, and induces an equivalence 
$$\bl_Z(X) - D_u \to X - Z,$$
where $D_u$ is the universal virtual Cartier divisor on $\bl_Z(X)$ lying over $Z$;

\item the derived blow up $\bl_Z(X) \to X$ is stable under derived base change;

\item there is a natural morphism
$$\Pb_Z(\Nc_{Z/X}) \to \bl_Z(X)$$
exhibiting the left hand side as the universal virtual Cartier divisor on $\bl_Z(X)$ lying over $Z$; moreover, the induced surjection 
$$g^*\Nc^\vee_{Z/X} \to \Nc^\vee_{D_u / \bl_Z(X)}$$
is naturally identified with the natural surjection $g^*\Nc^\vee_{Z/X} \to \Oc(1)$ on $\Pb_Z(\Nc_{Z/X})$;

\item if $Z \stackrel i \hookrightarrow X \stackrel j \hookrightarrow Y$ is a sequence of quasi-smooth closed embeddings, then the outer square in
\begin{center}
\begin{tikzcd}
D_u \arrow[hook]{rr}{i_{D_u}} \arrow[]{d}{g_u} & & \bl_Z(X) \arrow[]{ld}[swap]{\pi} \arrow[]{d}{j \circ \pi} \\
Z \arrow[hook]{r}{i} & X \arrow[hook]{r}{j} & Y
\end{tikzcd}
\end{center}
is a virtual Cartier divisor on $\bl_Z(X)$ lying over $Z \hookrightarrow Y$, and induces a quasi-smooth closed embedding (the \emph{strict transform})
$$\tilde j: \bl_Z(X) \hookrightarrow \bl_Z(Y);$$

\item if $Z$ and $X$ are classical schemes (so that $Z \hookrightarrow X$ is lci), there is a natural equivalence 
$$\bl_Z(X) \simeq \bl_{Z}^\mathrm{cl}(X),$$
where the right hand side is the classical blow up. 

\end{enumerate}
\end{thm}
\begin{proof}
Everything except the projectivity of $\pi$ follow from \cite{Khan}. The projectivity of $\pi$ is Proposition 2.73 of \cite{AY}.
\end{proof}

\subsection{Ideal-like sheaves and their symmetric algebras}\label{IdealLikeSheafSubSect}

The purpose of this section is to study a particularly nice class of coherent sheaves (with extra structure) on classical schemes, and homogeneous spectra of their symmetric algebras. Let us begin with the main definition.

\begin{defn}\label{IdealLikeSheafDef}
Let $X$ be a classical scheme, and let $\Fc$ be a discrete quasi-coherent sheaf on $X$ together with a morphism $\mu: \Fc \to \Oc_X$. Such a pair $(\Fc, \mu)$ is called an \emph{ideal-like sheaf} if the morphism
$$\Fc \otimes^\mathrm{cl}_{\Oc_X} \Fc \xrightarrow{\mathrm{Id} \otimes \mu - \mu \otimes \mathrm{Id}} \Fc$$
vanishes. Above, $\otimes^\mathrm{cl}$ denotes the underived tensor product.
\end{defn}

The main motivation to study such pairs (as well as the motivation behind the terminology) is given by the following result.

\begin{prop}\label{DerivedIdealsAreIdealLikeProp}
Let $i: Z \hookrightarrow X$ be a closed embedding of derived schemes, and let 
$$\Ic \xrightarrow{\mu} \Oc_X \xrightarrow{i^\sharp} i_* \Oc_Z$$
be a cofibre sequence. Then $(\pi_0(\Ic), \pi_0(\mu))$ is an ideal-like sheaf on the truncation $\tau_0(X)$.
\end{prop}
\begin{proof}
Without loss of generality we may assume that everything is affine, and we have a cofibre sequence
$$I \xrightarrow{\mu} A \xrightarrow{\psi} B$$
of simplicial $A$-modules, with $\psi$ a map of simplicial commutative rings inducing a surjection on $\pi_0$. Recall that a point of $I$ is just a point $a$ of $A$ together with a path $\gamma: \psi(a) \leadsto 0$ in $B$. In order to prove the claim, we have to show that given two such points, say $(a, \gamma)$ and $(a', \gamma')$, there exists a path
$$(aa', \psi(a') \gamma) \sim (aa', \psi(a) \gamma'),$$
so it is enough to find a path homotopy between $\psi(a') \gamma$ and $\psi(a) \gamma'$, which is taken care of by Lemma \ref{DerivedPathLem}.
\end{proof}

\begin{lem}\label{DerivedPathLem}
Let $A$ be a simplicial ring, and consider paths $\gamma$ and $\gamma'$ from $a$ to $0$ and $a'$ to 0 respectively. Then there exists a path homotopy
$$a \gamma' \sim a' \gamma.$$
\end{lem}
\begin{proof}
Regarding $\gamma$ and $\gamma'$ as maps of simplicial sets $\Delta^1 \to A$, we obtain
$$\gamma \cdot \gamma' := \Delta^1 \times \Delta^1 \xrightarrow{\gamma \times \gamma'} A \times A \xrightarrow{\cdot} A.$$
Noting that 
$$
(\gamma \cdot \gamma') \vert_{\sigma} = 
\begin{cases}
a \gamma' & \text{if $\sigma = [0]\times \Delta^1$;} \\
a' \gamma & \text{if $\sigma = \Delta^1 \times [0]$;} \\
0 & \text{if $\sigma = [1] \times \Delta^1$ or $\sigma = \Delta^1 \times [1],$}
\end{cases}
$$
the desired homotopy can be found by gluing a constant $2$-simplex at $0$ to $\Delta^1 \times \Delta^1$ along $[1] \times \Delta^1$ and $\Delta^1 \times [1]$ to obtain a horn, and then invoking the fact that the underlying simplicial set of $A$ is a Kan complex.
\end{proof}

Being ideal-like gives rise to virtual exceptional divisor on the homogeneous spectrum of the symmetric algebra. Indeed, let $X$ be a scheme and let $(\Fc, \mu)$ be an ideal-like sheaf on $X$. Then  $\mu$ induces well defined morphisms
$$\Sym^{r+1}_X(\Fc) \to \Sym^r_X (\Fc),$$
which are described affine locally by formula
$$m_0 m_1 \cdots m_r \mapsto \mu (m_0) m_1 \cdots m_r,$$
and these morphisms give rise to a morphism 
$$\mu: \bigoplus_{i=0}^\infty \Sym^{i+1}_X(\Fc) \to \bigoplus_{i=0}^\infty \Sym^{i}_X(\Fc)$$
of graded modules over the symmetric algebra $\Sym_X^*(\Fc)$. Passing to homogeneous spectra, we obtain a natural morphism
$$s_\mu^\vee: \Oc(1) \to \Oc$$
of line bundles on $\Pc roj\big( \Sym^*_X(\Fc) \big)$.

\begin{defn}\label{VirtualExceptionalDivDef}
Let everything be as above. Then the derived vanishing locus of the induced global section 
$$s_\mu \in \Gamma\big(\Pc roj(\Sym_X^*(\Fc)); \Oc(-1)\big)$$
is called the \emph{virtual exceptional divisor} and is denoted by $\Ec_\Fc$.
\end{defn}

We finish by showing that the virtual exceptional divisor associated to an ideal-like sheaf coming from a derived regular embedding $Z \hookrightarrow X$ is a virtual Cartier divisor lying over $Z$, and therefore induces a morphism $\Pc roj\big(\Sym^*_X(\Fc)\big) \to \bl_Z(X)$. This is a key step in identifying $\Pc roj\big(\Sym^*_X(\Fc)\big)$ with the truncation of $\bl_Z(X)$, which is done in Section \ref{BlowUpTruncationSubSect}.

\begin{thm}\label{BlowUpSquareOnProjThm}
Let $i: Z \hookrightarrow X$ be a derived regular embedding, and let $\Ic$ be the homotopy fibre of $i^\sharp$. Then there exists a natural (see Remark \ref{BlowUpSquareOnProjNaturalityRem}) commutative square
\begin{center}
\begin{tikzcd}
\Ec_{\pi_0(\Ic)}  \arrow[hook]{r} \arrow[]{d}{g_{Z/X}} & \Pc roj\big(\Sym_X^*(\pi_0(\Ic))\big) \arrow[]{d}{} \\
Z \arrow[hook]{r}{i} & X
\end{tikzcd}
\end{center}
exhibiting $\Ec_{\pi_0(\Ic)}$ as a virtual Cartier divisor lying over $Z$.
\end{thm}
\begin{proof}
Our goal is to find $g_{Z/X}$. Without loss of generality we may assume $X$ is classical. Let $\Spec(A)$ be an affine open set where $Z = \Spec(B)$ with $B = A\modmod(a_1,...,a_r)$, $I$ be the homotopy fibre of $A \to B$ and $I_0 := \pi_0(I)$. We start by finding $g_{\Spec(B)/\Spec(A)}$, and then glue these together to obtain $g_{Z/X}$. As we will see below, all the relevant mapping spaces are discrete, so we will not have to deal with any higher coherence problems.

Note that the $a_i$ together with the homotopy classes of the canonical paths $a_i \leadsto 0$ in $B$ give rise to generators $m_1,...,m_r$ of $I_0$ mapping to $a_1,...,a_r$ under $\mu: I_0 \to A$. Hence $\Proj\big(\Sym^*_A(I_0)\big)$ has an affine cover by $\Spec\big(\Sym^*_A(I_0)_{(m_i)}\big)$, where for a graded ring $R$ and a homogeneous element $a \in R$, $R_{(a)}$ denotes the degree $0$ part of the localization $R_a$. We notice that on $\Spec\big(\Sym^*_A(I_0)_{(m_i)}\big)$ the exceptional divisor is identified as $\Spec\big(\Sym^*_A(I_0)_{(m_i)} \modmod(a_i) \big)$. Letting $\gamma_i$ be the canonical path $a_i \leadsto 0$ in the underlying space of $\Sym^*_A(I_0)_{(m_1)} \modmod (a_i)$, we obtain a morphism of $A$-algebras 
$$\Big({a_1 \over a_i } \gamma_i, ..., {a_r \over a_i} \gamma_i \Big): B \to \Sym^*_A(I_0)_{(m_i)} \modmod (a_i)$$
using the universal property \cite{Khan} Lemma 2.3.5. 

To prove that these glue to give the desired morphism, we have to show that the two maps of $A$-algebras from $B$ to
$$\Sym^*_A(I_0)_{(m_i)} \left[\dfrac{m_i}{m_j} \right ] \modmod (a_i) \simeq \Sym^*_A(I_0)_{(m_j)} \left[\dfrac{m_j}{m_i} \right ] \modmod (a_j)$$
induced by the tuples of paths 
$$\left ( \dfrac{a_1}{a_i} \gamma_i, ..., \dfrac{a_r}{a_i} \gamma_i \right ) \text{  and  } \left ( \dfrac{a_1}{a_j} \gamma_j, ..., \dfrac{a_r}{a_j} \gamma_j \right )$$
are equivalent (the space of $A$-algebra maps from $B$ to any $1$-truncated $A$-algebra is discrete). However, on the grounds of $I_0$ being ideal-like, we have that $m_i/m_j = a_i/a_j$, so we can compute that
\begin{align*}
\dfrac{a_k}{a_i} \gamma_i &\sim \dfrac{a_k}{a_j} \dfrac{a_j}{a_i} \gamma_i \\
&\sim  \dfrac{a_k}{a_j} \gamma_j, & (\text{Lemma \ref{DerivedPathLem}})
\end{align*}
showing that the tuples are equivalent, and so are the two induced morphisms. Hence we can glue to obtain $g_{\Spec(B)/\Spec(A)}$.

A similar argument shows that the morphism depends only on $B$ and not the chosen representation $B = A \modmod(a_1,...,a_r)$. Hence there exists a canonical map $g_{\Spec(B)/\Spec(A)}$ for any sufficiently small affine open set $\Spec(A)$ of $X$ which is stable under restrictions, proving that we can glue these to obtain the desired morphism $g_{Z/X}$. We have therefore found the desired commutative square
$$
\begin{tikzcd}
\Ec_{\pi_0(\Ic)}  \arrow[hook]{r} \arrow[]{d}{g_{Z/X}} & \Pc roj\big(\Sym_X^*(\pi_0(\Ic))\big) \arrow[]{d}{} \\
Z \arrow[hook]{r}{i} & X.
\end{tikzcd}
$$
so the only thing left for us to do is to check that it truncates to a Cartesian square and that the induced morphism $g_{Z/X}^* \Nc^\vee_{Z/X} \to \Nc^\vee_{\Ec_{\pi_0(\Ic)} / \Pc roj\big(\Sym_X^*(\pi_0(\Ic))\big)}$ is surjective.

Both claims can be checked on the affine local presentation we considered above. Indeed, since $m_j / m_i = a_j / a_i$ in $\Spec^*_A(I_0)_{(m_i)}$, we see that 
$$\Sym^*_A(I_0)_{(m_i)}/(a_i) =  \Sym^*_A(I_0)_{(m_i)}/(a_1,...,a_r),$$
proving the first claim. On the other hand, to prove the second claim, we simply note that the $A$-morphism
$$\left ( \dfrac{a_1}{a_i} \gamma_i,..., \gamma_i,  ..., \dfrac{a_r}{a_i} \gamma_i \right ): B \to \Sym^*_A(I_0)_{(m_i)} \modmod (a_i),$$
induces a morphism $g_{\Spec(B)/\Spec(A)}^* I \to I'$ between the fibres of the horizontal morphisms in the above square, and this is a surjection after truncating since the image of $m_i$ clearly generates $I'$. The surjectivity of the map of conormal bundles then follows from Corollary \ref{ConormalVsIdealCor}, so we are done.
\end{proof}

\begin{rem}\label{BlowUpSquareOnProjNaturalityRem}
Let us record here what we mean by naturality in the statement of the above result. Suppose 
$$
\begin{tikzcd}
X' \arrow[]{d}{f'} \arrow[hook]{r}{i'} & Y' \arrow[]{d}{f} \\
X \arrow[hook]{r}{i} & Y
\end{tikzcd}
$$
is derived Cartesian, with the horizontal morphisms derived regular embeddings, and let $\Ic$ and $\Ic'$ be the derived fibres of $i^\sharp$ and $i'^\sharp$ respectively. Let $(\Ic_0, \mu)$ and $(\Ic'_0, \mu')$ be the induced ideal-like sheaves on $\tau_0(Y)$ and $\tau_0(Y')$. Then the canonical identification $f'^* \Ic \simeq \Ic'$ provides a canonical classical Cartesian square
$$
\begin{tikzcd}
\Pc roj\big(\Sym_{Y'}^*(\Ic'_0)\big) \arrow[]{d} \arrow[]{r}{f''} & \Pc roj\big(\Sym_Y^*(\Ic_0)\big) \arrow[]{d} \\
\tau_0(Y') \arrow[]{r}{f} & \tau_0(Y).
\end{tikzcd}
$$
Moreover, since the canonical global sections of $\Oc(-1)$ are clearly compatible in the sense that $s_{\mu'} = f''^* s_{\mu}$, we obtain a canonical derived Cartesian square
$$
\begin{tikzcd}
\Ec_{\Ic'_0} \arrow[hook]{d} \arrow[]{r}{f'''} & \Ec_{\Ic_0} \arrow[hook]{d} \\
\Pc roj\big(\Sym_{Y'}^*(\Ic'_0)\big)  \arrow[]{r}{f''} & \Pc roj\big(\Sym_Y^*(\Ic_0)\big).
\end{tikzcd}
$$
We claim that the induced square of derived $Y$-schemes
$$
\begin{tikzcd}
\Ec_{\Ic'_0} \arrow[hook]{d} \arrow[]{r}{f'''} & \Ec_{\Ic_0} \arrow[hook]{d} \\
Z'  \arrow[]{r}{f'} & Z
\end{tikzcd}
$$
commutes up to canonical homotopy. Proving this is easy, because all the relevant mapping spaces are discrete, so one can check everything in the affine local presentation used in the proof of Theorem \ref{BlowUpSquareOnProjThm}. 
\end{rem}

Finally, since the truncation of the virtual exceptional divisor has an easy description, we record it here.

\begin{prop}\label{TruncationOfVirtualExceptionalDivProp}
Let $X$ be a scheme and $(\Fc, \mu)$ an ideal-like sheaf on $X$. Then the underlying classical scheme of the virtual exceptional divisor $\Ec_\Fc$ is naturally identified with
$$\Pc roj \big( \Sym^*_{Z_\Fc}(\Fc \vert_{Z_\Fc}) \big),$$
where $Z_\Fc$ is the classical vanishing locus of the image of $\mu: \Fc \to \Oc_X$.
\end{prop}
\begin{proof}
The proof is trivial.
\end{proof}

\subsection{Truncation of derived blow up}\label{BlowUpTruncationSubSect}

Suppose that $i: Z \hookrightarrow X$ is a quasi-smooth closed embedding, and denote by $\Ic$ the homotopy fibre of $i^\sharp: \Oc_X \to i_* \Oc_Z$. The main purpose of this section is to provide a natural identification of $\tau_0\big(\bl_Z(X)\big)$ with $\Pc roj\big(\Sym_{\tau_0(X)}^* (\pi_0(\Ic))\big)$, and to give an alternative characterization of strict transforms under this identification.

We already know (see Theorem \fref{BlowUpSquareOnProjThm} and the remark following it) that the virtual exceptional divisor on $\Pc roj\big(\Sym_{\tau_0(X)}^* (\pi_0(\Ic))\big)$ gives rise to a canonical morphism 
$$\rho'_{Z/X}: \Pc roj \big(\Sym^*_{\tau_0(X)}(\pi_0(\Ic))\big) \to \tau_0\big(\bl_Z(X)\big).$$
Moreover, by the following Lemma, there is a morphism 
$$\rho_{Z/X}: \tau_0\big(\bl_Z(X)\big) \to \Pc roj \big(\Sym^*_{\tau_0(X)}(\pi_0(\Ic))\big)$$
in the other direction.

\begin{lem}\label{SurjOfIdealsLem}
Taking horizontal fibres in the commutative square
\begin{center}
\begin{tikzcd}
\pi^* \Oc_X \arrow[]{r} \arrow[]{d} & \pi^* i_* \Oc_Z \arrow[]{d} \\
\Oc_{\bl_Z(X)} \arrow[]{r} & i_{D_u *}\Oc_{D_u}
\end{tikzcd}
\end{center}
of coherent sheaves on $\bl_Z(X)$ induced by the universal blow up square, we obtain a natural morphism
$$\psi: \pi^* \Ic \to \Oc(-D_u)$$
which induces a surjection on $\pi_0$.
\end{lem}
\begin{proof}
We start by noting that there is a natural diagram
\begin{center}
\begin{tikzcd}
\pi_0 (\pi^* \Ic) \arrow[]{d}{\pi_0(\psi)} \arrow[]{rr} \arrow[twoheadrightarrow]{rd}{} & & \Oc \arrow[]{d}{=} \\
\pi_0 \big(\Oc(- D_u)\big) \arrow[twoheadrightarrow]{r}{\phi} & \Ic' \arrow[]{r}{} & \Oc
\end{tikzcd}
\end{center}
of coherent sheaves on $\tau_0 \big(\bl_Z(X)\big)$, where $\Ic'$ is the sheaf of ideals associated to the classical immersion $\tau_0(D_u) \hookrightarrow \tau_0 \big( \bl_Z(X) \big)$, and the diagonal morphism is surjective because the blow up square truncates to a Cartesian square. If a point $\mathfrak p \in X$ is not in $\tau_0(D_u)$, then $\pi_0(\psi)_\mathfrak{p}$ is surjective because $\phi_\mathfrak{p}$ is an isomorphism. On the other hand, if $\mathfrak{p} \in \tau_0(D_u)$, then the induced morphism
$$\pi_0(\psi) \otimes_{\Oc_\mathfrak{p}} \big(\Oc_\mathfrak{p} / \Ic'_\mathfrak{p}\big) : \pi_0 \big(\pi^* \Ic\big)_\mathfrak{p} \otimes_{\Oc_\mathfrak{p}} \big(\Oc_\mathfrak{p} / \Ic'_\mathfrak{p}\big) \to \pi_0 \big(\Oc(- D_u)\big) \otimes_{\Oc_\mathfrak{p}} \big(\Oc_\mathfrak{p} / \Ic'_\mathfrak{p}\big)$$
is surjective by the universal property of the derived blow up combined with  Corollary \fref{ConormalVsIdealCor}, so the surjectivity of $\pi_0(\psi)$ at $\mathfrak{p}$ follows from Nakayama's Lemma. Hence $\pi_0(\psi)$ is surjective at each point of $X$, proving the claim.
\end{proof}

The main result of this subsection is the following theorem. 

\begin{thm}\label{BlowUpAndSymmAlgebraThm}
The morphisms $\rho_{Z/X}$ and $\rho'_{Z/X}$ are inverses of each other.
\end{thm}
\begin{proof}
By naturality, it is enough to check this in the case of $i$ being $\{0\} \hookrightarrow \Ab^n$. Note that now $\Ic$ is just the ideal $I = \langle x_1, ... , x_n \rangle \subset k[x_1, ..., x_n]$, so it is discrete; moreover, $\bl_{\{0\}}(\Ab^n)$ is a classical scheme. Let us simplify the notation by denoting the morphisms of interest by $\rho$ and $\rho'$ respectively.

To prove that the composition $\rho' \circ \rho$ is the identity, we merely have to notice that the vertical morphism of
\begin{center}
\begin{tikzcd}
I \arrow[twoheadrightarrow]{d}{} \arrow[]{rd}{} & \\
\Oc(-D_u) \arrow[]{r}[swap]{s_{D_u}^\vee} & \Oc
\end{tikzcd}
\end{center}
being an epimorphism implies that no other morphism than $s_{D_u}^\vee$ can make the triangle commute. It follows that the virtual exceptional divisor on $\Pc roj \big(\Sym^*_{\Ab^n}(I)\big)$ pulls back along $\rho$ to the universal divisor $D_u$ on $\bl_{\{0\}}(\Ab^n)$. 

Similarly, by \cite{Mic} Chapter 1 Théorème 1 the symmetric and the Rees algebras of $I$ coincide, and it follows that the horizontal morphism of the triangle
\begin{center}
\begin{tikzcd}
I  \arrow[twoheadrightarrow]{d} \arrow[]{rd}{} & \\
\Oc(1) \arrow[]{r}{s_\mu^\vee} & \Oc
\end{tikzcd}
\end{center}
on $\Pc roj \big(\Sym^*_{\Ab^n}(I)\big)$ is a monomorphism. Consequently, the surjection $I \to \Oc(-D_u)$ on $\bl_{\{0\}}(\Ab^n)$ has to pull back along $\rho'$ to the universal surjection $I \to \Oc(-1)$ on $\Pc roj \big(\Sym^*_{\Ab^n}(I)\big)$, and therefore the composition $\rho \circ \rho'$ is the identity.
\end{proof}

Next we are going to provide a concrete characterization of strict transforms under this identification. Suppose $Z \stackrel i \hookrightarrow X \stackrel j \hookrightarrow Y$ is a sequence of quasi-smooth closed embeddings, and consider the induced commutative square 
\begin{center}
\begin{tikzcd}
\Oc_Y \arrow[]{r}{(j \circ i)^\sharp} \arrow[]{d}{j^\sharp} & j_* i_* \Oc_Z \arrow[]{d}{\mathrm{Id}} \\
j_* \Oc_X \arrow[]{r}{j_* (i^\sharp)} & j_* i_* \Oc_Z,
\end{tikzcd}
\end{center}
Taking horizontal fibres, we obtain a natural morphism
$$\psi_{Z/X/Y}: \Ic' \to j_* \Ic,$$
and comparing the long exact homotopy sequences, we see that $\pi_0(\psi_{Z/X/Y})$ is a surjection of sheaves. It turns out that $\psi_{Z/X/Y}$ determines the truncation of the strict transform.

\begin{prop}\label{StrictTransformProp}
Let everything be as above. Then, under the identification of Theorem \fref{BlowUpAndSymmAlgebraThm}, the truncation $\tau_0(\tilde j): \tau_0\big(\bl_Z(X)\big) \hookrightarrow \tau_0\big(\bl_Z(Y)\big)$ of the strict transform coincides with the morphism induced by the surjection
$$\Sym^*_{\tau_0(Y)}\big(\pi_0 (\Ic')\big) \twoheadrightarrow \Sym^*_{j_* \Oc_{\tau_0(X)}}\big(\pi_0 (j_* \Ic)\big)$$
which is induced by the surjection of sheaves
$$\pi_0(\psi_{Z/X/Y}): \pi_0(\Ic') \twoheadrightarrow \pi_0(j_* \Ic)$$
and $j^\sharp$.
\end{prop}
\begin{proof}
We need to show that the morphism 
$$\rho_{Z/Y} \circ \tau_0(\tilde j) \circ \rho'_{Z/X}$$
coincides with the morphism corresponding to the composition 
\begin{center}
\begin{tikzcd}
\pi_0 (j^*\Ic') \arrow[twoheadrightarrow]{r}{\pi_0(\psi'_{Z/X/Y})} & \pi_0(\Ic) \arrow[twoheadrightarrow]{r}{} & \Oc(1)
\end{tikzcd}
\end{center}
where the second morphism is the universal surjection on $\Pc roj\big(\Sym^*_{\tau_0(X)}(\pi_0 (\Ic))\big)$, and $\psi'_{Z/X/Y}$ is the adjoint of $\psi_{Z/X/Y}$. But this is trivial, since by construction the strict transform $\tilde j$ is induced by the virtual Cartier divisor over $Z$ given by the outer square of 
$$
\begin{tikzcd}
D_u \arrow[hook]{r}{} \arrow[]{d}{} & \bl_Z(X) \arrow[]{d}{} \\
Z \arrow[hook]{r}{} \arrow[]{d}{\mathrm{Id}} & X \arrow[hook]{d}{} \\
Z \arrow[hook]{r}{} & Y
\end{tikzcd}
$$
so that $\rho_{Z/Y} \circ \tau_0(\tilde j)$ is induced by the surjection 
$$
\begin{tikzcd}
\pi_0 (j^*\Ic') \arrow[twoheadrightarrow]{r}{\pi_0(\psi'_{Z/X/Y})} & \pi_0(\Ic) \arrow[twoheadrightarrow]{r}{} & \Oc(-D_u).
\end{tikzcd}
$$
The claim then follows from Theorem \fref{BlowUpAndSymmAlgebraThm}, as $\pi_0(\Ic) \twoheadrightarrow \Oc(-D_u)$ pulls back along $\rho'_{Z/X}$ to the universal surjection $\pi_0(\Ic) \twoheadrightarrow \Oc(1)$ on $\Pc roj\big(\Sym^*_{\tau_0(X)}(\pi_0 (\Ic))\big)$.
\end{proof}

\subsection{Comparison with classical blow up}\label{ComparisonWithClassicalSubSect}

The purpose of this section is to study how the classical blow up interacts with the derived blow up. Suppose $i: Z \hookrightarrow X$ is a quasi-smooth closed immersion, and consider the classical blow up
$$\bl^\mathrm{cl}_{\tau_0(Z)}\big(\tau_0(X)\big) := \Pc roj\big(\Oc_{\tau_0(X)}[\Ic_0 t]\big),$$
 which is given by the homogeneous spectrum of the Rees algebra on the sheaf of ideals $\Ic_0$ cutting out $\tau_0(Z)$ from $\tau_0(X)$. Let us denote by $\Ic$ be the fibre of $i^\sharp: \Oc_X \to i_* \Oc_Z$, and notice that taking horizontal fibres in the induced commutative square
\begin{center}
\begin{tikzcd}
\Oc_{X} \arrow[]{r} \arrow[]{d} & \Oc_{Z} \arrow[]{d} \\
\Oc_{\tau_0(X)} \arrow[]{r} & \Oc_{\tau_0(Z)}
\end{tikzcd}
\end{center}
we obtain a natural morphism $\psi_{Z/X}: \Ic \to \Ic_0$. Let us make the following simple observation.

\begin{lem}\label{PsiZXSurjLem}
The morphism $\pi_0(\psi_{Z/X}): \pi_0(\Ic) \to \Ic_0$ is a surjection of sheaves.
\end{lem}
\begin{proof}
Indeed, comparing the induced homotopy long exact sequences
$$
\begin{tikzcd}
\pi_1(\Oc_Z) \arrow[]{r} \arrow[]{d} & \pi_0(\Ic) \arrow[]{r} \arrow[]{d}{\pi_0(\psi_{Z/X})} & \pi_0(\Oc_X) \arrow[]{r} \arrow[]{d}{\cong} & \pi_0(\Oc_Z) \arrow[]{r} \arrow[]{d}{\cong} & 0 \arrow[]{d} \\
0 \arrow[]{r} & \Ic_0 \arrow[]{r} & \Oc_{\tau_0(X)} \arrow[]{r} & \Oc_{\tau_0(Z)} \arrow[]{r} & 0,
\end{tikzcd}
$$
we see that the surjectivity of $\pi_0(\psi_{Z/X})$ follows from the 4-lemma.
\end{proof}

Comparing the long exact homotopy sequences, we conclude that $\psi_{Z/X}$ induces a surjection on $\pi_0$. 

\begin{lem}
The evident square
\begin{equation*}
\begin{tikzcd}
\Ec \arrow[hook]{r} \arrow[]{d} & \bl^\mathrm{cl}_{\tau_0(Z)}\big(\tau_0(X)\big) \arrow[]{d} \\
Z \arrow[hook]{r} & X.
\end{tikzcd}
\end{equation*}
where $\Ec$ is the exceptional divisor of the classical blow up, exhibits $\Ec$ as a virtual Cartier divisor lying over $Z$. 
\end{lem}
\begin{proof}
The square is of course the outer square in
$$
\begin{tikzcd}
\Ec \arrow[hook]{r} \arrow[]{d} & \bl^\mathrm{cl}_{\tau_0(Z)}\big(\tau_0(X)\big) \arrow[]{d} \\
\tau_0(Z) \arrow[hook]{r} \arrow[]{d} & \tau_0(X) \arrow[]{d} \\
Z \arrow[hook]{r} & X.
\end{tikzcd}
$$
The only thing for us to check is that the third condition of Definition \fref{DivisorOverZDefn} is satisfied. By Corollary \fref{ConormalVsIdealCor} it is enough to check that the induced morphism
$$\Ic \to \Oc(-\Ec)$$
of sheaves on $\bl^\cl_{\tau_0(Z)}(\tau_0(X))$ between the homotopy fibres induces a surjection on $\pi_0$. But this is easy, since the above morphism factors as
$$\Ic \xrightarrow{\psi_{Z/X}} \Ic_0 \to \Oc(-\Ec),$$
the latter morphism is surjective by the universal property of classical blow up, and $\pi_0(\psi_{Z/X})$ is surjective by Lemma \ref{PsiZXSurjLem}.
\end{proof}

We therefore obtain a natural morphism
$$i_{Z/X}: \bl^\mathrm{cl}_{\tau_0(Z)}\big(\tau_0(X)\big) \to \tau_0\big(\bl_Z(X)\big).$$
This morphism has the following alternative description as well:

\begin{thm}\label{InclusionOfClassicalBlowUpAsSymmetricAlgThm}
Let everything be as above. Then, under the identification 
$$\tau_0\big(\bl_Z(X)\big) = \Pc roj\big(\Sym^*_{\pi_0(X)}(\pi_0(\Ic))\big)$$
of Theorem \fref{BlowUpAndSymmAlgebraThm}, the morphism $i_{Z/X}$ coincides with the morphism induced by the evident surjection
$$\Sym^*_{\tau_0(X)}\big(\pi_0(\Ic)\big) \twoheadrightarrow \Oc_{\tau_0(X)}[\Ic_0 t]$$
which is induced by the surjection $\pi_0(\psi_{Z/X})$. In particular, $i_{Z/X}$ is a closed embedding.
\end{thm}
\begin{proof}
By unwinding the definitions, we see that the universal surjection
$$\pi_0 (\Ic) \twoheadrightarrow \Oc(1)$$
on $\Pc roj\big(\Sym^*_{\tau_0(X)}(\pi_0 (\Ic))\big)$
pulls back to
$$
\begin{tikzcd}
\pi_0 (\Ic) \arrow[twoheadrightarrow]{r}{\psi_{Z/X}} & \Ic_0 \arrow[twoheadrightarrow]{r}{} & \Oc(- \Ec)
\end{tikzcd}
$$
on $\bl^\mathrm{cl}_{\tau_0(Z)}\big(\tau_0(X)\big)$. It follows that the map 
$$\Sym^*_{\tau_0(X)}\big(\pi_0(\Ic)\big) \twoheadrightarrow \Oc_{\tau_0(X)}[\Ic_0 t]$$
corresponding to $i_{Z/X}$ is as desired, because it behaves as expected in degree 1, and because the degree 1 part completely determines the map. 
\end{proof}

\subsection{Deformation to the normal bundle}\label{DeformationToNormalBundleSubSect}

The notion of a derived blow up allows us to define derived deformation to the normal bundle. The purpose of this section is to use the results obtained in the previous sections in order to study the truncation of the derived deformation space, as well as to study how it interacts with the classical deformation space. 

Let us begin with the main definition:

\begin{defn}\label{DeformationToNormalBundleDef}
Let $i: Z \hookrightarrow X$ be a quasi-smooth closed immersion of derived schemes. We define the \emph{derived deformation to the normal bundle} as the inclusion of derived schemes
\begin{center}
\begin{tikzcd}
\Ab^1 \times Z \arrow[hook]{r}{j_{Z/X}} & M(Z/X)
\end{tikzcd}
\end{center}
over $\Ab^1 \times X$, where the \emph{deformation space} $M(Z/X)$ is the open complement of the strict transform of $\{0\} \times X$ inside $\bl_{\{0\} \times Z}(\Ab^1 \times X)$, and $j_{Z/X}$ is the strict transform of $\Ab^1 \times Z$ (with restricted codomain). In particular, $j_{Z/X}$ is a quasi-smooth closed immersion.
\end{defn}

Let us recall some basic properties from \cite{Khan} Theorem 4.1.13.

\begin{thm}
Derived deformation to the normal bundle satisfies the following basic properties.
\begin{enumerate}
\item The morphism $j_{Z/X}$ is stable under derived base change.

\item Over $\Gb_m = \Ab^1 \backslash \{0\}$, $j_{Z/X}$ is equivalent to the inclusion
$$
\begin{tikzcd}
\Gb_m \times Z \arrow[hook]{r}{\mathrm{Id} \times i} & \Gb_m \times X.
\end{tikzcd}
$$

\item Over $0$, $j_{Z/X}$ is equivalent to the zero section
$$Z \hookrightarrow \Nc_{Z/X},$$
where the right hand side is considered as an $X$-scheme via the composition  $$\Nc_{Z/X} \to Z \hookrightarrow X.$$
\end{enumerate}
\end{thm}

Our first result identifies the truncation of the derived deformation space with something more concrete, and it is inspired by Proposition 2.15 of \cite{Ver}. Let us start by explaining the setup. The scheme $\Ab^1 \times X$ can be identified as the relative spectrum of $\Oc_X[t]$. Given a quasi-smooth closed immersion $i: Z \hookrightarrow X$, the sequence
$$Z \stackrel i \hookrightarrow X \hookrightarrow \Ab^1 \times X$$
of closed immersions (the second morphism being the zero section) induces a diagram 
$$
\begin{tikzcd}
\Oc_X [t] \arrow[]{r}{} \arrow[]{d}{} & \Oc_X [t] \arrow[]{r}{} \arrow[]{d}{\cdot t} & 0 \arrow[]{d}{} \\
\Ic' \arrow[]{r}{u'} \arrow[]{d}{\psi} & \Oc_X [t] \arrow[]{r}{} \arrow[]{d}{} & \Oc_Z \arrow[]{d}{} \\
\Ic \arrow[]{r}{u} & \Oc_X \arrow[]{r}{} & \Oc_Z
\end{tikzcd}
$$
of quasi-coherent sheaves on $X$ with all rows and columns cofibre sequences. Since the morphism $\Oc_X[t] \to \Oc_X$ admits a section, so does $\psi$, and therefore 
$$\Ic' \simeq \Oc_X[t] \cdot t \oplus \Ic.$$
as quasi-coherent sheaves on $X$.

Moreover, the $\Oc_X[t]$-module structure is easy to describe, at least on the truncation: by chasing the above diagram we see that
$$t \cdot (ft,m) = (ft^2 + u(m)t, 0) \in \Oc_{\tau_0(X)}[t] \cdot t \oplus \pi_0(\Ic).$$
It is therefore easy to conclude that the symmetric algebra over $\Oc_{\tau_0(X)}[t]$ of $\pi_0(\Ic')$, whose homogeneous spectrum represents the truncated blow up $\tau_0\big(\bl_{\{0\} \times Z}(\Ab^1 \times X)\big)$ by Theorem \fref{BlowUpAndSymmAlgebraThm}, can be decomposed as the bigraded algebra
\begin{equation}\label{BigradedDeformationAlgebra}
\begin{tikzcd}
\vdots & \vdots & \vdots & \\
\Sym^2_{\tau_0(X)}\big(\pi_0(\Ic)\big) \cdot v^2 & \Sym^1_{\tau_0(X)}\big(\pi_0(\Ic)\big) \cdot v^2 t & \Oc_{\tau_0(X)} \cdot v^2t^2 & \cdots\\
\Sym^1_{\tau_0(X)}\big(\pi_0(\Ic)\big) \cdot v & \Oc_{\tau_0(X)} \cdot vt & \Oc_{\tau_0(X)} \cdot vt^2 & \cdots \\
\Oc_{\tau_0(X)} & \Oc_{\tau_0(X)} \cdot t & \Oc_{\tau_0(X)} \cdot t^2 & \cdots
\end{tikzcd}
\end{equation}
where $v$ is an abstract variable keeping track of the degree of the symmetric power of $\pi_0(\Ic')$.  By unwinding the definitions, one sees that the coefficient module of the monomial $v^n t^m$ is
$$\Sym_{\tau_0(X)}^{n-m}\big(\pi_0(\Ic)\big)$$
with the convention that the negative symmetric powers are just $\Oc_{\tau_0(X)}$, and the multiplication
$$\cdot: \Sym_{\tau_0(X)}^{n-m}\big(\pi_0(\Ic)\big) \cdot v^n t^m \times \Sym_{\tau_0(X)}^{n'-m'}\big(\pi_0(\Ic)\big) \cdot v^{n'} t^{m'} \to \Sym_{\tau_0(X)}^{n+n'-m-m'}\big(\pi_0(\Ic)\big) \cdot v^{n+n'} t^{m+m'}$$
is given by the table
\begin{equation}\label{MultiplicationTable}
(\alpha \cdot v^n t^m) \cdot (\beta \cdot v^{n'} t^{m'}) =  
\begin{cases}
\alpha \beta \cdot v^{n+n'} t^{m+m'} & \text{if $n-m \geq 0$ and $n'-m' \geq 0$}; \\
\beta u^{m' - n'}(\alpha) \cdot v^{n+n'} t^{m+m'} & \text{if $n-m \geq 0$ and $n'-m' < 0$}; \\
\alpha u^{m-n}(\beta) \cdot v^{n+n'} t^{m+m'} & \text{if $n-m < 0$ and $n'-m' \geq 0$}; \\
\alpha \beta \cdot v^{n+n'} t^{m+m'} & \text{otherwise.}
\end{cases}
\end{equation}
where we have abusively denoted by $u$ the morphisms 
$$\Sym_{\tau_0(X)}^{i+1}\big(\pi_0(\Ic)\big) \to \Sym^i_{\tau_0(X)}\big(\pi_0(\Ic)\big)$$
given by 
$$m_0 m_1 \cdots m_i \mapsto u(m_0) m_1 \cdots m_i,$$
which is well defined since $(\pi_0(\Ic), u)$ is an ideal-like sheaf. It is then easy to prove the following theorem. 

\begin{thm}\label{DeformationSpaceAsSpecThm}
Let $i: Z \hookrightarrow X$ be a quasi-smooth closed immersion, and let $\Ic$ be the fibre of $i^\sharp: \Oc_X \to i_* \Oc_Z$. Then the isomorphism $\rho_{\{0\} \times Z / \Ab^1 \times X}$ of Theorem \fref{BlowUpAndSymmAlgebraThm} restricts to a natural isomorphism
$$\rho^\circ_{Z/X}: \tau_0\big(M(Z/X)\big) \stackrel\cong\to \Sc pec \Big( \bigoplus_{i \in \Zb} \Sym^i_{\tau_0(X)} \big(\pi_0(\Ic)\big) \cdot t^{-i}\Big)$$
of schemes over $\tau_0(X)$, where by convention $\Sym^i_{\tau_0(X)} \big(\pi_0(\Ic)\big) := \Oc_{\tau_0(X)}$ for $i$ negative and the algebra structure is obvious (see the discussion preceding the statement).
\end{thm}
\begin{proof}
Using the notation of the discussion preceding the statement, we can translate Proposition \fref{StrictTransformProp} as saying that the truncated strict transform of $\{0\} \times X$ inside the truncated blow up $\tau_0(\bl_{\{0\} \times Z}\big(\Ab^1 \times X)\big)$ is the vanishing locus of $vt$. Hence $\tau_0\big(M(Z/X)\big)$ is naturally identified with the relative spectrum of the $v$-degree $0$ part of the bigraded algebra in (\fref{BigradedDeformationAlgebra}) after inverting $vt$ (which is a non zero divisor), and the result is clearly as claimed in the statement.   
\end{proof}

Since being affine can be checked on the truncation, we immediately obtain the following piece of trivia. 

\begin{cor}
The structure morphism $M(Z/X) \to \Ab^1 \times X$ is affine. 
\end{cor}

Let us also record an alternative description for the truncation of $j_{Z/X}$:

\begin{cor}\label{InclusionToDeformationSpaceCor}
The composition
\begin{center}
\begin{tikzcd}
\Ab^1 \times \tau_0(Z) \arrow[hook]{r}{\tau_0(j_{Z/X})} & \tau_0\big(M(Z/X)\big) \arrow[]{r}{\rho^\circ_{Z/X}}  & \Sc pec \Big( \bigoplus_{i \in \Zb} \Sym^i_{\tau_0(X)} \big(\pi_0(\Ic)\big) \cdot t^{-i}\Big) 
\end{tikzcd}
\end{center}
is the morphism induced by the obvious surjection
$$\bigoplus_{i \in \Zb} \Sym^i_{\tau_0(X)} \big(\pi_0(\Ic)\big) \cdot t^{-i} \twoheadrightarrow \bigoplus_{i \leq 0} \Oc_{\tau_0(Z)} \cdot t^{-i}.$$
\end{cor}
\begin{proof}
In the notation of the discussion preceding Theorem \fref{DeformationSpaceAsSpecThm}, we can identify $\Ab^1 \times \tau_0(Z)$ (considered as the truncation of $\bl_{\{0\} \times Z}(\Ab^1 \times Z)$) as the homogeneous spectrum (with respect to $v$-degree) of the bigraded algebra
\begin{equation}\label{BigradedDeformationAlgebra2}
\begin{tikzcd}
\vdots & \vdots & \vdots & \\
0 & 0 & \Oc_{\tau_0(Z)} \cdot v^2t^2 & \cdots\\
0  & \Oc_{\tau_0(Z)} \cdot vt & \Oc_{\tau_0(Z)} \cdot vt^2 & \cdots \\
\Oc_{\tau_0(Z)} & \Oc_{\tau_0(Z)} \cdot t & \Oc_{\tau_0(Z)} \cdot t^2 & \cdots
\end{tikzcd}
\end{equation}
and by Proposition \fref{StrictTransformProp}, the morphism we are interested in is the one induced by the morphism of bigraded algebras from (\fref{BigradedDeformationAlgebra}) to (\fref{BigradedDeformationAlgebra2}) which is the natural surjection $\Oc_{\tau_0(X)} \to \Oc_{\tau_0(Z)}$ on each bidegree where it makes sense. Inverting $vt$ and passing to the $v$-degree 0 homogeneous parts, we obtain the desired result.
\end{proof}

It is also easy to describe how the classical deformation space sits inside the derived one.

\begin{thm}\label{ClassicalToDerivedDeformationSpaceThm}
Let $i: Z \hookrightarrow X$ be a quasi-smooth closed immersion. Then the closed immersion $i_{\{0\} \times Z/ \Ab^1 \times X}$ of Theorem \fref{InclusionOfClassicalBlowUpAsSymmetricAlgThm} restricts to the closed immersion
$$i^\circ_{Z/X}: \Sc pec\Big( \bigoplus_{i \in \Zb} \Ic_0^i \cdot t^{-i}\Big) = M^\mathrm{cl}\big(\tau_0(Z)/ \tau_0(X)\big) \hookrightarrow \Sc pec \Big( \bigoplus_{i \in \Zb} \Sym^i_{\tau_0(X)} \big(\pi_0(\Ic)\big) \cdot t^{-i}\Big)$$
induced by the natural surjection
$$\bigoplus_{i \in \Zb} \Sym^i_{\tau_0(X)} \big(\pi_0(\Ic)\big) \cdot t^{-i} \twoheadrightarrow \bigoplus_{i \in \Zb} \Ic_0^i \cdot t^{-i},$$
where $\Ic_0$ is the sheaf of ideals cutting $\tau_0(Z)$ from $\tau_0(X)$, $\Ic_0^i := \Oc_{\tau_0(X)}$ for $i$ negative and $\pi_0(\Ic) \twoheadrightarrow \Ic_0$ is the natural map $\psi_{Z/X}$ as in Section \fref{ComparisonWithClassicalSubSect}.
\end{thm}
\begin{proof}
The proof is essentially the same as the previous two proofs, but uses Theorem \fref{InclusionOfClassicalBlowUpAsSymmetricAlgThm} instead of Proposition \fref{StrictTransformProp}.
\end{proof}

Given a closed inclusion $Z \hookrightarrow X$ of classical schemes, let us denote by $C_{Z/X}$ the \emph{normal cone} of $Z$ in $X$. Recall that this is the relative spectrum over $Z$ of the symmetric algebra $\Sym^*_Z(\pi_0(\Nc^\vee_{Z/X}))$ 
on the classical conormal sheaf of the inclusion. Note that if $Z \hookrightarrow X$ is a closed embedding of derived schemes, then the canonical surjection
$$\pi_0(\Nc^\vee_{Z/X}) \twoheadrightarrow \pi_0(\Nc^\vee_{\tau_0(Z)/\tau_0(X)})$$
induces a canonical closed immersion
$$\iota_{Z/X}: C_{\tau_0(Z)/\tau_0(X)} \hookrightarrow \tau_0(\Nc^\vee_{Z/X})$$
of schemes over $\tau_0(Z)$. We then have the following result.

\begin{cor}\label{ClassicalToDerivedDeformSpaceSpecialFibreCor}
Let $i: Z \hookrightarrow X$ be a quasi-smooth closed immersion. Under the identification of Theorem \fref{ClassicalToDerivedDeformationSpaceThm}, the (classical) fibre of the natural inclusion
$$i^\circ_{Z/X}: M^\mathrm{cl}\big(\tau_0(Z)/ \tau_0(X)\big) \hookrightarrow \tau_0\big(M(Z/X)\big)$$
over $0$ is identified with $\iota_{Z/X}$.  
\end{cor}
\begin{proof}
Combine Theorem \fref{ClassicalToDerivedDeformationSpaceThm} with Proposition \fref{TruncationOfVirtualExceptionalDivProp} and Corollary \fref{ConormalVsIdealCor}.
\end{proof}

Finally, we are going to prove an analogue of Corollaire 2.18 in \cite{Ver} for derived deformation spaces. Let us first lay out the setup. Suppose $Z \hookrightarrow X$ is a quasi-smooth closed immersion of derived schemes, and let $E$ be a vector bundle on $X$ with zero section $s: X \hookrightarrow E$. We want to show that $\tau_0\big(M(Z/E)\big)$ can be identified as the pullback of $E$ over $\tau_0\big(M(Z/X)\big)$. 

Let us proceed somewhat similarly as in the discussion preceding Theorem \fref{DeformationSpaceAsSpecThm}. The sequence $Z \hookrightarrow X \hookrightarrow E$ induces a diagram 
$$
\begin{tikzcd}
\langle E^\vee \rangle \arrow[]{r}{} \arrow[]{d}{} & \langle E^\vee \rangle \arrow[]{r}{} \arrow[]{d}{} & 0 \arrow[]{d}{} \\
\Ic' \arrow[]{r}{u'} \arrow[]{d}{\psi} & \LSym^*_X(E^\vee) \arrow[]{r}{} \arrow[]{d}{} & \Oc_Z \arrow[]{d}{} \\
\Ic \arrow[]{r}{u} & \Oc_X \arrow[]{r}{} & \Oc_Z
\end{tikzcd}
$$
of quasi-coherent sheaves on $X$ with all rows and columns cofibre sequences, and this induces again a natural equivalence 
$$\Ic' \simeq \langle E^\vee \rangle \oplus \Ic.$$
The $\Sym^*_{\tau_0(X)}(E^\vee)$-module structure on the zeroth homotopy sheaf $\langle E^\vee \rangle \oplus \pi_0(\Ic)$ is induced by the usual structure on $\langle E^\vee \rangle$, and the multiplication by elements of $E^\vee$ sends elements from $\pi_0(\Ic)$ to $\langle E^\vee \rangle$ via the composition
$$E^\vee \otimes_{\Oc_{\tau_0(X)}} \pi_0(\Ic) \xrightarrow{\mathrm{Id} \times u} E^\vee \otimes_{\Oc_{\tau_0(X)}} \Oc_{\tau_0(X)} = E^\vee \hookrightarrow \langle E^\vee \rangle.
$$
It follows that the symmetric algebra on $\pi_0(\Ic')$ over $\Sym^*_{\tau_0(X)}(E^\vee)$, whose relative spectrum represents the truncated deformation space $\tau_0\big(M(Z/E)\big)$ by Theorem \fref{DeformationSpaceAsSpecThm}, can be decomposed as ``Laurent polynomials'' on variables $s$ and $t$, where the sheaf of coefficients for $t^{-n} s^m$ is
$$\Sym^{n-m}_{\tau_0(X)}\big(\pi_0(\Ic)\big) \otimes_{\Oc_{\tau_0(X)}} \Sym^m_{\tau_0(X)}(E^\vee)$$
(compare this to the bigraded algebra (\fref{BigradedDeformationAlgebra}); the multiplication is given by the obvious analogue of (\fref{MultiplicationTable})), where $-n$ keeps track of the symmetric power of $\pi_0(\Ic')$, and where $n$ (but not $m$!) is allowed to take negative values as well.

On the other hand, if we change the grading slightly by replacing $s$ by $s' := st$, then the sheaf of coefficients for $t^{-n}{s'}^m$ is
$$\Sym^n_{\tau_0(X)}\big(\pi_0(\Ic)\big) \otimes_{\Oc_{\tau_0(X)}} \Sym^m_{\tau_0(X)}(E^\vee),$$
and investigating how the multiplication is defined on this bigraded algebra, we can observe it to be isomorphic to
$$\Big( \bigoplus_{i \in \Zb} \Sym^i_{\tau_0(X)} \big(\pi_0(\Ic)\big) \cdot t^{-i} \Big) \otimes_{\Oc_{\tau_0(X)}} \Big( \bigoplus_{j \geq 0} \Sym^j_{\tau_0(X)} (E^\vee) \cdot s'^{j} \Big).$$
This takes us a long way in proving the following result.

\begin{thm}\label{DeformationSpaceAndVectorBundlesThm}
Let $Z \hookrightarrow X$ be a quasi-smooth closed immersion, and let $E$ be a vector bundle on $X$ with zero section $s: X \hookrightarrow E$. Then we can form in a commutative diagram
$$
\begin{tikzcd}
& \Ab^1 \times \tau_0(Z) \arrow[hook]{ld} \arrow[hook]{rd} & \\
\tau_0\big(M(Z/X)\big) \times_{\tau_0(X)} \tau_0(E) \arrow[]{rr}{w} \arrow[]{rd} & & \tau_0\big(M(Z/E)\big) \arrow[]{ld} \\
& \Ab^1 &
\end{tikzcd}
$$
with $w$ an isomorphism of schemes, the upper left diagonal morphism is the composition
$$
\begin{tikzcd}
\Ab^1 \times \tau_0(Z) \arrow[hook]{r}{\tau_0(j_{Z/X})} & \tau_0\big(M(Z/X)\big) \arrow[hook]{r}{s} & \tau_0\big(M(Z/X)\big) \times_{\tau_0(X)} \tau_0(E),
\end{tikzcd}
$$
where $s$ is the zero section, and the upper right diagonal morphism is $\tau_0(j_{Z/E})$.
\end{thm} 
\begin{proof}
We already have the isomorphism $w$ by the above discussion. The lower diagram commutes since $w$ clearly preserves $t$. To prove that the upper triangle commutes, it suffices to observe that the morphism
$$\bigoplus_{n \in \Zb, m \geq 0} \Sym^{n-m}_{\tau_0(X)}\big(\pi_0(\Ic)\big) \otimes_{\Oc_{\tau_0(X)}} \Sym^m_{\tau_0(X)}(E^\vee) \cdot t^{-n} s^{m} \to \bigoplus_{n \leq 0} \Oc_{\tau_0(Z)} \cdot t^{-n}$$
associated to $\tau_0(j_{Z/E})$ sends terms with nontrivial power of $s$ to 0 (easy consequence of Corollary \fref{InclusionToDeformationSpaceCor}), and coincides with the morphism associated to the composition
$$
\begin{tikzcd}
\Ab^1 \times \tau_0(Z) \arrow[hook]{r}{\tau_0(j_{Z/X})} & \tau_0\big(M(Z/X)\big) \arrow[hook]{r}{s} & \tau_0\big(M(Z/X)\big) \times_{\tau_0(X)} \tau_0(E),
\end{tikzcd}
$$
for the terms with $m=0$.
\end{proof}

\section{Virtual pullbacks in algebraic bordism}\label{VirtualPullbackSect}

The purpose of this section is to construct virtual pullbacks for the algebraic bordism groups $\Omega^\LM_*$ of Levine--Morel, essentially following the arguments of Section 4 of Lowrey--Schürg's article \cite{LS}, while filling several gaps in their treatment. We note that Adeel Khan performs a similar construction in his recent preprint \cite{Khan2} in the context of motivic homotopy theory of derived stacks. The arguments of this section should go through for any homology theory in algebraic geometry having nice enough functorial properties without assuming anything about being representable in the motivic homotopy category. We will work over a ground field $k$ of characteristic 0 only because that is the generality in which the bordism groups of Levine--Morel are defined. For simplicity, all derived schemes in this section will be assumed to be quasi-projective over $k$.  
 
Let us now state our goal more precisely. Suppose $f: X \to Y$ is a quasi-smooth morphism of relative virtual dimension $d$. We want to construct \emph{virtual pullback morphisms}
$$f^\vir : \Omega^\LM_*\big(\tau_0(Y)\big) \to \Omega^\LM_{* + d}\big(\tau_0(X)\big),$$
which will coincide with the lci pullback morphisms constructed in Section 6.5 of \cite{LM} whenever $X$ and $Y$ are classical schemes (recall that a morphism of classical schemes is quasi-smooth if and only if it is lci). The main result of this section, besides the construction of $f^\vir$, is the following theorem, whose proof is given in Section \fref{VirtPullbackPropertiesSect}.

\begin{thm}\label{VirtPullbackPropertiesThm}
The virtual pullbacks satisfy the following properties.

\begin{enumerate}
\item The morphisms $f^\vir$ are contravariantly functorial in compositions of quasi-smooth morphisms. In other words, if $f: X \to Y$ and $g: Y \to Z$ are quasi-smooth, then
$$(g \circ f)^\vir = f^\vir \circ g^\vir.$$

\item Let 
\begin{center}
\begin{tikzcd}
X' \arrow[]{r}{f'} \arrow[]{d}{g'} & Y' \arrow[]{d}{g} \\
X \arrow[]{r}{f} & Y
\end{tikzcd}
\end{center}
be a homotopy Cartesian square of derived schemes, and suppose $f: X \to Y$ is an lci morphism between classical schemes. Then $f'^\vir$ coincides with the refined lci pullback morphism $f^!_g$ of \cite{LM} Section 6.6.

\item Let 
\begin{center}
\begin{tikzcd}
X' \arrow[]{r}{f'} \arrow[]{d}{g'} & Y' \arrow[]{d}{g} \\
X \arrow[]{r}{f} & Y
\end{tikzcd}
\end{center}
be a homotopy Cartesian square of derived schemes with $f$ quasi-smooth and $g$ projective. Then 
$$f^\vir \circ \tau_0(g)_* = \tau_0(g')_* \circ f'^\vir.$$

\end{enumerate}
\end{thm}

In our construction of virtual pullbacks, we are going to need some basic properties of the classical algebraic bordism groups $\Omega^\LM_*$ proved in \cite{LM}. For the reader's convenience, we are going to list them below. In order not to overburden the exposition, we will usually not explicitly refer to the following theorem when using the listed properties, so the reader is suggested to familiarize themselves with them.

\begin{thm}
The classical algebraic bordism groups $\Omega^\LM_*$ satisfy the following basic properties.

\begin{enumerate}
\item The lci pullback morphisms $f^!$ are contravariantly functorial in compositions.

\item If we have a Tor-independent Cartesian square of classical $k$-schemes
\begin{center}
\begin{tikzcd}
X' \arrow[hook]{r}{f'} \arrow[]{d}{g'} & Y' \arrow[]{d}{g} \\
X \arrow[hook]{r}{f} & Y
\end{tikzcd}
\end{center}
with $f$ lci and $g$ proper, then $f^! \circ g_* = g'_* \circ f'^!$.

\item Let $p: E \to X$ be a vector bundle over $X$ of rank $r$. Then the pullback morphism
$$p^!: \Omega^\LM_*(X) \to \Omega^\LM_{*+r}(E)$$
is an isomorphism and its inverse $(p^!)^{-1}$ coincides with the lci pullback $s^!$ for a section $s$ of $E$.
\end{enumerate}
\end{thm}
\begin{proof}
The first two claims follow from the fact that $\Omega^\LM_*$ is an \emph{oriented Borel--Moore homology theory} (see Definition 5.1.3. of \cite{LM}), which is the main result of the book \cite{LM}. The third claim  is the third part of Corollary 6.5.5 of \emph{op. cit.} 
\end{proof}

We will continue using notation of the previous section. Given a closed embedding $i: Z \hookrightarrow X$ of derived schemes, the conormal complex $\Nc^\vee_{Z/X}$ is by definition the shifted cotangent complex $\Lb_{Z/X}[-1]$. If $i$ is quasi-smooth, then the conormal complex is a vector bundle.  We will denote by $C_{\tau_0(Z)/\tau_0(X)}$ the classical normal cone, and by $\iota_{Z/X}$ the canonical inclusion $C_{\tau_0(Z)/\tau_0(X)} \hookrightarrow \tau_0(\Nc_{Z/X})$ of schemes over $\tau_0(Z)$. Derived deformation to normal cone is denoted by $M(Z/X)$ and the classical one by $M^\mathrm{cl}\big(\tau_0(Z)/\tau_0(X)\big)$.

\subsection{Construction of virtual pullbacks}

The purpose of this section is to construct virtual pullbacks in algebraic bordism along quasi-smooth morphisms. In the end, the construction reduces to a very special instance: pulling back along a virtual Cartier divisor, which is already defined in \cite{LM} under the slightly different name of ``intersection with a pseudo-divisor''. The rest will then follow using the usual strategy of constructing pullbacks: we can construct pullbacks along an arbitrary derived regular embedding $Z \hookrightarrow X$ using the derived deformation to the normal bundle spaces $M(Z/X)$ studied in Section \fref{DeformationToNormalBundleSubSect}, and finally we can construct pullbacks along arbitrary quasi-smooth morphisms by factoring them as a compositions of a derived regular embeddings and a smooth morphisms.

Given a quasi-smooth closed immersion $i: Z \hookrightarrow X$, we can form the sequence
$$
\begin{tikzcd}
X & \arrow[]{l}[swap]{\mathrm{pr}_2} \Gb_m \times X \arrow[hook]{r}{u} & M(Z/X) & \arrow[hook']{l}[swap]{s_{Z/X}} \Nc_{Z/X} \arrow[]{r}{p} & Z 
\end{tikzcd}
$$
of derived schemes, where $u$ is the canonical open embedding to the derived deformation space, $s_{Z/X}$ is the virtual Cartier divisor defined as the derived fibre over $0$ of the canonical morphism $M(Z/X) \to \Ab^1$ (inclusion of the \emph{special} fibre), and $p$ is the vector bundle projection. Analogously to Section 6.5.2 of \cite{LM}, we make the following definition.

\begin{defn}\label{DerivedSpecializationMorphismDef}
Let everything be as above. We can then define the \emph{derived specialization morphism} $\sigma_{Z/X}: \Omega^\LM_*\big(\tau_0(X)\big) \to \Omega^\LM_*\big(\tau_0(\Nc_{Z/X})\big)$ as the composition
$$\Omega^\LM_*\big(\tau_0(X)\big) \xrightarrow{\mathrm{pr}_2^!} \Omega^\LM_{*+1}\big(\Gb_m \times \tau_0(X)\big) \xrightarrow{(u^!)^{-1}} \Omega^\LM_{*+1}\big(\tau_0(M(Z/X))\big) \xrightarrow{s_{Z/X}^p} \Omega^\LM_*\big(\tau_0(\Nc_{Z/X})\big),$$
where $s^p_{Z/X}$ is the pullback along pseudo-divisor defined in Section 6.5.1 of \cite{LM}. 
\end{defn}
\begin{rem}\label{WellDefinedSpecializationRem}
Note that since $u^!$ is not a bijection (only a surjection), $(u^!)^{-1}$ is not well defined. However, since $s^p_{Z/X} \circ \tau_0(s_{Z/X})_* = 0$, the well definedness of the composition $s^p_{Z/X} \circ (u^!)^{-1}$ follows from the localization exact sequence.
\end{rem}

The following result gives perhaps a more concrete formula for the derived specialization morphism. We begin by noting that given a derived regular embedding $i: Z \hookrightarrow X$, the classical deformation space $M^\mathrm{cl}\big(\tau_0(Z)/\tau_0(X)\big)$ allows us to define a specialization morphism
$$\sigma_{\tau_0(Z)/\tau_0(X)}: \Omega^\LM_*\big(\tau_0(X)\big) \to \Omega^\LM_*\big(C_{\tau_0(Z) / \tau_0(X)}\big)$$
landing in the bordism group of the normal cone. Recalling that the normal cone sits canonically inside $\tau_0\big(\Nc_{Z/X}\big)$, we obtain the following.

\begin{prop}
Let $Z \hookrightarrow X$ be a quasi-smooth closed embedding. Then 
$$\sigma_{Z/X} = \iota_{Z/X*} \circ \sigma_{\tau_0(Z)/\tau_0(X)},$$
where $\iota_{Z/X}$ is the canonical closed embedding $C_{\tau_0(Z) / \tau_0(X)} \hookrightarrow \tau_0(\Nc_{Z/X})$.
\end{prop}
\begin{proof}
By Corollary \fref{ClassicalToDerivedDeformSpaceSpecialFibreCor} and basic properties of (derived) deformation spaces, we have the Cartesian diagram
$$
\begin{tikzcd}
C_{\tau_0(Z) / \tau_0(X)} \arrow[hook]{d}{\iota_{Z/X}} \arrow[hook]{rr}{s_{\tau_0(Z) / \tau_0(X)}} && M^\mathrm{cl}\big(\tau_0(Z) / \tau_0(X)\big) \arrow[hook]{d}{i^\circ_{Z/X}} & \arrow[hook']{l}[swap]{u'} \Gb_m \times \tau_0(X) \arrow[]{d}{ \mathrm{Id}} \\
\tau_0(\Nc_{Z/X}) \arrow[hook]{rr}{\tau_0(s_{Z/X})} && \tau_0\big(M(Z/X)\big) & \arrow[hook']{l}[swap]{u} \Gb_m \times \tau_0(X).
\end{tikzcd}
$$
Since $C_{\tau_0(Z)/\tau_0(X)}$ is the pullback of $\tau_0(\Nc_{Z/X})$ along $i^\circ_{Z/X}$ as pseudo-divisors, it follows from \cite{LM} Lemma 6.2.1 (1) that
\begin{equation}\label{PseudoPushPullEq}
s_{Z/X}^p \circ i^\circ_{Z/X*} = \iota_{Z/X*} \circ s^p_{\tau_0(Z)/\tau_0(X)}.
\end{equation}
On the other hand, as $i^\circ_{Z/X*} \circ (u'^!)^{-1}$ is a one sided inverse ``morphism'' of $u^!$, we can compute that
\begin{align*}
\sigma_{Z/X} &= s^p_{Z/X} \circ (u^!)^{-1} \circ \mathrm{pr}_2^! \\
&= s^p_{Z/X} \circ i^\circ_{Z/X*} \circ (u'^!)^{-1} \circ \mathrm{pr}_2^! \\
&= \iota_{Z/X*} \circ s^p_{\tau_0(Z)/\tau_0(X)} \circ (u'^!)^{-1} \circ \mathrm{pr}_2^! & (\fref{PseudoPushPullEq}) \\
&= \iota_{Z/X*} \circ \sigma_{\tau_0(Z)/\tau_0(X)}
\end{align*}
proving the claim.
\end{proof}

We can now define virtual pullbacks along quasi-smooth immersions.

\begin{defn}\label{VirtualPullbackAlongImmersionDef}
Suppose $i: Z \hookrightarrow X$ is a derived regular embedding of virtual codimension $d$. We define the \emph{virtual pullback} $i^\vir: \Omega^\LM_*\big(\tau_0(X)\big) \to \Omega^\LM_*\big(\tau_0(Z)\big)$ as the composition 
$$\Omega^\LM_*\big(\tau_0(X)\big) \xrightarrow{\sigma_{Z/X}} \Omega^\LM_*\big(\tau_0(\Nc_{Z/X})\big) \xrightarrow{(p^!)^{-1}} \Omega^\LM_{*-d}\big(\tau_0(Z)\big),$$
where $p$ is the natural projection of the vector bundle $\tau_0(\Nc_{Z/X}) \to \tau_0(Z)$.
\end{defn}

Many desirable properties are verified by the following three lemmas.

\begin{lem}[cf. \cite{LS} Lemma 4.13]\label{PullbackAlongImmersionCommutesLem}
Suppose
\begin{center}
\begin{tikzcd}
Z' \arrow[hook]{r}{i'} \arrow[]{d}{f'} & X' \arrow[]{d}{f} \\
Z \arrow[hook]{r}{i} & X
\end{tikzcd}
\end{center}
is homotopy Cartesian with $i$ a derived regular embedding. Then
\begin{enumerate}
\item if $f$ is proper, 
$$i^\vir \circ \tau_0(f)_* = \tau_0(f')_* \circ i'^\vir;$$
\item if $f$ is smooth,
$$\tau_0(f')^! \circ i^\vir = i'^\vir \circ \tau_0(f)^!.$$
\end{enumerate}
\end{lem}
\begin{proof}
Since derived deformation spaces are stable under homotopy pullbacks, the induced diagram
$$
\begin{tikzcd}
X' \arrow[]{d}{f} & \arrow[]{l}[swap]{\mathrm{pr}_2} \Gb_m \times X'  \arrow[]{d}{\mathrm{Id} \times f} \arrow[hook]{r}{u'} & M(Z'/X') \arrow[]{d}{f''} & \arrow[hook']{l}[swap]{s_{Z'/X'}} \Nc_{Z'/X'} \arrow[]{d}{f'''} \arrow[]{r}{p'} & Z' \arrow[]{d}{f'} \\
X & \arrow[]{l}[swap]{\mathrm{pr}_2} \Gb_m \times X \arrow[hook]{r}{u} & M(Z/X) & \arrow[hook']{l}[swap]{s_{Z/X}} \Nc_{Z/X} \arrow[]{r}{p} & Z 
\end{tikzcd}
$$
is homotopy Cartesian. Hence, the pseudo-divisor induced by $s_{Z/X}$ pulls back along $\tau_0(f'')$ to the pseudo divisor induced by $s_{Z'/X'}$, and therefore Lemma 6.2.1 of \cite{LM} tells us that in 
\begin{enumerate}
\item we have the identity 
$$s_{Z/X}^p \circ \tau_0(f''')_* = \tau_0(f'')_* \circ s_{Z'/X'}^p;$$
\item we have the identity
$$\tau_0(f''')^! \circ s_{Z/X}^p = s_{Z'/X'}^p \circ \tau_0(f'')^!.$$
\end{enumerate} 
The commutativity of the three other squares needed to prove the result follow from the basic properties of smooth pullbacks and proper pushforwards in $\Omega^\LM_*$, so we are done.
\end{proof}

\begin{lem}\label{AgreesWithPseudoPullLem}
Suppose $i: D \hookrightarrow X$ is a virtual Cartier divisor. Then
$$i^\vir = i^p: \Omega^\LM_*\big(\tau_0(X)\big) \to \Omega^\LM_{*-1}\big(\tau_0(D)\big).$$ 
\end{lem}
\begin{proof}
Since both types of pullbacks satisfy a similar push-pull formula in homotopy Cartesian squares ($i^\vir$ by Lemma \fref{PullbackAlongImmersionCommutesLem} (1), $i^p$ by \cite{LM} Lemma 6.2.1), it is evidently enough to deal with the following special case (see \emph{op. cit.} Definition 6.5.1): $X$ is smooth and connected and $D$ is either
\begin{enumerate}
\item an snc divisor on $X$; 
\item the derived vanishing locus of the zero section of a line bundle $\Ls$ on $X$.
\end{enumerate}
Case 1. is a special case of \cite{LM} Lemma 6.5.6, so it is enough for us to deal with case 2. But now the classical deformation space is equivalent to $\Ab^1 \times X$ and therefore $\sigma_{X/\tau_0(D)}$ is the identity. Moreover, $\iota_{D/X}$ is the zero section $s: X \hookrightarrow \Ls$, so that $i^\vir = s^! \circ s_*$ is the multiplication by the first Chern class of $\Ls$. But since also $i^p$ is the multiplication by the first Chern class of $\Ls$ in this case, we are done.
\end{proof}

\begin{lem}\label{PullbackAlongImmersionDCartesianLem}
Suppose the square
\begin{center}
\begin{tikzcd}
Z' \arrow[hook]{r}{i'} \arrow[hook]{d}{j'} & X' \arrow[hook]{d}{j} \\
Z \arrow[hook]{r}{i} & X
\end{tikzcd}
\end{center}
is homotopy Cartesian, with all morphisms quasi-smooth closed immersions. Then
$$i'^\vir \circ j^\vir = j'^\vir \circ i^\vir.$$
\end{lem}
\begin{proof}
By assumption, the diagram
$$
\begin{tikzcd}
X' \arrow[hook]{d}{} & \arrow[]{l}[swap]{\mathrm{pr}_2} \Gb_m \times X'  \arrow[hook]{d} \arrow[hook]{r}{u'} & M(Z'/X') \arrow[hook]{d} & \arrow[hook']{l}[swap]{s_{Z'/X'}} \Nc_{Z'/X'} \arrow[hook]{d} \arrow[]{r}{p'} & Z' \arrow[hook]{d} \\
X & \arrow[]{l}[swap]{\mathrm{pr}_2} \Gb_m \times X \arrow[hook]{r}{u} & M(Z/X) & \arrow[hook']{l}[swap]{s_{Z/X}} \Nc_{Z/X} \arrow[]{r}{p} & Z 
\end{tikzcd}
$$
is homotopy Cartesian, and since the virtual pullbacks commute with smooth pullbacks by Lemma \fref{PullbackAlongImmersionCommutesLem}, we are reduced to showing that virtual pullbacks commute with pullbacks along pseudo-divisors (induced from virtual Cartier divisors).

But given a homotopy Cartesian diagram
$$
\begin{tikzcd}
D' \arrow[hook]{r}{i_{D'}} \arrow[hook]{d}{j''} & X' \arrow[hook]{d}{j} \\
D \arrow[hook]{r}{i_D} & X
\end{tikzcd}
$$
with $D$ and $D'$ virtual Cartier divisors, we can form the following homotopy Cartesian diagram
$$
\begin{tikzcd}
D \arrow[hook]{d}{} & \arrow[]{l}[swap]{\mathrm{pr}_2} \Gb_m \times D  \arrow[hook]{d} \arrow[hook]{r}{u'} & M(D'/D) \arrow[hook]{d} & \arrow[hook']{l}[swap]{s_{D'/D}} \Nc_{D'/D} \arrow[hook]{d} \arrow[]{r}{} & D' \arrow[hook]{d} \\
X & \arrow[]{l}[swap]{\mathrm{pr}_2} \Gb_m \times X \arrow[hook]{r}{u} & M(X'/X) & \arrow[hook']{l}[swap]{s_{X'/X}} \Nc_{X'/X} \arrow[]{r} & X', 
\end{tikzcd}
$$
and \cite{LM} Lemma 6.2.1 reduces the problem to showing that pullbacks along pseudo divisors commute with each other. But this follows from \cite{LM} Proposition 6.3.3, so we are done.
\end{proof}

Moreover, the following result gives a criterion for when the virtual pullback can be computed as the pullback along the truncation.

\begin{lem}\label{CoincidesWithLMLem}
Suppose that
\begin{center}
\begin{tikzcd}
\tau_0(Z) \arrow[hook]{r}{\tau_0(i)} \arrow[hook]{d}{} & \tau_0(X) \arrow[hook]{d}{} \\
Z \arrow[hook]{r}{i}& X
\end{tikzcd}
\end{center}
is homotopy Cartesian. Then 
$$i^\vir = \tau_0(i)^!.$$
\end{lem}
\begin{proof}
In this situation the truncation $\tau_0\big(M(Z/X)\big)$ is naturally isomorphic to the classical deformation space $M^\mathrm{cl}\big(\tau_0(Z)/\tau_0(X)\big)$, and the claim follows by comparing the two definitions.
\end{proof}

Finally, we can extend the definition of virtual pullbacks to arbitrary quasi-smooth morphisms using the following lemma.

\begin{lem}[\cite{LS} Lemma 4.15]\label{VirtualPullbackWellDefinedLem}
Suppose $f: X \to Y$ is quasi-smooth, and suppose we have factorizations
$$X \stackrel{i_1}{\hookrightarrow} P_1 \xrightarrow{p_1}  Y$$
and
$$X \stackrel{i_2}{\hookrightarrow} P_2 \xrightarrow{p_2}  Y$$
of $f$ into a composition of a smooth morphism $p$ and a closed immersion $i$ (necessarily a derived regular embedding). Then
$$i_1^\vir \circ \tau_0(p_1)^! = i_2^\vir \circ \tau_0(p_2)^!.$$

\end{lem}
\begin{proof}
We can now form commutative diagrams
\begin{center}
\begin{tikzcd}
  & P_1 \times_Y P_2 \arrow[]{d}{\mathrm{pr}_j} \\
X \arrow[hook]{ru}{i_{12}} \arrow[hook]{r}{i_j} \arrow[]{rd}{f} & P_j \arrow[]{d}{p_{j}} \\
  & Y
\end{tikzcd}
\end{center}
for $j=1,2$, and since smooth pullbacks are functorial in $\Omega^\LM_*$, we are reduced to showing the following: given a commutative triangle
\begin{center}
\begin{tikzcd}
 & Y \arrow[]{d}{p} \\
X \arrow[hook]{ru}{j} \arrow[hook]{r}{i} & Z 
\end{tikzcd}
\end{center}
with $i$ and $j$ derived regular embeddings and $p$ is smooth, then 
$$i^\vir = j^\vir \circ \tau_0(p)^!.$$
We will show this.

We start by forming the homotopy Cartesian square
\begin{center}
\begin{tikzcd}
X' \arrow[hook]{r}{i'} \arrow[]{d}{p'} & Y \arrow[]{d}{p} \\
X \arrow[hook]{r}{i} & Z
\end{tikzcd}
\end{center}
and noting that $p'$ admits a section $s: X \hookrightarrow X'$ such that $i' \circ s$ is naturally equivalent to $j$. Moreover, as the square
\begin{center}
\begin{tikzcd}
\tau_0(X) \arrow[hook]{r}{\tau_0(s)} \arrow[]{d} &  \tau_0(X') \arrow[]{d} \\
X \arrow[hook]{r}{s} & X'
\end{tikzcd}
\end{center}
is homotopy Cartesian (since $s$ is a section of a smooth morphism), we see that $s^\vir = \tau_0(s)^!$. We can then compute that
\begin{align*}
i^\vir &= \tau_0(s)^! \circ \tau_0(p')^! \circ i^\vir & \\
&=  \tau_0(s)^! \circ i'^\vir \circ \tau_0(p)^! & (\text{Lemma \fref{PullbackAlongImmersionCommutesLem}}) \\
&= s^\vir \circ i'^\vir \circ \tau_0(p)^! & (\text{Lemma \fref{CoincidesWithLMLem}}) \\
&= j^\vir \circ \tau_0(p)^! & (\text{Lemma \fref{PullbackAlongImmersionIsFunctorialLem}})
\end{align*}
proving the claim.
\end{proof}

We are finally ready to make the main definition of the section.

\begin{defn}
Let $f: X \to Y$ be a quasi-smooth morphism of relative dimension $d$, and let 
$$X \stackrel{i}{\hookrightarrow} P \stackrel{p}{\to} Y$$
be a factorization of $f$ where $i$ is a closed embedding and $p$ is a smooth morphism. By Lemma \fref{VirtualPullbackWellDefinedLem}, the morphism
$$f^\vir := i^\vir \circ \tau_0(p)^!: \Omega^\LM_*\big(\tau_0(Y)\big) \to \Omega^\LM_{* + d}\big(\tau_0(X)\big)$$
depends only on $f$, and we define this to be the \emph{virtual pullback along $f$}.
\end{defn}

\subsection{Proof of Theorem \fref{VirtPullbackPropertiesThm}}\label{VirtPullbackPropertiesSect}

The purpose of this section is to prove that the virtual pullbacks satisfy various desirable properties. However, before giving the proof, we need the following result, which was already used in the proof of Lemma \fref{VirtualPullbackWellDefinedLem}.

\begin{lem}[\cite{LS} Proposition 4.14]\label{PullbackAlongImmersionIsFunctorialLem}
Suppose we have quasi-smooth immersions
$$Z \stackrel i \hookrightarrow X \stackrel j \hookrightarrow Y.$$
Then $i^\vir \circ j^\vir = (j \circ i)^\vir$.
\end{lem}
\begin{proof}
We prove the result in two parts.

\begin{enumerate}
\item Let us first deal with the case where $j: X \hookrightarrow Y$ is the zero section of a vector bundle $s: X \hookrightarrow E$. By Theorem \fref{DeformationSpaceAndVectorBundlesThm}, we can form a commutative diagram
$$
\begin{tikzcd}
& \Ab^1 \times \tau_0(Z) \arrow[hook]{ld} \arrow[hook]{rd} & \\
\tau_0\big(M(Z/X)\big) \times_{\tau_0(X)} \tau_0(E) \arrow[]{rr}{w} \arrow[]{rd} & & \tau_0\big(M(Z/E)\big) \arrow[]{ld} \\
& \Ab^1 &
\end{tikzcd}
$$
with $w$ an isomorphism of schemes. Hence the morphisms $\sigma_{X / E}$ and $\sigma_{E \vert_X / E}$ coincide, and therefore we can use Lemma \fref{PullbackAlongImmersionCommutesLem} to conclude that the left small square in
$$
\begin{tikzcd}
\Omega^\LM_*\big(\tau_0(E)\big) \arrow[]{r}{\sigma_{Z/E}} \arrow[]{d}{\tau_0(s)^!} & \Omega^\LM_*\big(\tau_0(\Nc_{Z/E})\big)\arrow[]{d}{s'^!} \arrow[]{r}{\cong} & \Omega^\LM_*(\tau_0(Z)) \arrow[]{d}{\Id} \\
\Omega^\LM_*\big(\tau_0(X)\big) \arrow[]{r}{\sigma_{Z/X}} & \Omega^\LM_*\big(\tau_0(\Nc_{Z/X})\big) \arrow[]{r}{\cong} & \Omega^\LM_*(\tau_0(Z))
\end{tikzcd}
$$
commutes, where $s'$ is the ``zero-section'' $\tau_0(\Nc_{Z/X}) \hookrightarrow \tau_0(\Nc_{Z/E})$ induced by the isomorphism $w$. As $s^\vir = \tau_0(s)^!$ by Lemma \fref{CoincidesWithLMLem}, we can combine the above with basic properties of $\Omega^\LM_*$ to conclude that
$$i^\vir \circ s^\vir = (s \circ i)^\vir,$$
which is exactly what we wanted.

\item Now for the general case. Consider the quasi-smooth immersion $i'$ defined as the composition
$$
\begin{tikzcd}
\Ab^1 \times Z \arrow[hook]{r}{\mathrm{Id} \times i} & \Ab^1 \times X \arrow[hook]{r}{j_{X/Y}} & M(X/Y),   
\end{tikzcd}
$$
and form the homotopy Cartesian diagram 
$$
\begin{tikzcd}
Z \arrow[hook]{d}{i'_0} \arrow[hook]{r}{s \circ i} & \Nc_{X/Y} \arrow[hook]{d}{i_0} \\
\Ab^1 \times Z \arrow[hook]{r}{i'} & M(X/Y) \\
Z \arrow[hook]{r}{j \circ i} \arrow[hook']{u}[swap]{i'_1} & Y \arrow[hook']{u}[swap]{i_1}
\end{tikzcd}
$$
by restricting to fibres over $0$ and over $1$ ($s$ is the zero section $X \hookrightarrow \Nc_{X/Y}$).

Let then $\alpha$ be an element in $\Omega^\LM_*\big(\tau_0(Y)\big)$, and choose $\wtil\alpha$ in $\Omega^\LM_*\big(\tau_0(M(X/Y))\big)$ that restricts to the pullback of $\alpha$ on $\Omega^\LM_*\big(\Gb_m \times \tau_0(Y)\big)$. We can then compute that
\begin{align*}
(j \circ i)^\vir (\alpha) &= (j \circ i)^\vir \circ i_1^\vir (\tilde\alpha)\\ 
&=i'^\vir_1 \circ i'^\vir(\wtil \alpha) & (\text{Lemma \fref{PullbackAlongImmersionDCartesianLem}}) \\
&= i'^\vir_0 \circ i'^\vir(\wtil \alpha) & (\text{homotopy invariance}) \\
&= (s \circ i)^\vir \circ \sigma_{X/Y}(\alpha) & (\text{Lemma \fref{PullbackAlongImmersionDCartesianLem}}) \\
&= i^\vir \circ s^\vir \circ \sigma_{X/Y}(\alpha) & (\text{part 1.}) \\
&= i^\vir \circ j^\vir (\alpha)
\end{align*}
proving the claim. \qedhere
\end{enumerate}
\end{proof}

We are now ready to prove the main result.

\begin{proof}[Proof of Theorem \fref{VirtPullbackPropertiesThm}]
Throughout the proof, we will usually denote the truncations $\tau_0(Y)$ and $\tau_0(f)$ of a derived schemes $Y$ and morphisms $f$ between them by $Y_0$ and $f_0$ respectively.

\begin{enumerate}
\item Since both $f$ and $g$ are quasi-projective, we can factor them as
$$X \stackrel i \hookrightarrow U \times Y \stackrel{p}{\to} Y$$
and
$$Y \stackrel j \hookrightarrow V \times Z \stackrel{q}{\to} Z$$
respectively, where $U$ and $V$ are open subschemes of $\Pb^n$ and $\Pb^m$ respectively, and $p$ and $q$ are the obvious projection morphisms. Form the diagram
\begin{center}
\begin{tikzcd}
X \arrow[hook]{r}{i} & U \times Y \arrow[hook]{r}{\mathrm{Id}_{U} \times j} \arrow[]{d}{p} & U \times V \times Z \arrow[]{d}{p'} \\
& Y \arrow[hook]{r}{j} & V \times Z \arrow[]{d}{q} \\
& & Z,
\end{tikzcd}
\end{center}
where $p'$ is the obvious projection, and note that the middle square is homotopy Cartesian. We can then simply compute that
\begin{align*}
f^\vir \circ g^\vir &= i^\vir \circ p_0^! \circ j^\vir \circ q_0^! & \\
&= i ^\vir \circ (\mathrm{Id}_{U} \times j)^\vir  \circ p_0'^! \circ q_0^! & (\text{Lemma \fref{PullbackAlongImmersionCommutesLem}}) \\
&=  \big((\mathrm{Id}_{U} \times j) \circ i\big)^\vir  \circ (q \circ p')_0^! & (\text{Lemma \fref{PullbackAlongImmersionIsFunctorialLem}}) \\
&= (g \circ f)^\vir. &
\end{align*}
 
\item Recall that the normal bundle is stable under derived pullbacks. The proof is now immediate by comparing the definition of the virtual pullback to that of the refined pullback in \cite{LM} Section 6.6. 
 
\item Let us factor $f$ as
$$X \stackrel i \hookrightarrow P \stackrel{p}{\to} Y$$
where $p$ is smooth. We can now form the homotopy Cartesian diagram
\begin{center}
\begin{tikzcd}
X' \arrow[hook]{r}{i'} \arrow[]{d}{g'} & P' \arrow[]{d}{g''} \arrow[]{r}{p'} & Y' \arrow[]{d}{g} \\
X \arrow[hook]{r}{i} & P \arrow[]{r}{p} & Y 
\end{tikzcd}
\end{center}
and compute that
\begin{align*}
f^\vir \circ g_{0*} &= i^\vir \circ p_0^! \circ g_{0*} & \\
&= i^\vir \circ g''_{0*} \circ p'^!_0  \\
&= g'_{0*} \circ i'^\vir \circ p'^!_0 & (\text{Lemma \fref{PullbackAlongImmersionCommutesLem}}) \\
&= g'_{0*} \circ f'^\vir, &  
\end{align*}
proving the claim. \qedhere
\end{enumerate}
\end{proof}

\Addresses
\end{document}